%% file: report.tex
\pdfoutput=1
\documentclass[11pt,twoside]{article}
\usepackage{amsmath,amsthm,amssymb,amsbsy,fancyhdr,graphicx,
            calc,ifthen,float,psfrag,mathrsfs,hhline}
\input epsf
\PassOptionsToPackage{ctagsplt,Righttag}{amstex}
\usepackage[colorlinks=true,citecolor=blue,linkcolor=blue,bookmarksnumbered=true,pdfstartview={FitH},hyperindex=true,linktoc=all,linktocpage=true]{hyperref}
\usepackage[sort&compress, comma, numbers]{natbib}
\usepackage{mathtools}
\usepackage{framed}
\usepackage{dsfont}
\usepackage{ulem}
\usepackage{color}
\definecolor{darkgreen}{rgb}{0,0.65,0}
\definecolor{darkred}{rgb}{0.65,0,0}
\newcommand{\norm}[1]{\left\lVert {#1} \right\rVert}
\newcommand{\RNum}[1]{\uppercase\expandafter{\romannumeral #1\relax}}
\newcommand{\rnum}[1]{\expandafter{\romannumeral #1\relax}}
\newcounter{hours}
\newcounter{minutes}
\newcommand{\printtime}{\setcounter{hours}{\time/60}%
                        \setcounter{minutes}{\time-\value{hours}*60}%
\ifthenelse{\value{hours}<10}{0}{}\thehours:%
\ifthenelse{\value{minutes}<10}{0}{}\theminutes}
\def\lsp{\def\baselinestretch{0.75}\large\normalsize}
\def\ssp{\def\baselinestretch{1.0}\large\normalsize}
\def\dsp{\def\baselinestretch{1.37}\large\normalsize}
\advance\textheight by \topskip
\input{preamble.tex}

\newtheorem{theorem}{Theorem}[section]
\newtheorem{lemma}[theorem]{Lemma}
\newtheorem{proposition}[theorem]{Proposition}

\newtheorem{remark}[theorem]{Remark}
\newtheorem{assumption}[theorem]{Assumption}
\theoremstyle{definition}
\newtheorem{example}{Example}

\begin{document}

\flushbottom
\pagestyle{empty}
\pagenumbering{arabic}
%
\ssp 

\title{\vspace{-0.75in}\bf 
A simplified nonsmooth nonconvex bundle method with applications to security-constrained ACOPF problems} 
\vspace{0.1in}
\author{\hspace{-0.25in} Jingyi WANG\footnote{wang125@llnl.gov} and Cosmin G. PETRA 
\\[0.05in]
\sl\small \textit{Center for Applied Scientific Computing} \\
\textit{Lawrence Livermore National Laboratory}\\
Livermore, California, USA 
\rm \normalsize}
\date{\today}
\maketitle
\small
\lsp
\par\noindent
\protect\vspace{-0.3in}
\renewcommand{\contentsname}{\normalsize\centerline{Table of contents}}
\tableofcontents
\ssp
\protect\vspace{0.1in}
\normalsize
\begin{abstract}
An optimization algorithm for a group of nonsmooth nonconvex problems inspired by two-stage 
stochastic programming problems is proposed. The main challenges for these problems include 
(1) the problems lack the popular lower-type properties such as prox-regularity assumed in many nonsmooth nonconvex optimization algorithms,
(2) the objective can not be analytically expressed and (3) the evaluation of function 
values and subgradients are computationally expensive. To address these challenges, this report  
first examines the properties that exist in many two-stage problems, specifically upper-$C^2$ objectives. 
Then, we show that quadratic penalty method for security-constrained alternating current optimal power flow (SCACOPF) contingency problems 
can make the contingency solution functions upper-$C^2$. 
Based on these observations, a simplified bundle algorithm that bears similarity to sequential quadratic programming (SQP) method 
is proposed.  
It is more efficient in implementation and computation compared to conventional bundle methods.
Global convergence analysis of the algorithm is presented under novel and reasonable assumptions. 
The proposed algorithm therefore fills the gap of theoretical convergence for smoothed SCACOPF problems.
The inconsistency that might arise in our treatment of the constraints are addressed through a penalty 
algorithm whose convergence analysis is also provided. Finally, theoretical capabilities and numerical performance of the algorithm 
are demonstrated through numerical examples.
\end{abstract}
\par\noindent
\protect\vspace{-0.1in}
\nopagebreak[5]
\begin{description}
\item[Keywords:] Optimization; Nonsmooth; Nonconvex; Upper regularity; 
\end{description}
\dsp
\newpage
\pagestyle{fancyplain}
\large
%
%
\normalsize
\input{Sections/Introduction.tex}

\input{Sections/Twostage.tex}

\input{Sections/Algorithm.tex}

\input{Sections/Examples.tex}

\input{Sections/Conclusion.tex}

\section*{Acknowledgments}
Prepared by LLNL under Contract DE-AC52-07NA27344.
This document was prepared as an account of work sponsored by an agency of the United States government. Neither the United States government nor Lawrence Livermore National Security, LLC, nor any of their employees makes any warranty, expressed or implied, or assumes any legal liability or responsibility for the accuracy, completeness, or usefulness of any information, apparatus, product, or process disclosed, or represents that its use would not infringe privately owned rights. Reference herein to any specific commercial product, process, or service by trade name, trademark, manufacturer, or otherwise does not necessarily constitute or imply its endorsement, recommendation, or favoring by the United States government or Lawrence Livermore National Security, LLC. The views and opinions of authors expressed herein do not necessarily state or reflect those of the United States government or Lawrence Livermore National Security, LLC, and shall not be used for advertising or product endorsement purposes.

\clearpage
\appendix
\input{Sections/Appendix.tex}

\clearpage
\bibliographystyle{plain}
\bibliography{bibliography}
\end{document}

%% file: preamble.tex
\usepackage{amsmath}
\usepackage{algpseudocode}
\usepackage{amssymb}
\usepackage{bm}
\usepackage{mathtools}
\usepackage[font=scriptsize]{caption}
\usepackage[font=scriptsize]{subcaption}
\usepackage[ruled,vlined,linesnumbered,noend]{algorithm2e}

\newcommand{\Rbb}{\ensuremath{\mathbb{R} }}



%% file: Sections/Introduction.tex

\normalsize
\section{Introduction}\label{se:intro}

In this report, we consider the class of nonsmooth nonconvex constrained optimization problems in the form of 

\begin{equation} \label{eqn:opt0}
 \centering
  \begin{aligned}
   &\underset{\substack{x}}{\text{minimize}} 
	  & & f(x)+ \mathcal{R}(x)\\
   &\text{subject to}
	  & & c(x) =c_E\\
	  &&& d^l \leq d(x) \leq d^u\\
	  &&& x^l \leq x \leq x^u,\\
  \end{aligned}
\end{equation}
where the functions $f(\cdot):\Rbb^n\to\Rbb$, $c(\cdot):\Rbb^n\to\Rbb^{m_c}$, $d(\cdot):\Rbb^n\to\Rbb^{m_d}$ are continuously differentiable. 
The entries of the bound vectors $d^l$ and $d^u$ are in 
$\Rbb$. The bounds
on the optimization variables $x$ are such that $x^l,x^u\in \Rbb^n$, 
$x_j^l<x_j^u$, for all $j\in\{1,...,n\}$. The function $\mathcal{R}(\cdot)$ is nonsmooth and nonconvex, as 
in a large number of important applications.
In addition, the analytical form of $\mathcal{R}(\cdot)$ might not be available, 
 forcing a potential algorithm to rely on known points in the optimization space. 
Prominent problems in the form of~\eqref{eqn:opt0} include two-stage stochastic programming problems with recourse~\cite{Shapiro_book,Birge97Book,KallWallace}. 
While general to apply to various paradigms of two-stage optimization under uncertainty (or other nonsmooth problems), the  methodology presented in this report is driven by the problem of optimal operation of large-scale electrical transmission power grids. 

\subsection{Power grid optimization}\label{se:intro-scacopf}
Electricity generation and distribution in nationwide power grid systems rely upon optimization models and tools to find the  power generation injection levels and transmission power flows at each of the grid nodes so that the demand at given substations is met at the lowest generation cost and minimum transmission losses~\cite{irc2017}. 
Among them, alternating current optimal power flow (ACOPF) models have been proposed, researched, and adopted in some cases in operations because they model the power grid more accurately (\textit{e.g.}, capture reactive power and include transmission losses) than the economic dispatch models. 
ACOPF models are becoming increasingly challenged with the penetration of (highly intermittent) renewable sources of energy (\textit{e.g.}, wind and solar) and ongoing shifts in demand, which are caused by the emergence of commodity solar systems, battery storage, and electric vehicles~\cite{Capitanescu2016,Molzahn2019}. To better accommodate these emerging technologies, power grid operators need to operate increasingly complex power grid systems under highly stochastic demand and generation profiles and frequent equipment failures.

SCACOPF is one of the salient emerging optimization paradigms for increasing the reliability of the power grid and ensuring its operation~\cite{NERCStandards} under various types of failures. SCACOPF  extends the capabilities of ACOPF by requiring that the state of the grid is secure with respect to a comprehensive list of equipment \textit{contingencies} (\textit{e.g.}, failures of generators, transmission lines, and transformers) and sometimes under stochastic demand and/or generation~\cite{NERCStandards, Ekisheva2015}. As a result, the SCACOPF mathematical optimization problem reaches extreme scale as it needs to simulate multiple ACOPF models (routinely $O(10^5)$) in order to find a secure state of the grid. An equally important challenge is given by the highly nonlinear and nonconvex nature of SCACOPF (as well as of ACOPF) problems, which makes it difficult to find global (or at least good quality) optima of the problem.  On the other hand, SCACOPF models  need to be solved under strict time limitations, \textit{i.e.}, in real-time, to allow ample adjustment time for the equipment (generation ramp up or down, load shedding, transmission switching, etc.). These challenges have sparked research over the last decades to study new scalable optimization algorithms and develop parallel computer implementations for SCACOPF problems. 

Parallel computing has recently shown promising results for reaching real-time solutions for SCACOPF. In~\cite{ChiangPetraZavala_14_PIPSNLP,Qiu2005,petra_14_augIncomplete,petra_14_realtime}, the SCACOPF problem is solved in parallel by decomposing the
linear algebra of interior-point methods~\cite{Nocedal_book} using a Schur complement technique. Alternative parallel computing
approaches, such as the optimization-based decompositions from~\cite{Liu2013} and~\cite{petra_21_gollnlp}, break down  the SCACOPF problem at the level of the formulation into base case ACOPF and contingency response ACOPF subproblems and enforce the reconciliation between subproblems' coupling variables using first-order gradient-based methods~\cite{Liu2013} or carefully chosen approximations for the coupling terms~\cite{petra_21_gollnlp}. 
Decomposed problem formulation however often generates a nonsmooth nonconvex $\mathcal{R}(\cdot)$, posing challenge to 
the design of an algorithm that converges theoretically and is computationally efficient. 

While effective in practice, decomposed SCACOPF algorithms such as the smoothed two-stage solver in~\cite{petra_21_gollnlp} has not been fully analyzed in theory. The existing nonsmooth nonconvex optimization literature does not apply directly to the problem. In this report, we aim to contribute to the theoretical analysis of the decomposition algorithms for SCACOPF problems, 
and more broadly two-stage optimization problems. In addition, we provide algorithm design and implementation details that could be valuable for large-scale nonsmooth nonconvex optimization problems.
\subsection{Nonsmooth optimization}\label{se:intro-nonsmooth}
Nonsmooth optimization has been researched extensively for decades. 
Most prominent methods include subgradient methods and bundle methods. 
Subgradient method takes steps in the direction of a subgradient at a given point, 
relying heavily on a robust step size control algorithm to achieve good rates of convergence~\cite{shor1985}.
Bundle methods are widely regarded as
one of the most efficient optimization methods to address discontinuous first-order derivatives~\cite{makela1992,Kiwiel1996}. 
The bundle method develops an approximation model for the objective with the information from previous iterations, 
referred to as a bundle, and solves optimization subproblems with the model~\cite{mifflin1982,kiwiel1985}. 
The solution to the subproblem is regarded as a trial step, which through a rejection criterion is either taken as 
a serious step or rejected but included in the bundle to improve the trial step for the next iteration.
In the case with convex objectives, 
the linearization error between the objective function and the tangent planes that comprise the approximation model is positive, a property that is not 
valid for nonconvex functions.
Therefore, adjustment to the approximation model is needed.
A commonly used one, called the down-shift mechanism, is introduced in~\cite{mifflin1982} and used in~\cite{lemarechal1996,lemarechal1978,schramm1992,zowe1989}. Convergence analysis using this mechanism can be found in~\cite{apkarian2008,noll2009}. 
Given additional local convexity properties, \textit{e.g.}, lower-$C^2$, the slope of the tangent planes can be titled as well to generate positive linearization errors~\cite{hare2010}. These redistributed bundle methods are shown to work in practice under less ideal conditions~\cite{hare2015}.
Constrained nonsmooth nonconvex optimization adds another layer of complexity on top for bundle methods. Convex constraints can typically be maintained as they are in the subproblems and convergence analysis would stand valid~\cite{hare2015}. In particular, affine constraints, commonly appearing in applications do not pose extra challenge~\cite{dao2016} in convergence analysis.

In dealing with nonsmooth nonconvex objectives with general constraints, ideas from penalty and filter methods are often applied to incorporate the constraints into the objective in bundle methods. The global convergence studies in this case typically shows the algorithm can either converge to a KKT point of the original problem or to a stationary/critical point of the constraint violation. The latter case could lead to an infeasible solution.
An exact penalty merit function that measures the progress of both the objective and inequality constraint violation is used in the redistributed bundle method in~\cite{yang2014}.
Lower-$C^2$ and a special strong Slater condition are assumed to ensure global convergence.
In~\cite{dao2015}, a progress function that is the maximum of a penalized objective reduction and constraint violation is chosen. The bundle method is applied to the subproblem whose objective is replaced by the progress function, eliminating the general constraint. Given lower-$C^1$ or upper-$C^1$ property, convergence is proved. Similar algorithm with direct assumptions on the penalty and quadratic parameters are presented in~\cite{monjezi2019}. Others have chosen a different set of penalized objective functions~\cite{lv2018}.

Outside bundle methods,~\cite{curtis2012} proposed a sequential quadratic programming (SQP) method using gradient sampling for nonconvex nonsmooth inequality constrained optimization. As conventional in SQP, the constraint is relaxed through linearization while the iterations are taken at differentiable points of the Lipschitz objectives and constraints. The global convergence result of the algorithm shows that accumulation points are the stationary points of the exact penalty function, which can be equal to the constraint violation function depending on the penalty parameter. 
A more efficient BFGS-SQP is proposed in~\cite{curtis2017} which shows promising numerical behaviors without theoretical guarantee of convergence.
In \cite{xu2015}, a smoothing function of the objective that satisfies gradient consistency property is used together with augmented Lagrangian constraint relaxation. Extensive discussion of constraint qualifications are presented in order to achieve convergence. Line search of the Lagrangian function is possible due to the smoothing function which converges in the limit to the nonsmooth objective.  
Alternating direction method of
multipliers (ADMM) has also been applied to nonsmooth nonconvex problems and in particular interest to us, distributed and asynchronous parallel algorithm~\cite{hong2018}. In~\cite{wang2019} the convergence analysis requires the objective to be locally prox-regular with affine constraints, similar to those in bundle method literature. One of the difficulties in applying ADMM to our application is the non-existence of the analytical form of part of the objective.
Recently, difference-of-convex functions have been systematically studied in~\cite{cui2021}. In particular interest to this paper, the recourse function of linearly bi-parameterized two-stage problems with quadratic recourse is shown to have convexity-concavity property~\cite{liu2020}. The authors then proposed an iterative algorithm with a quadratic convex subproblem that converges subsequently to generalized critical points. The presence of a nonsmooth concave function from recourse function in the objective is novel and closely related to upper-$C^2$ property.

The report is organized as follows. In Section~\ref{sec:prob}, we describe general two-stage stochastic programming problems
with an emphasis on the SCACOPF problem. In particular, we discuss the upper-type 
properties of $\mathcal{R}(\cdot)$ that arise from such problems, which serve as the guild lines in designing the algorithm.
We also propose quadratic penalty smoothing 
to enable a large group of problems to possess some upper-type property that they do not have otherwise. 
In Section~\ref{sec:alg}, our simplified bundle algorithm is proposed and its global convergence analysis  
is provided given novel assumptions drawn from Section~\ref{sec:prob}.
The algorithm can be seen as an extension of SQP to nonsmooth nonconvex problems.
We also discuss the adjustable update rules
for the approximations of second-stage optimal value functions and its application to two-stage stochastic 
programming problems.
An algorithm to address possible inconsistency due to our treatment of the constraints in the subproblems is presented in Section~\ref{sec:lincons}, whose global convergence analysis is provided as well.
Numerical experiments are shown in Section~\ref{sec:exp} that illustrate both the 
theoretical and numerical capabilities of the proposed algorithm.
We note that the proposed algorithms can be parallelized efficiently for two-stage stochastic programming problems 
as shown in~\cite{wang2021}, which greatly improves computational  
efficiency since evaluation of $\mathcal{R}(\cdot)$ and its general subgradients can be the 
computationally expensive. This prepares the algorithm well for large-scale SCACOPF applications.

%% file: Sections/Twostage.tex
\section{Problem description, preliminaries and regularization}\label{sec:prob}
Two-stage stochastic optimization problems with recourse fits within the formulation of~\eqref{eqn:opt0}. 
Using expectation as an example, the nonsmooth nonconvex function $\mathcal{R}(\cdot)$, also referred to as the expected \textit{recourse} function~\cite{Shapiro_book}, can be expressed as 
 $\mathcal{R}(x)=\mathbb{E}_\Omega[r(x, \omega)]$, where $\mathbb{E}$ is the expectation operator.
The function $r(x,\omega)$ is the optimal value function of the second-stage problem parameterized by $x$ and the random variable $\omega$ over a probability space $\Omega$. More specifically, $r(x,\omega)$ has the following mathematical form:
 \begin{equation} \label{eqn:2ndstage}
 \centering
  \begin{aligned}
      r(x,\omega) = &\underset{y}{\text{minimize} }
	  & & p(y,x, \omega)\\
   &\text{subject to}
	  & & c(y,x, \omega) =c_E(\omega)\\
	  &&& d^l(\omega) \leq d(y,x, \omega) \leq d^u(\omega)\\
   &&& y^l(\omega) \leq y \leq y^u(\omega).\\
 \end{aligned}
\end{equation}
In~\eqref{eqn:2ndstage}, the functions $p(\cdot,\cdot, \omega):\Rbb^p\times\Rbb^n\times\Omega\to\Rbb$, $c(\cdot,\cdot,\omega):\Rbb^p\times\Rbb^n\times\Omega\to\Rbb^{m_c}$, $d(\cdot,\cdot,\omega):\Rbb^p\times\Rbb^n\times\Omega\to\Rbb^{m_d}$ are 
assumed to be smooth. 
The entries of the bound vector $d^l(\omega)$ and $d^u(\omega)$ are in 
$\Rbb$ and the bounds
on the optimization variables $y$ is such that 
$y^l(\omega)\in\Rbb^p$, $y^u(\omega)\in\Rbb^p$ and
$y^l_{j}(\omega)<y_{j}^u(\omega)$, for all $j\in\{1,\dots,p\}$ and $\omega\in\Omega$.

SCACOPF models can be established in the form of~\eqref{eqn:opt0}--\eqref{eqn:2ndstage}, 
where a secure state of the power grid is found that minimizes operation cost $f(\cdot)$ plus the average monetary penalties $p(\cdot,\cdot, \cdot)$ associated with not satisfying power demand and violating grid's power flows over all contingencies.
Assuming uniform distribution, 
the sample space of $\omega$ consists of the set of all possible $K$ contingencies, each taken with equal probability $\frac{1}{K}$. The first-stage optimization variables $x$ in ~\eqref{eqn:opt0} correspond to power generation and power flow levels that are to be implemented instantly in practice; while the second stage variables $y$ are recourse actions to be implemented should a contingency $\omega$ occur.  
Thus, problem~\eqref{eqn:opt0} simplified with discrete uniform probability distribution can be written as   
\begin{equation} \label{eqn:opt-ms}
 \centering
  \begin{aligned}
   &\underset{\substack{x}}{\text{minimize}} 
	  & & f(x)+ \frac{1}{K}\sum_{i=1}^K r_i(x)\\
   &\text{subject to}
	  & & c(x) =c_E\\
	  &&& d^l \leq d(x) \leq d^u\\
	  &&& x^l \leq x \leq x^u,\\
  \end{aligned}
\end{equation}
where recourse functions $r_i(\cdot):\Rbb^n\to \Rbb$, for all $i \in {1,\dots,K}$, are
the optimal solution functions to the deterministic second-stage optimization subproblems, namely, 
\begin{equation} \label{eqn:opt-rc}
 \centering
  \begin{aligned}
   r_i(x) =& \underset{\substack{y_i }}{\text{minimize}} 
	  & & p_i(x,y_i)\\
   &\text{subject to}
	  & & c(x,y_i) =c_E\\
	  &&& d^l_i\leq d(x,y_i) \leq d^u_i \\
   &&& y^l_i \leq y_i \leq y^u_i.\\
  \end{aligned}
\end{equation}
When $K$ is relatively small, the problem can be solved as a single-stage problem through off-the-shelf optimization packages. However, it is usually difficult to satisfy the requirement of real-time solution time.  If the number of contingencies $K$ is exceedingly large, which is common in real-world power grid operations, then solution through serial optimization solvers is intractable. 
One approach to tackle such large-scale problems is to use memory-distributed algorithms  
with the help of parallel computing~\cite{capitanescu2011}.
The evaluation of  $r_i(\cdot)$ at a given $x$ is of considerable computational cost 
and can reach $O(10^2)$ seconds for real-world power grids. 
This characteristic requires the design of the optimization algorithm to avoid evaluating $r_i(\cdot)$ as much as possible.
\subsection{Preliminaries and notations}\label{sec:pri}
As mentioned earlier, throughout this work, we assume functions $f(\cdot),c(\cdot),d(\cdot)$ in~\eqref{eqn:opt0} are smooth, while 
functions $r_i(\cdot)$ are proper (A1,~\cite{rockafellar1998})
and locally Lipschitz. To simplify notation, we use $r(\cdot)$ to replace $r_{i}(\cdot)$ and let $K=1$.
A closed ball in $\Rbb^n$ centered at $\bar{x}\in \Rbb^n$ 
with radius $\rho>0$ is denoted as $B_{\rho}(\bar{x})$.
In nonsmooth nonconvex optimization literature, both Clarke subgradient~\cite{clarke1983} and lower regular subgradient (8.3,~\cite{rockafellar1998}) have been widely adopted in analysis. 
While they both enjoy the outer/upper-semicontinuity property necessary in establishing convergence (6.6,~\cite{rockafellar1998} ), 
Clarke subdifferential, denoted by $\bar{\partial} r{(\bar{x})}$, of function $r(\cdot)$ at $\bar{x}$
is used in this  work.
In addition, the less common upper regular/general subgradient offers a critical view in discussing the properties of interest.  

The lower regular subdifferential of $r(\cdot)$ at point $\bar{x}$, 
denoted as $\hat{\partial} r(\bar{x})$, is defined by  
\begin{equation}\label{eqn:subgradient-def}
\centering
 \begin{aligned}
	\hat{\partial} r (\bar{x}) = \left\{ g\in\Rbb^n |\liminf_{ \substack{x\to \bar{x}\\x\neq \bar{x} }}\frac{r(x)-r(\bar{x})-\langle g,x-\bar{x}\rangle  }{\norm{x-\bar{x}}}\geq 0\right\},
 \end{aligned}
\end{equation}
where the 2-norm $\norm{\cdot}$ is used and $\langle \cdot \rangle$ is the inner product in $\Rbb^n$.
The notion of \textit{f-attentive} convergence, which is crucial for the concept of general subgradient, is defined as 
\begin{equation}\label{def:f-attentive}
\centering
 \begin{aligned}
	 x^{\nu} \xrightarrow[r]{} \bar{x} \quad \ratio\Leftrightarrow \quad x^{\nu}\to\bar{x} \quad \text{with} \quad r(x^{\nu})\to r(\bar{x}),
 \end{aligned}
\end{equation}
where $\{x^{\nu}\}$ is a sequence of points. Given the assumption of Lipschitz $r(\cdot)$, this becomes trivial.
If there exists a sequence $\{x^{\nu}\}$ such that $x^{\nu} \xrightarrow[r]{} \bar{x}$ and $g^{\nu}\in \hat{\partial}r(x^{\nu})$ with $g^{\nu}\to g$,
$g$ is called a lower general subgradient of $r(\cdot)$ at $\bar{x}$, written as $g \in \partial r(\bar{x})$. 
If a Lipschitz function $r$ is lower regular (or subdifferentially regular as in 7.25,~\cite{rockafellar1998}), 
then its lower general subdifferential is the same as its lower regular subdifferential (Corollary 8.11,~\cite{rockafellar1998}). 
Due to the popularity of lower-type properties in optimization, lower general subgradient is often simply called general subgradient
(8.3,~\cite{rockafellar1998}).

Next, as introduced in~\cite{rockafellar1998} and detailed in~\cite{mordukhovich2004upp}, upper regular subdifferential is defined through
\begin{equation}\label{def:upp-subgradient}
\centering
 \begin{aligned}
	 \hat{\partial}^+ r(\bar{x}) \coloneqq& - \hat{\partial} (-r)(\bar{x})
	 =\left\{ g\in\Rbb^n |\limsup_{ \substack{x\to \bar{x}\\x\neq \bar{x} }}\frac{r(x)-r(\bar{x})-\langle g,x-\bar{x}\rangle  }{\norm{x-\bar{x}}}\leq 0\right\}.
 \end{aligned}
\end{equation}
The general upper subdifferential at $\bar{x}$ is given as $\partial^+ r(\bar{x}) \coloneqq -\partial (-r)(\bar{x})$.
A function $r(\cdot)$ is called upper regular if $-r(\cdot)$ is lower regular~\cite{rockafellar1998}. 
Some examples of upper regular functions include all continuous concave functions and all functions strictly differentiable. 
The upper/lower general subdifferential is locally bounded for Lipschitz continuous function (9.13,~\cite{rockafellar1998}). 
An important relationship between an upper regular subgradient and Clarke subdifferential is established next.

\begin{lemma}\label{lem:low-sub-clarke}
	(lower subdifferential and Clarke subdifferential) Let $-r(\cdot)$ be a Lipschitz and lower regular function at $\bar{x}$, then its Clarke subdifferential and lower general subdifferential at $\bar{x}$ are equivalent, \textit{i.e.} $ \partial (-r)(\bar{x}) = \bar{\partial} (-r)(\bar{x})$.
\end{lemma}
\begin{proof}
	This Lemma is readily available from the results in~\cite{rockafellar1998}. By 8.49 and 9.13 in~\cite{rockafellar1998}, for a Lipschitz continuous function $-r(\cdot)$, we have its Clarke subdifferential  
	$\bar{\partial} (-r)(\bar{x}) = \text{con}\partial (-r)(\bar{x})$.
	For a lower regular function $-r(\cdot)$, $\hat{\partial} (-r)(\bar{x}) = \partial (-r)(\bar{x})$ (8.11,~\cite{rockafellar1998}).
	Given the convexity of lower regular subdifferential $\hat{\partial} (-r)(\bar{x})$ (8.6,~\cite{rockafellar1998}), 
	 $\bar{\partial} (-r)(\bar{x})= \text{con}\hat{\partial} (-r)(\bar{x})=\hat{\partial} (-r)(\bar{x})=\partial(-r)(\bar{x})$.
\end{proof}

\begin{lemma}\label{lem:upp-sub-clarke}
	(upper subdifferential and Clarke subdifferential) Let $r(\cdot)$ be Lipschitz and upper regular at $\bar{x}$, 
	then its upper general subdifferential and Clarke subdifferential at $\bar{x}$ are equivalent. In particular, if $g \in \partial^+ r(\bar{x})$, then $g \in \bar{\partial} r(\bar{x})$.
\end{lemma}
\begin{proof}
	We rely on the properties of lower regular functions in Lemma~\ref{lem:low-sub-clarke}.
	Since $r(\cdot)$ is upper regular, by definition $-r(\cdot)$ is lower regular and  $ -\hat{\partial} (-r)(\bar{x}) = \hat{\partial}^+ r(\bar{x}) =\partial^+ r(\bar{x})$. 
By the symmetry property of the Clarke
	subgradient for locally Lipschitz functions (2.3.1,~\cite{clarke1983}), we have  $ \bar{\partial} r (\bar{x})=-\bar{\partial} (-r) (\bar{x})$. 
	By Lemma~\ref{lem:low-sub-clarke}, $ -\bar{\partial} (-r) (\bar{x}) =  -\hat{\partial} (-r)(\bar{x}) $.  
	Therefore, $ \partial^+ r(\bar{x}) = -\hat{\partial} (-r)(\bar{x}) = \bar{\partial} r (\bar{x})$.
	As a result, $g \in \partial^+ r(\bar{x})$ is also a Clarke subgradient, \textit{i.e.} $g\in \bar{\partial} r(\bar{x})$.
\end{proof}
Due to their equivalence, for upper regular functions, Clarke subgradient and upper general subgradient can be used interchangeably.
Moving on to more restrictive assumptions than regularity, lower-$C^1$ functions, introduced in~\cite{Spingarn1981SubmonotoneSO} and~\cite{rockafellar1998}, are commonly assumed in nonsmooth optimization and have a few equivalent definitions~\cite{daniilidis2004}. 
A function $r(\cdot):O\to \Rbb$, where $O\subset \Rbb^n$ is open is said to be lower-$C^k$ on $O$, if on some neighborhood $V$ of each $\bar{x}\in O$ there is a representation
\begin{equation}\label{eqn:lowc1-def-0}
\centering
 \begin{aligned}
	 r(x) =  \underset{\substack{t \in T}}{\text{max}} \ r_t(x),
 \end{aligned}
\end{equation}
where the functions $r_t(\cdot)$ are of class $C^k$ on $V$ and the index set $T$ is a compact space such that $r_t(\cdot)$ 
and all of its partial derivatives through order $k$ are jointly continuous on $(t,x)\in T\times V$. 
Similarly, a function is upper-$C^k$ on $O$ if on a neighborhood $V$ of $\bar{x} \in O$ we can write
\begin{equation}\label{eqn:uppc1-def-0}
\centering
 \begin{aligned}
	 r(x) =  \underset{\substack{t \in T}}{\text{min}} \ r_t(x),
 \end{aligned}
\end{equation}
where $r_t(\cdot)$ are of class $C^k$ on $V$. The set $T$ is compact and $r_t(\cdot)$ 
and all of the partial derivatives through order $k$ are jointly continuous on $(t,x)\in T\times V$.

A widely used $T$ is a closed and bounded subset of $\Rbb^p$.  Thus, if
\begin{equation}\label{eqn:uppc1-def-1}
\centering
 \begin{aligned}
	 r(x) =  \underset{\substack{t \in T}}{\text{min}} \ p(t,x)
 \end{aligned}
\end{equation}
for all $x\in O$, and $p(\cdot,\cdot)$ and its first- and second-order partial derivatives in $x$ 
depend continuously on $(t,x)$, $r(\cdot)$ is upper-$C^2$.
Rockafellar~\cite{rockafellar1981Fav} and Clarke~\cite{clarke1995Proximal} further simplified the objective in~\eqref{eqn:uppc1-def-1} for Lipschitz functions. If $r(\cdot):U\to\Rbb$, where $U\subset\Rbb^n$ is an open, convex and bounded set, is Lipschitz, then it is upper-$C^2$ 
if there exists $\sigma>0$, a compact set $S$ and continuous functions $b(\cdot):S\to\Rbb^n, c(\cdot):S\to\Rbb$ such that
\begin{equation}\label{eqn:uppc1-def-lip}
\centering
 \begin{aligned}
	 r(x) =  \underset{\substack{s \in S}}{\text{min}} \{ \sigma \norm{x}^2 - \langle b(s),x\rangle -c(s) \}
 \end{aligned}
\end{equation}
for $\forall x\in U$. 

While the original definition has clear indication for two-stage optimization problems, 
an alternative definition based on the function and subgradient value is more useful in analysis. 
A function is called lower-$C^1$ at $\bar{x}$ if
\begin{equation}\label{eqn:lowc1-def}
\centering
 \begin{aligned}
	 \forall \epsilon >0, \ \exists \ \rho >0, s.t. \ \forall x,x' \in B_{\rho}(\bar{x}),\ g \in \partial r(x),\\
	 r(x') - r(x) - \langle g,x'-x\rangle \geq -\epsilon \norm{x'-x}.
 \end{aligned}
\end{equation}
A function is lower-$C^1$ on an open set $O$ if it is lower-$C^1$ for all $x\in O$.
By definition, a function $r(\cdot)$ is upper-$C^1$ if $-r(\cdot)$ is lower-$C^1$ at $\bar{x}$. 


An intuitive, equivalent definition of a finite-valued, lower-$C^2$ function $r(\cdot)$ on an open set $O$ is that for any point $\bar{x}\in O$, there exists a threshold value $\rho_{0}>0$ such that 
$r(\cdot)+\frac{\rho}{2}\norm{\cdot}^2$ is convex on an open neighborhood of $\bar{x}$ for all $\rho>\rho_0$. 
It is worth noting that another popular property: prox-regularity is closely related to lower-$C^2$.
Lower-$C^2$ functions are prox-regular and for Lipschitz continuous functions, 
prox-regularity also guarantees lower-$C^2$ on an open set $O\subset\Rbb^n$ (13.33,~\cite{rockafellar1998}). 

Since a function $r(\cdot)$ is called upper-$C^2$ if and only if $-r(\cdot)$ is lower-$C^2$ at $\bar{x} \in \Rbb^n$, for a finite, Lipschitz and upper-$C^2$ function $r(\cdot)$ on an open set $O$ with $\bar{x}\in O$ , there exists $\rho>0$, such that
\begin{equation}\label{eqn:lowc2-def}
\centering
 \begin{aligned}
	 -r(x) + r(\bar{x}) - \langle -g,x-\bar{x}\rangle \geq -\frac{\rho}{2}\norm{x-\bar{x}}^2, 
 \end{aligned}
\end{equation}
where $-g\in\partial (-r)(\bar{x})$ and $x,\bar{x}\in O$. Since $-r(\cdot)$ is lower-$C^2$ and thus lower regular, 
$-g\in \hat{\partial} (-r)(\bar{x})$. By definition, $g\in \partial^+ (r)(\bar{x})$ and by Lemma~\ref{lem:upp-sub-clarke}, $g\in \bar{\partial} r(\bar{x})$.
Inequality~\eqref{eqn:lowc2-def} is equivalent to 
\begin{equation}\label{eqn:uppc2-def}
\centering
 \begin{aligned}
	 r(x) - r(\bar{x}) - \langle g,x-\bar{x} \rangle \leq \frac{\rho}{2}\norm{x-\bar{x}}^2, 
 \end{aligned}
\end{equation}
where $g \in \partial^+ r(\bar{x})$ and $x,\bar{x}\in O$.
Notice that there exists a uniform $\rho$ such that~\eqref{eqn:uppc2-def} stands 
for all $x\in D \subset O$ and $g\in \partial^+ r(\bar{x})$, where $D$ is compact~(10.54,~\cite{rockafellar1998},~\cite{hare2010}).
We refer to~\eqref{eqn:uppc2-def} as the upper-$C^2$ inequality. 
It is worth pointing out that upper-$C^2$ does not guarantee differentiability, lower-$C^1$ or lower regularity. 
It does ensure upper regularity where the upper subdifferential is equivalent to Clarke subdifferential.
A simple example is
\begin{equation} \label{eqn:opt-ms-appx}
 \centering
  \begin{aligned}
	  f(x) =
              \begin{cases}
	       \begin{aligned}
	        & x,  &-1\leq &x\leq 0, \\
		&\frac{1}{2}x,  & 0\leq &x\leq 1,	  
	       \end{aligned}
	      \end{cases}
  \end{aligned}
\end{equation}
which is concave and not differentiable or lower regular at $x=0$.

A large number of nonsmooth nonconvex functions satisfy some of the properties described in this section. They allow 
the use of Clarke subgradient in the analysis of global convergence.
To obtain desired properties for the objective function, regularization techniques might be necessary. 

\subsection{Smoothing of the second-stage problem}\label{sec:smoothing}
In many two-stage stochastic optimization problems, the second-stage solution function $r(\cdot)$ 
while lacking differentiability, satisfy the conditions for upper-$C^2$ property. 
For example, if the coupling of variables are in the smooth objective only, then by~\eqref{eqn:uppc1-def-1}, 
$r(\cdot)$ is upper-$C^2$.  
In our target application, the first-stage variable $x$  
is coupled linearly in the constraints of the second-stage problems~\cite{petra_21_gollnlp}. 
If the coupling exists in inequality however, upper-$C^2$ conditions might not be satisfied.
Regularization of the problem could help smooth out the non-differentiability.
To see this, we first make the following assumption for this section based on observation from SCACOPF problems.
\begin{assumption}\label{assp:rec-linear}
	The problem~\eqref{eqn:opt-rc} can be reformulated with uncoupled objective and coupled constraints that are linear in the first-stage variable $x$. 
\end{assumption}
Using non-negative slack variables $s^l_i\geq 0,s^u_i\geq 0$, the coupled inequality constraints in~\eqref{eqn:opt-rc} can be
converted to equality constraints with 
\begin{equation} \label{eqn:opt-rc-cons-req}
 \centering
  \begin{aligned}
    d(x,y_i) - d^l_i &= s^l_i\\
    d^u_i - d(x,y_i) &= s^u_i
  \end{aligned}
\end{equation}
For simplicity reasons, the subscript $i$ for the $i$th second-stage problem is dropped.
Moreover, the slack variables $s^l_i,s^u_i$ can be regarded as part of the 
optimization variables $y$. Clearly the relevant equality constraints are linear in $s$ by definition. 
Separating the coupled and uncoupled constraints, by Assumption~\ref{assp:rec-linear}, we denote 
the coupled constraints as 
\begin{equation} \label{eqn:opt-rc-cons-req-lin}
 \centering
  \begin{aligned}
	   Wx  - h(y) = 0,\\
  \end{aligned}
\end{equation}
where the slack variables are considered part of $y$ 
and $W$ is the corresponding linear operator. Using one of the standard matrix norms, 
$W$ is assumed to be bounded with $\norm{W}=w$.
The second-stage problem from~\eqref{eqn:opt-rc} now becomes
\begin{equation} \label{eqn:opt-rc-eq-2cons}
 \centering
  \begin{aligned}
   r(x) =& \underset{\substack{y }}{\text{minimize}} 
	  & & p(y)\\
   &\text{subject to}
          & & Wx  - h(y) = 0,\\
	  &&& c(y) =c_{E,2}\\
	  &&& d^l\leq d(y) \leq d^u \\
   &&& y^l \leq y \leq y^u,\\
  \end{aligned}
\end{equation}
where $c_{E,2}$ is used to emphasize the uncoupled constraint.
Given smooth $h(\cdot)$, $r(\cdot)$ still might not be differentiable or upper-$C^2$. 
However, it is possible now to apply the quadratic penalty method~\cite{Nocedal_book} to achieve upper-$C^2$.
To illustrate both points, two simple examples are presented where differentiability can be improved. 
Example 1 is a one-dimensional optimization problem with equality constraint given in~\eqref{eqn:opt-rc-eg1} 
\begin{equation} \label{eqn:opt-rc-eg1}
 \centering
  \begin{aligned}
   r(x) =& \underset{\substack{y}}{\text{min}} 
	  & & y\\
   &\text{s.t.}
	  & & y^2 = x\\
	  &&&  y\geq 0,
  \end{aligned}
\end{equation}
where $y\in\Rbb$ and $x\geq 0$.
It is obvious that the solution function is $r(x) = \sqrt{x}, x\geq 0$,
which is continuous yet not Lipschitz continuous at $x=0$.
Using the quadratic penalty function with coefficient $\mu$, the optimization problem
is smoothed to~\eqref{eqn:opt-rc-eg1-al}
\begin{equation} \label{eqn:opt-rc-eg1-al}
 \centering
  \begin{aligned}
   r_\mu(x) =& \underset{\substack{y}}{\text{min}} 
	  & & y + \mu\norm{y^2-x}^2\\
   &\text{s.t.}
	  & & y \geq 0.
  \end{aligned}
\end{equation}
The smoothed solution function $r_{\mu}(\cdot)$ becomes
Lipschitz continuous at $x=0$, as illustrated on the left plot in Figure~\ref{fig:smoothing}.
The value of $\mu$ is a trade-off between accurate approximation of $r(\cdot)$ and 
the range of the transition period close to $x=0$.
Example 2 considers an optimization problem on $y\in\Rbb$ with an inequality constraint 
\begin{equation} \label{eqn:opt-rc-eg2}
 \centering
  \begin{aligned}
   r(x) =& \underset{\substack{y}}{\text{min}} 
	  & & ay^2+by\\
   &\text{s.t.}
	  & & y \geq x\\
	  &&&  y \geq 0.
  \end{aligned}
\end{equation}
The solution function $r(\cdot)$ is differentiable but not continuously differentiable at $x=0$. 
With quadratic penalty and a slack variable $s$ for the inequality constraint, the problem transforms to
\begin{equation} \label{eqn:opt-rc-eg2-al}
 \centering
  \begin{aligned}
	  r_{\mu}(x) =& \underset{\substack{y}}{\text{min}} 
	  & & ay^2+by+\mu\norm{x+s-y}^2\\
   &\text{s.t.}
	  & & y,s \geq 0.
  \end{aligned}
\end{equation}
The function $r_{\mu}(\cdot)$ is then smoothed into a continuously differentiable one as seen on the right in Figure~\ref{fig:smoothing}.
The smoothed function $r_{\mu}(\cdot)$ in both examples are now upper-$C^2$.
\begin{figure}[H]
  \centering
  \includegraphics[width=0.95\textwidth]{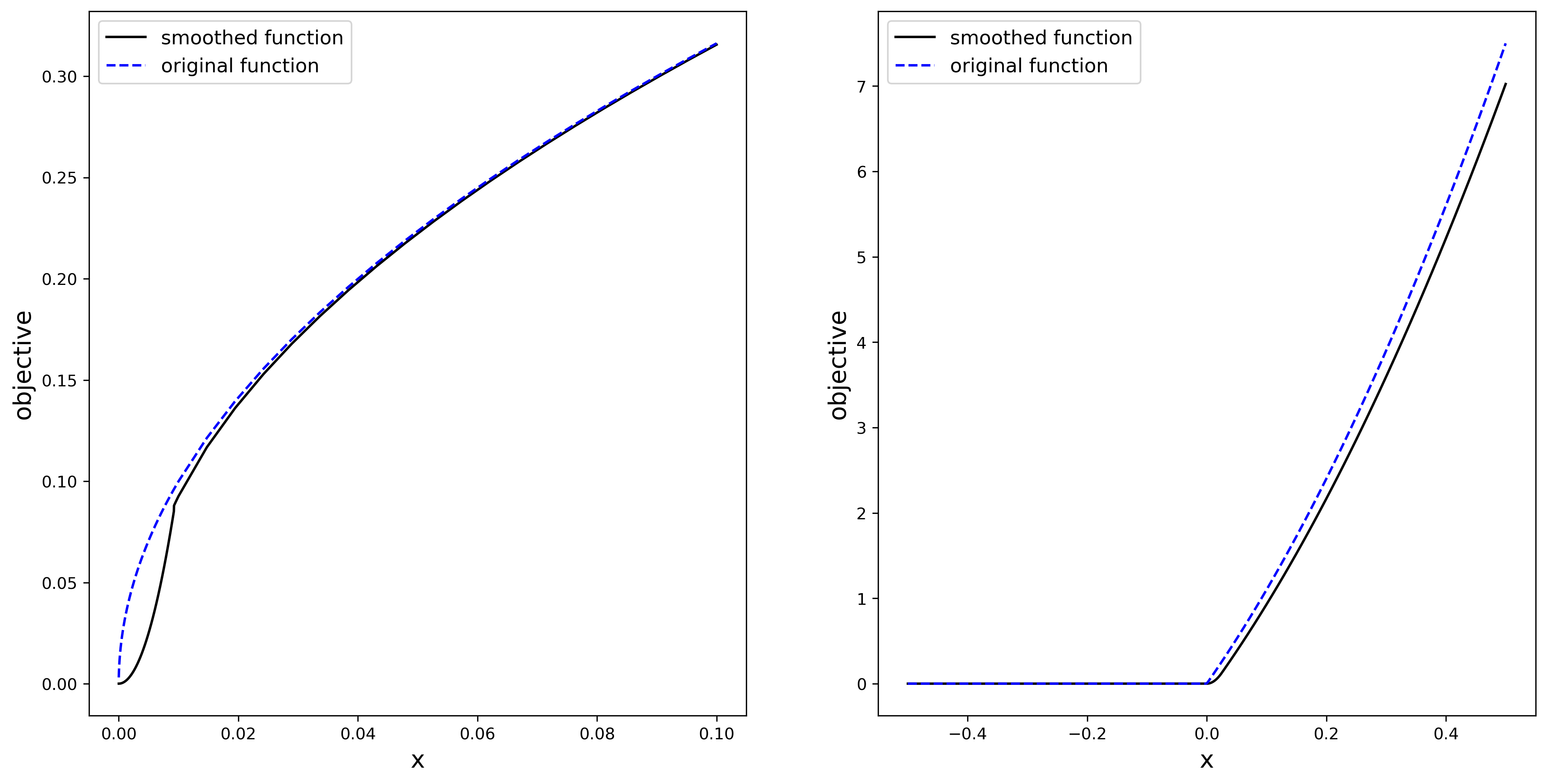}
 \caption{Quadratic penalty smoothing example: example 1 on the left, example 2 on the right }
 \label{fig:smoothing}
\end{figure}
Similarly, for the more general case, it is possible to obtain desirable properties such as upper-$C^2$ 
for second-stage solution functions by incorporating the coupled constraints into the objective through a quadratic penalty 
such as 
\begin{equation} \label{eqn:opt-rc-mu-al}
 \centering
  \begin{aligned}
	  r_{\mu}(x) =& \underset{\substack{y}}{\text{minimize}} 
	  & & \mu\norm{Wx-h(y)}^2+p(y) \\
   &\text{subject to}
	  & & c(y) =c_{E,2}\\
	  &&& d_l\leq d(y) \leq d_u \\
   &&& y^l \leq y \leq y^u\\
  \end{aligned}
\end{equation}
where $\mu$ is the penalty coefficient.
As $\mu \to \infty$, the feasible accumulation points of the solutions to
~\eqref{eqn:opt-rc-mu-al} become the solution to that of~\eqref{eqn:opt-rc}~\cite{Nocedal_book}. 
It is worth pointing out that the coupling part of $x$ is converted into a squared distance function, 
which has been studied extensively~\cite{poliquin2000,rockafellar1998}. 
While we focus on the linearly coupled constraint from SCACOPF, as the goal is to pursue upper-$C^2$ of $r(\cdot)$, it 
is clear from its definition~\eqref{eqn:uppc1-def-1} that linearity is not necessary. In fact a smooth coupling in both $x$ and $y$ 
with quadratic penalty would suffice.

The properties of $r_{\mu}(\cdot)$ from~\eqref{eqn:opt-rc-mu-al} is examined next.
The feasible set of $y$ is denoted as $\Phi(x)$.
An important fact that is repeatedly used is that $\Phi(x)\equiv \Phi$, independent of $x$ due to 
the regularization.
For $r_{\mu}(\cdot)$ to be continuous at $x$,  the problem needs to have certain bounded properties.
For example, it is common to assume coercivity~\cite{dao2016} or level-bounded objective functions~\cite{dao2015}. On the 
other hand, in our target applications, 
the variables $x$ and $y$ are bounded above and below by real, finite values. The slack 
variable $s$, defined through functions on $x,y$, are effectively bounded as well.
Hence, we choose to assume bounded domain for $x$ and compact domain for $y$, denoted as $X\in\Rbb^n$ and $Y\in\Rbb^p$, respectively for simplicity.
Notice again $Y$ is now independent of $x$.
The optimal solution set is denoted as $S(x)\subset Y$ and the continuity result 
is given in Lemma~\ref{lemma:continuity} based on Chapter 4 from~\cite{shapiro2000}, the proof of which 
requires additional definitions and is left for the Appendix.
\begin{lemma}\label{lemma:continuity}
The optimal value function of the smoothed second-stage problem $r_{\mu}(\cdot)$ is continuous,
and the multifunction $x\to S(x)$ is upper semicontinuous at $x$
\end{lemma}

\noindent In addition, the compact domain and linear coupling of $r_{\mu}(\cdot)$ leads to the 
following Lemma.
\begin{lemma}\label{lem:Lip}
The optimal solution function $r_{\mu}(\cdot)$ is Lipschitz continuous on its domain.
\end{lemma}
\begin{proof}
Given that $x$ is bounded, the domain of $y$ is compact, and the continuous differentiability of the objective in~\eqref{eqn:opt-rc-mu-al},
$r_{\mu}(x)$ is bounded.
Let $M\in \Rbb >0$ be the upper bound of the absolute value of the coupled constraint, such that
\begin{equation} \label{eqn:opt-rc-mu-lem1}
\centering
 \begin{aligned}
	 \norm{Wx-h(y)} \leq M,  \quad \forall   x \in X,  y\in Y.
 \end{aligned}
\end{equation}
Denote by $x_1, x_2$ two points in the domain and $y_1 \in S(x_1), y_2 \in S(y_2)$ their corresponding optimal solutions.
	To simplify the notations, write $h_1=h(y_1),h_2=h(y_2),p_1=p(y_1),p_2=p(y_2)$. Since $y_1\in S(x_1)$ and 
	$Y$ is independent of $x$, we have
\begin{equation} \label{eqn:opt-rc-mu-lem2}
\centering
 \begin{aligned}
	 \mu\norm{Wx_1 - h_1}^2 +p_1&\leq \mu\norm{Wx_1 - h_2}^2 +p_2\\
	 &=\mu\norm{W(x_1-x_2)+Wx_2-h_2}^2+p_2\\
	 &\leq \mu\norm{W(x_1 -x_2)}^2+2\mu[W(x_1-x_2)]^T (Wx_2-h_2)   \\
	 &\ \ \ +\mu\norm{Wx_2- h_2}^2+p_2.\\
 \end{aligned}
\end{equation}
	Given $w = \norm{W}$,  
\begin{equation} \label{eqn:opt-rc-mu-lem3}
\centering
 \begin{aligned}
	 \mu\norm{Wx_1 - h_1}^2 +p_1- &\mu\norm{Wx_2- h_2}^2 -p_2\\
	 \leq& \mu\norm{W(x_1 -x_2)}^2
	 +2\mu[W(x_1-x_2)]^T[Wx_2-h_2] \\
	 \leq& \mu\norm{W(x_1-x_2)}\left(  \norm{W(x_1-x_2)} 
	    +2\norm{Wx_2-h_2}\right)\\
	 \leq& \mu w\norm{x_1-x_2}\left(w\norm{x_1-x_2}+2M\right).
 \end{aligned}
\end{equation}
Similarly,
\begin{equation} \label{eqn:opt-rc-mu-lem4}
\centering
 \begin{aligned}
	 \mu\norm{Wx_2 - h_2}^2 +p_2-& \mu\norm{Wx_1- h_1}^2 -p_1\\
	 \leq& \mu\norm{W(x_1 -x_2)}^2
	 +2\mu[W(x_2-x_1)]^T[Wx_1-h_1] \\
	 \leq& \mu\norm{W(x_1-x_2)}|\left(\norm{W(x_1-x_2)} 
	  +2\norm{Wx_1-h_1}\right)\\
	 \leq& \mu w\norm{x_1-x_2}\left(w\norm{x_1-x_2}+2M\right).
 \end{aligned}
\end{equation}
Let $\norm{x^u-x^l}=D, L=\mu w(wD+2M)$, we have  
\begin{equation} \label{eqn:opt-rc-mu-lem5}
\centering
 \begin{aligned}
	 |r_{\mu}(x_1)-r_{\mu}(x_2)|
	 =&\left|\mu \norm{Wx_2 - h_2}^2 +p_2- \mu\norm{Wx_1- h_1}^2 -p_1\right|\\
	 \leq& \mu w\norm{x_1-x_2}\left(w\norm{x_1-x_2}+2M\right)\\
	 \leq& \mu w\norm{x_1-x_2}\left(wD+2M\right)\\
	 \leq& L\norm{x_1-x_2}.
 \end{aligned}
\end{equation}
\end{proof}
\noindent Next we show that an upper general subgradient of $r_{\mu}(\cdot)$ at $\bar{x}$ can be expressed as
\begin{equation}\label{eqn:opt-rc-mu-al-sub}
\begin{aligned}
	g_{\mu}(\bar{x})= 2\mu W^T(W\bar{x}-h(\bar{y})) \in \partial^{+} r_{\mu}(\bar{x}).
\end{aligned}
\end{equation}
\begin{proposition}\label{prop:upper-sub}
	The vector $g_{\mu}(\bar{x})$ in~\eqref{eqn:opt-rc-mu-al-sub} is an upper general subgradient of 
	$r_{\mu}(\cdot)$ in~\eqref{eqn:opt-rc-mu-al}. In addition,
	the upper-$C^2$ inequality in~\eqref{eqn:uppc2-def} is satisfied with $g_{\mu}(\bar{x})$, \textit{i.e.}, 
	for $x,\bar{x}$ in the domain, there exists $C>0$ such that $r_{\mu}(x)-r_{\mu}(\bar{x})- g_{\mu}^T(\bar{x})(x-\bar{x})\leq C\norm{x-\bar{x}}^2$.
\end{proposition}
\begin{proof}
    Let $p=p(y),\bar{p}=p(\bar{y}),h=h(y),\bar{h}=h(\bar{y})$,
	where $y\in S(x)$ and $\bar{y}\in S(\bar{x})$. The left-hand side of the inequality in~\eqref{eqn:uppc2-def} can be written as
\begin{equation} \label{eqn:opt-rc-mu-simp-p1}
\centering
 \begin{aligned}
	 r_{\mu}(x)-&r_{\mu}(\bar{x})- g_{\mu}^T(\bar{x})(x-\bar{x}) \\
	 =& \mu (x^T W^T W x -2h^TWx +h^Th)+p
	   - \mu (\bar{x}^T W^T W \bar{x} -2\bar{h}^TW\bar{x} +\bar{h}^T\bar{h})-\bar{p}\\
	  &  - 2\mu  (\bar{x}^TW^TWx- \bar{x}^TW^TW\bar{x}-\bar{h}^TWx+\bar{h}^TW\bar{x}  )\\
	 =& \mu( \norm{W(x-\bar{x})}^2 -2h^TWx +h^Th-\bar{h}^T\bar{h}
		    +2\bar{h}^TWx )+p-\bar{p}\\
	 =& \mu( \norm{W(x-\bar{x})}^2 +\norm{h - Wx}^2 -\norm{Wx}^2  -\norm{\bar{h} - Wx}^2
	  +\norm{Wx}^2)+p-\bar{p}\\
	 =& \mu\norm{W(x-\bar{x})}^2 +\mu\norm{h - Wx}^2+p  -\mu\norm{\bar{h} - Wx}^2-\bar{p}.\\
 \end{aligned}
\end{equation}
	Further, since $y\in S(x)$, we have  
\begin{equation} \label{eqn:opt-rc-mu-simp-p2}
\centering
 \begin{aligned}
	 r_{\mu}(x)-r_{\mu}(\bar{x})- g_{\mu}^T(\bar{x})(x-\bar{x})
	 \leq& \mu w^2 \norm{x-\bar{x}}^2 +\mu\norm{h - Wx}^2+p  -\mu\norm{\bar{h} - Wx}^2-\bar{p}\\
	 \leq& \mu w^2 \norm{x-\bar{x}}^2. \\
 \end{aligned}
\end{equation}
	Taking the limit so that $x\to\bar{x}$,  
\begin{equation} \label{eqn:opt-rc-mu-simp-p3}
\centering
 \begin{aligned}
	\limsup_{ \substack{x\to \bar{x}\\x\neq \bar{x} }}  \frac{ r_{\mu}(x)-r_{\mu}(\bar{x})- g_{\mu}^T(\bar{x})(x-\bar{x})}{\norm{x-\bar{x}}} \leq 0 \\
 \end{aligned}
\end{equation}
	By definition~\eqref{def:upp-subgradient}, $g_{\mu}(\bar{x})$ is a upper regular subgradient, hence a upper general subgradient. 
Let $C=\mu w^2$, the upper-$C^2$ inequality is satisfied.
\end{proof}

\noindent The quadratic penalty smoothing of the second-stage problems allows the following important property to be achieved.
\begin{proposition}\label{lem:upperC2}
	The second-stage optimal solution function $r_{\mu}(\cdot)$ is upper-$C^2$ and thus satisfies the upper-$C^2$ inequality in~\eqref{eqn:uppc2-def} on its domain.
\end{proposition}
\begin{proof}
	There are multiple ways to show this. The most straightforward one is to apply directly the definition~\eqref{eqn:uppc1-def-1}. From Lemma~\ref{lem:Lip}, $r_{\mu}(\cdot)$ is Lipschitz continuous. 
	 From the definition of $r_{\mu}(\cdot)$ in~\eqref{eqn:opt-rc-mu-al},
	 the feasible set of the optimization variables $y$ is compact in $\Rbb^m$ and independent of $x$.
	 The coupling is now only in the objective, with the non-coupled part $p(\cdot)$  smooth in $y$.
	 Further, the coupling in objective is quadratic in $x$, rendering it twice continuously differentiable in a neighborhood of both $x$ 
	 and $y$. By definition leading to~\eqref{eqn:uppc1-def-1}, $r_{\mu}(\cdot)$ is upper-$C^2$.

	Intuitively, as pointed out in 10.57 of~\cite{rockafellar1998}, squared distance function on a nonempty closed set is upper-$C^2$. The objective, while having an additional smooth function $p(\cdot)$, is only coupled in the squared distance function.
		 From the viewpoint of~\eqref{eqn:uppc1-def-lip}, given a $x \in \Rbb^n$, an open, convex and bounded neighborhood of $x$ can be found where 
	 the objective that defines $r_{\mu}(\cdot)$ in~\eqref{eqn:opt-rc-mu-al} fits the form in~\eqref{eqn:uppc1-def-lip}. It is pointed out that such a neighborhood could contain infeasible points for the first-stage problem, while not affecting the property of the function $r_{\mu}(\cdot)$ itself.
	 Therefore, $r_{\mu}(\cdot)$ is upper-$C^2$. The proposition then follows.
\end{proof}

\begin{proposition}\label{lem:regular-subgradient}
	If the solution set $S(x)$ is a singleton at $\bar{x}$ such that $S(\bar{x}) = \{\bar{y}\}$, 
	then $r_{\mu}(\cdot)$ is differentiable at $\bar{x}$ and $g_{\mu}(\bar{x})= 2\mu W^T(W\bar{x}-h(\bar{y}))$ is the gradient  $\nabla r_{\mu}(\bar{x})$. 
\end{proposition}
\begin{proof}
	Let us take $x,\bar{x}\in X$ and use shorthands $\bar{p}=p(\bar{y}),p=p(y),\bar{h}=h(\bar{y}),h=h(y)$.
	We can write 
\begin{equation} \label{eqn:opt-rc-mu-al-subgradient}
\centering
 \begin{aligned}
	 r_{\mu}&(x)-r_{\mu}(\bar{x})- \bar{g}_{\mu}^T(x-\bar{x}) \\
	 =&\Big(\mu \norm{Wx-h}^2+p \Big)-\Big(\mu \norm{W\bar{x}-\bar{h}}^2+\bar{p} \Big) - 2\mu (W\bar{x}-\bar{h})^TW(x-\bar{x})\\
	 =& \mu( \norm{W(x-\bar{x})}^2 -2h^TWx +h^Th  -\bar{h}^T\bar{h}
		    +2\bar{h}^TWx )+p-\bar{p}\\
	 =& \mu( \norm{W (x-\bar{x})}^2+2h^TW\bar{x}-2\bar{h}^TW\bar{x}-2h^TWx +2\bar{h}^TWx)\\
	  &       +\mu\norm{h - W\bar{x}}^2+p-\mu\norm{\bar{h} - W\bar{x}}^2-\bar{p}\\
	 =& \mu(\norm{W (x-\bar{x})}^2-2\bar{h}^TW(\bar{x}-x) +2h^TW(\bar{x}-x) )\\
	  &       +\mu\norm{h - W\bar{x}}^2+p-\mu\norm{\bar{h} - W\bar{x}}^2-\bar{p}\\
	 \geq& \mu( \norm{W (x-\bar{x})}^2 -2\bar{h}^TW(\bar{x}-x) +2h^TW(\bar{x}-x))\\
	 =& \mu( \norm{W (x-\bar{x})}^2 -2[h-\bar{h}]^TW(x-\bar{x}) )\\
 \end{aligned}
\end{equation}
Since $\bar{h}=h(\bar{y})$ is unique at $\bar{x}$, $h\to\bar{h}$ as $x\to\bar{x}$ and
\begin{equation} \label{eqn:subgradient-lem2.6}
\centering
 \begin{aligned}
	 \liminf_{\substack{x\to \bar{x}\\ x \neq \bar{x} }}\frac{r_{\mu}(x)-r_{\mu}(\bar{x})-g_{\mu}^T(x-\bar{x})}{\norm{x-\bar{x}}}
	 &\geq\lim_{\substack{x\to \bar{x}\\x\neq \bar{x} }}\frac{\mu( -2w \norm{h-\bar{h}}\norm{x-\bar{x}})}{\norm{x-\bar{x}}}\\
	 &= \lim_{\substack{x\to x\\\bar{x}\neq \bar{x} }}-2\mu w\norm{h-\bar{h}} = 0.
 \end{aligned}
\end{equation}
	In addition, from the proof of Proposition~\ref{prop:upper-sub}, we know
\begin{equation} \label{eqn:gradient-up}
\centering
 \begin{aligned}
	 \limsup_{\substack{x\to \bar{x}\\ x \neq \bar{x} }}\frac{r_{\mu}(x)-r_{\mu}(\bar{x})-g_{\mu}^T(x-\bar{x})}{\norm{x-\bar{x}}}
	 \leq 0 
 \end{aligned}
\end{equation}
	Therefore, 
\begin{equation} \label{eqn:gradient-equal}
\centering
 \begin{aligned}
	 \lim_{\substack{x\to \bar{x}\\ x \neq \bar{x} }}\frac{r_{\mu}(x)-r_{\mu}(\bar{x})-g_{\mu}^T(x-\bar{x})}{\norm{x-\bar{x}}}
	 = 0, 
 \end{aligned}
\end{equation}
and the proposition follows.
\end{proof}

On the other hand, if the optimal solution $\bar{h}$ is not unique, there could exist multiple upper regular subgradients
and $r_{\mu}(\cdot)$ might not be differentiable. Indeed, uniqueness of the solution $y\in S(x)$ is not guaranteed. 
It is possible to put more restrictions on $h(\cdot)$ so that $r_{\mu}(\cdot)$ becomes continuously differentiable.
\begin{proposition}\label{lem:lowerc2}
	The optimal solution function $r_{\mu}(\cdot)$ is lower-$C^2$ on a neighborhood $O$ of $\bar{x}$, 
	if $h(y),y\in S(x)$, is Lipschitz continuous on $O$, \textit{i.e.}, $\forall x_1,x_2\in O$,
	there exists $L_h>0$ such that $\norm{h(y_1)-h(y_2)}<L_h\norm{x_1-x_2}$ on $O$. Moreover, $r_{\mu}(\cdot)$ is 
	continuously differentiable at $\bar{x}$.
\end{proposition}

To summarize, while in many nonsmooth nonconvex optimization methods, lower-$C^2$ or prox-regularity 
are assumed, for some important applications such as the decomposed formulation 
of SCACOPF, lower regularity of the objective is not available.  
The upper-type properties however appears natural for two-stage problems,
and has motivated us to make assumptions  
differently than the conventional ones and design algorithms accordingly. 

There are multiple convergence definitions in nonsmooth nonconvex analysis, \textit{e.g.},
stationary point, Karush–Kuhn–Tucker (KKT) point, (Fritz-John) critical point, etc. 
In this paper, 
the focus is on first-order optimality condition with Clarke subgradient. 
Without losing generality, problem~\eqref{eqn:opt-ms} can be recast for simplicity as
\begin{equation} \label{eqn:opt-ms-simp}
 \centering
  \begin{aligned}
   &\underset{\substack{x}}{\text{minimize}} 
	  & & r(x)\\
   &\text{subject to}
	  & & c(x) = 0 \\
	  &&& 0 \leq x \leq x_u,
  \end{aligned}
\end{equation}
where $c(x):\Rbb^n\to \Rbb^m$.
As mentioned earlier in this section, transforming~\eqref{eqn:opt-ms} requires the new variables $x$ in~\eqref{eqn:opt-ms-simp} 
to contain slack variables, which are implicitly bounded
from bound constraints on $x$. 
We opt to keep the upper bound $x_u$ in the bound constraints 
explicit, as in~\eqref{eqn:opt-ms} to emphasize that the feasible set of $x$ is bounded . 
We point out that the bound constraints form a convex set, 
with the nonconvexity left to the equality constraints.

Problem~\eqref{eqn:opt-ms-simp} is assumed to be calm (6.4,~\cite{clarke1983}) at its local minimum.  
Calmness can be viewed as a weak constraint qualification (6.4,~\cite{clarke1983}).
In particular, the widely adopted linear
independence constraint qualification (LICQ) (17.2,~\cite{Nocedal_book}) in smooth optimization ensures
calmness. 
Calmness guarantees the Lagrange 
multiplier for the objective function in Fritz-John critical point equation~\cite{clarke1983} is nonzero(6.4.4,~\cite{clarke1983}). 
Therefore, we can use a KKT point instead of a Fritz-John critical point in the optimality condition~\cite{clarke1983}.  
For problem~\eqref{eqn:opt-ms-simp}, a first-order optimality condition at a local minimum  $\bar{x}$ is that there exists $\bar{\lambda} \in \Rbb^m $ and 
$\bar{\zeta}_l \geq 0$, $\bar{\zeta}_l \in \Rbb^n$, $\bar{\zeta}_u \geq 0$, $\bar{\zeta}_u \in \Rbb^n$ such that 
\begin{equation} \label{eqn:opt-ms-simp-KKT}
   \centering
  \begin{aligned}
	  0 \in  \bar{\partial} r(\bar{x}) + \nabla c (\bar{x}) \bar{\lambda} - \bar{\zeta}_l + \bar{\zeta}_u&, \\
	  \bar{Z}_u(\bar{x}-x_u) = 0, \bar{Z}_l \bar{x} = 0&,\\
	  c_j(\bar{x}) = 0,j=1,\dots,m&,\\
	  \bar{\lambda}_j c_j(\bar{x}) = 0,j=1,\dots,m&,\\
	  \bar{\zeta}_l,\bar{\zeta}_u,x_u-\bar{x},\bar{x}\geq 0&.
  \end{aligned}
\end{equation}
The matrix $\nabla c(\bar{x})$ is of dimension $n\times m$.
The matrices $\bar{Z}_u,\bar{Z}_l$ are diagonal matrices whose diagonal values are $\bar{\zeta}_u$ and $\bar{\zeta}_l$, respectively.
A point that satisfies~\eqref{eqn:opt-ms-simp-KKT} is called a KKT point of problem~\eqref{eqn:opt-ms-simp}.
For upper regular functions, it is possible to establish a stronger form 
of optimality condition, especially the first condition in~\eqref{eqn:opt-ms-simp-KKT} as 
explained in~\cite{mordukhovich2004upp}.

%% file: Sections/Algorithm.tex
\section{\normalsize Simplified bundle algorithm}\label{sec:alg}
Given the nonsmooth nonconvex nature of problem~\eqref{eqn:opt-ms-simp}, we consider the bundle methods which have 
proven to be one of the most successful methods in solving such 
problems~\cite{lemarechal2001}. 
Bundle methods utilize information generated through previous iterative steps to form 
an approximation of the objective. Typically such an approximation model is a supportive one that produces
smaller function value than the real function. Meanwhile, many such algorithms 
rely on the quadratic coefficient in the approximation to avoid line search~\cite{hare2010,noll2013}. 
Another feature is the existence of a robust rejection mechanism to ensure the approximation is reasonable, 
similar to trust-region methods~\cite{Nocedal_book}. A solution to an iterative subproblem generates 
a trial step that is either accepted or rejected. A trial point is called a serious 
point if it is accepted.
Convergence analysis for bundle methods typically require the objective to have properties such as lower-$C^2$ and lower-$C^1$. 
For large-scale problems such as SCACOPF problems, a clear drawback is that
the complex update rule for the bundle and the large number of bundle points needed in the approximation 
could increase computing time considerably. 

The proposed algorithm simplifies the bundle method while retain many of its features.
Motivated by the properties exhibited from two-stage stochastic optimization 
problems discussed in Section~\ref{sec:prob}, we make the assumption that
the objective $r(\cdot)$ is upper-$C^2$, formalized below. 
\begin{assumption}\label{assp:upperC2}
	The Lipschitz continuous objective function $r(\cdot)$ in problem~\eqref{eqn:opt-ms-simp} is upper-$C^2$.
\end{assumption}
\noindent In particular, the inequality~\eqref{eqn:uppc2-def} is satisfied and since $x$ is bounded, there exists a $\rho$ 
for the entire domain.
In general, upper regularity, which is closely related to concavity, is less explored for optimization problems. 
We point out that 
~\cite{noll2009,dao2015} have studied bundle methods and to our knowledge were the first to prove convergence for upper-$C^1$
objective and constraints. 
However, in that case the parameters of the approximation model are not guaranteed to be finite, besides the aforementioned challenges in applying bundle method.
To take full advantage of the smooth constraints $c(\cdot)$, we assume uniform boundedness on their Hessian, a common assumption in literature~\cite{curtis2012}.
\begin{assumption}\label{assp:boundedHc}
	The constraints $c(\cdot)$ are twice differentiable. There exists a constant $H_u^c$ such that the Hessian of constraints $c(\cdot)$ satisfy $\frac{1}{2}x^T \nabla^2c_j(x) x \leq H_u^c\norm{x}^2$ for any $x\in\Rbb^n$ and $1\leq j\leq m$. 
\end{assumption}

\subsection{Algorithm description}\label{sec:alg-1}
The simplified bundle algorithm is an iterative method with approximated objective at each iteration. 
It bears similarity to SQP methods in the treatment of constraints and can be viewed as its extension.
Compared to conventional bundle methods, the convex quadratic approximation $\phi_k(\cdot)$ to the objective $r(\cdot)$ in~\eqref{eqn:opt-ms-simp} 
is dependent only on the current serious point instead of a bundle of points. 
More specifically, at iteration $k$ and its serious step $x_k$,  the local approximation model $\phi_k(\cdot)$ is
\begin{equation} \label{eqn:opt-rc-appx-x}
 \centering
  \begin{aligned}
	  \phi_k(x) = r(x_k) + g_k^T(x-x_k) + \frac{1}{2}\alpha_k \norm{x-x_k}^2,
  \end{aligned}
\end{equation}
where $g_k\in\bar{\partial} r(x_k)$, and $\alpha_k>0$ is a scalar quadratic coefficient.
Equivalently, denoting $d=x-x_k$, $\phi_k(x)$ can be reformulated as $\Phi_k(d)$ such that
\begin{equation} \label{eqn:opt-rc-appx}
 \centering
  \begin{aligned}
	  \Phi_k(d) &= r_k + g_k^T d +\frac{1}{2}\alpha_k\left\lVert d\right\rVert ^2,\\ 
  \end{aligned}
\end{equation}
where $r_k=r(x_k)$. The function 
value and subgradient at $x_k$ are exact, \textit{i.e.}, $\Phi_k(0)=r_k,\nabla\Phi_k(0)=g_k$.
Furthermore, the smooth constraints in~\eqref{eqn:opt-ms-simp} are linearized. The subproblem to be solved 
at iteration $k$ is 
\begin{equation} \label{eqn:opt-ms-simp-bundle}
 \centering
  \begin{aligned}
   &\underset{\substack{d}}{\text{minimize}} 
	  & & \Phi_k(d)\\
   &\text{subject to}
	  & & c(x_k) + \nabla c(x_k)^T d= 0, \\
	  &&& d_l^k \leq d \leq d_u^k,
  \end{aligned}
\end{equation}
where $d_l^k=-x_k,d_u^k=x_u-x_k$.
As in SQP algorithms, it is possible that the linearized constraints 
cause the problem~\eqref{eqn:opt-ms-simp-bundle} to be infeasible. 
There are multiple ways to address this issue, one of which will be presented in Section~\ref{sec:lincons}.
In this section, we operate under the assumption that~\eqref{eqn:opt-ms-simp-bundle} can be solved and 
its solution is denoted as $d_k$.
To measure progress in both the objective and the constraints, 
the $l_1$ merit function is adopted:
\begin{equation} \label{eqn:opt-ms-simp-bundle-merit}
 \centering
  \begin{aligned}
	  \phi_{1\theta_k}(x) = r(x) + \theta_k \norm{c(x)}_1, \\
  \end{aligned}
\end{equation}
where $\norm{\cdot}_1$ is the 1-norm and $\theta_k>0$ is a penalty parameter. 
A line search step on the constraints is needed in order to ensure progress
in the merit function~\eqref{eqn:opt-ms-simp-bundle-merit}.
The predicted change on the objective is defined as 
\begin{equation} \label{def:pd}
 \centering
  \begin{aligned}
	  \delta_k =& \Phi_k(0) - \Phi_k(d_k) 
	   = -g_k^T d_k -\frac{1}{2}\alpha_k\norm{d_k}^2.
  \end{aligned}
\end{equation}
To measure whether the approximation model $\Phi_k(\cdot)$ of the objective formed at $x_k$ is still valid at the trial step $x_k+d_k$,
we define ratio $\rho_k$ as
\begin{equation} \label{eqn:decrease-ratio-1}
 \centering
  \begin{aligned}
	  \rho_k = \begin{cases}
		  r(x_k)-r(x_k+d_k)-\eta_l^+ \delta_k, \ &\delta_k \geq 0,\\
		  r(x_k)-r(x_k+d_k)-\eta_l^- \delta_k, \  &\delta_k <0,
	  \end{cases}
  \end{aligned}
\end{equation}
where $0<\eta_l^+\leq1$ and $\eta_l^-\geq 1$ are two parameters of the algorithm. 
If $\rho_k>0$, the model is valid and the algorithm proceeds to line search.
Otherwise, the trial step $x_k+d_k$ is 
rejected and the parameter $\alpha_k$ is updated to find a different trial step. This process 
draws inspiration from trust-region methods. 

The change in the model objective $\delta_k$ is not necessarily positive. Therefore, the 
corresponding threshold $\eta_l^+$ and $\eta_l^-$ differ based on the sign of $\delta_k$. 
In both cases, the actual change in the objective $r(x_k)-r(x_{k}+d_k)$ is allowed to
be slightly worse than the predicted change. 
This means that if $\delta_k$ is non-negative, the actual decrease can be smaller than the predicted decrease $\delta_k$, 
though a fraction $\eta_l^+$ of  $\delta_k$ is required.
If $\delta_k$ is negative, 
the actual increase in objective value can be slightly larger than the predicted increase value $-\delta_k$, 
the degree to which is governed by $\eta_l^- \geq 1$.

Let the line search parameter be $\beta_k\in (0,1]$. Then, the serious step taken is given as 
$x_{k+1} = x_k + \beta_k d_k$.
Let $\delta_k^{\beta} = \Phi_k(0) - \Phi_k(\beta_k d_{k})$, we have 
\begin{equation} \label{def:pd2}
 \centering
  \begin{aligned}
	  \delta_k^{\beta} =& \Phi_k(0) - \Phi_k(\beta_k d_k)
	   = -\beta_k g_k^T d_k -\frac{1}{2}\beta_k^2 \alpha_k\norm{d_k}^2.
  \end{aligned}
\end{equation}
Similar to $\rho_k$, the ratio between predicted and actual change in objective at $x_{k+1}$ is denoted as $\rho_k^{\beta}$, whose 
definition is 
\begin{equation} \label{eqn:decrease-ratio-beta}
 \centering
  \begin{aligned}
	  \rho_k^{\beta} = \begin{cases}
		  r(x_k)-r(x_{k+1})-\eta_{\gamma}^+ \delta_k^{\beta}, \ & \delta_k^{\beta} \geq 0,\\
		  r(x_k)-r(x_{k+1})-\eta_{\gamma}^- \delta_k^{\beta} , \  & \delta_k^{\beta} <0.
	  \end{cases}
  \end{aligned}
\end{equation}
 The parameter $\eta_{\gamma}^{+}$ and $\eta_{\gamma}^-$ can have different values than $\eta_l^+$ and $\eta_l^-$ to increase  
 flexibility of the algorithm.
The first-order optimality conditions of the subproblem~\eqref{eqn:opt-ms-simp-bundle} are
\begin{equation} \label{eqn:simp-bundle-KKT}
  \centering
   \begin{aligned}
	   g_k + \alpha_k d_k - \nabla c(x_k) \lambda^{k+1} -\zeta_l^{k+1}+\zeta_u^{k+1} &=0,\\
	   Z_u^{k+1}(d_k-d_u^k) = 0, Z_l^{k+1} (d_k-d_l^k) &= 0,\\
	   \Lambda^{k+1}\left[c(x_k)+\nabla c(x_k)^T d_k\right] &= 0,\\
	   \zeta_u^{k+1},\zeta_l^{k+1},d_k -d_l^k,d_u^k-d_k&\geq 0,\\
	   c(x_k)+\nabla c(x_k)^Td_k&= 0,
   \end{aligned}
 \end{equation}
	where $\lambda^{k+1}\in \Rbb^m$ is the Lagrange multiplier for $c(\cdot)$, and $\zeta_u^{k+1},\zeta_l^{k+1}\in \Rbb^n$ are the Lagrange multipliers for the bound constraints. 
	The matrices $\Lambda^{k+1},Z_u^{k+1},Z_l^{k+1}$ are 
	diagonal matrices whose diagonal values are $\lambda^{k+1},\zeta_u^{k+1}$ and $\zeta_l^{k+1}$, respectively.
      An equivalent form of the complementarity conditions 
      of bound constraints based on $x_u$ instead of $d_l^k,d_u^k$ are
\begin{equation} \label{eqn:simp-bundle-KKT-bound}
  \centering
   \begin{aligned}
	   Z_u^{k+1}(x_k+d_k-x_u) = 0, Z_l^{k+1} (x_k+d_k) &= 0,\\
	   \zeta_u^{k+1},\zeta_l^{k+1}, x_k+d_k ,x_u-x_k-d_k&\geq 0.\\
   \end{aligned}
 \end{equation}

\noindent The line search conditions are given as follows
\begin{equation} \label{eqn:line-search-cond}
 \centering
  \begin{aligned}
	  \theta_k \norm{c(x_k)}_1+\beta_k (\lambda^{k+1})^T c(x_k) &\geq \theta_k \norm{c(x_{k+1})}_1-\eta_{\beta}\frac{1}{2}\alpha_k\beta_k\norm{d_k}^2, \\
	  \theta_k \norm{c(x_k)}_1+\eta_{\gamma}^+ \beta_k (\lambda^{k+1})^T c(x_k) &\geq \theta_k \norm{c(x_{k+1})}_1-\eta_{\beta}\frac{1}{2}\alpha_k\beta_k\norm{d_k}^2, \\
	  \theta_k \norm{c(x_k)}_1+\eta_{\gamma}^- \beta_k (\lambda^{k+1})^T c(x_k) &\geq \theta_k \norm{c(x_{k+1})}_1-\eta_{\beta}\frac{1}{2}\alpha_k\beta_k\norm{d_k}^2. \\
  \end{aligned}
\end{equation}
The differences between the conditions are the parameters $\eta_{\gamma}^+$ and $\eta_{\gamma}^-$ in the second and third inequalities, which stem from the unknown sign of $\delta_k$ and $\delta_k^{\beta}$. 
For simplicity in implementation and analysis, we use the following alternative 
condition for line search that encompasses all three
\begin{equation} \label{eqn:line-search-cond-alt}
 \centering
  \begin{aligned}
	  \theta_k \norm{c(x_k)}_1 - \eta_{\gamma}^-  \beta_k\left| (\lambda^{k+1})^T c(x_k)\right| &\geq \theta_k \norm{c(x_{k+1})}_1-\eta_{\beta}\frac{1}{2}\alpha_k\beta_k\norm{d_k}^2. \\
  \end{aligned}
\end{equation}
We will show that condition~\eqref{eqn:line-search-cond-alt} implies conditions in~\eqref{eqn:line-search-cond} in Lemma~\ref{lem:line-search-alt}.
The simplified bundle method is presented in Algorithm~\ref{alg:simp-bundle}, where $\norm{\cdot}_{\infty}$ 
is the infinity norms. The items involving consistency restoration such as $\pi_{k-1}$ are explained in Section~\ref{sec:lincons}.
\begin{algorithm}
 \DontPrintSemicolon
 \SetAlgoNoLine 
   \caption{Simplified bundle method}\label{alg:simp-bundle}
	Initialize $x_0$, $\alpha_0$, stopping error tolerance $\epsilon$, and $k=1$.
	Choose scalars $0<\eta_l^+\leq 1$, $0<\eta_{\beta}<\eta_{\gamma}^+\leq 1$, $\eta_l^-\geq 1$,$\eta_{\gamma}^-\geq 1$, $\eta_{\alpha}>1$ and $\gamma>0$. 
	Evaluate the function value $r(x_0)$ and subgradient $g(x_0)$.\;
	\For{$k=0,1,2,...$}{
	  Form the quadratic function $\Phi_k$ in~\eqref{eqn:opt-rc-appx} and solve 
	  subproblem~\eqref{eqn:opt-ms-simp-bundle} to obtain $d_{k}$ and Lagrange multiplier $\lambda^{k+1}$. (If inconsistent constraints are encountered, enter consistency restoration and go back to step 2 with $k=k+1$.)\;
	  \If{$\norm{d_k}\leq \epsilon$}{
            Stop the iteration and exit the algorithm. \;
	  }
	  Evaluate function value $r(x_k+d_k)$. 
	  Compute $\delta_k$ in~\eqref{def:pd} and $\rho_k$ in~\eqref{eqn:decrease-ratio-1}.\;
	  Set the merit function parameter $\theta_k$ so that $\theta_k = \max{\{\theta_{k-1},\eta_{\gamma}^- \norm{\lambda^{k+1}}_{\infty}+\gamma\}}$. If feasibility restoration is called for iteration $k-1$, let $\theta_k = \max{\{\frac{1}{\pi_{k-1}},\eta_{\gamma}^- \norm{\lambda^{k+1}}_{\infty}+\gamma\}}$.\;
	  \If{$\rho_k > 0$}{
		  Find the line search parameter $\beta_k>0$ using backtracking, starting at $\beta_k=1$ and halving 
		  if too large, such that the conditions in~\eqref{eqn:line-search-cond-alt} are satisfied.
		  Evaluate $r(x_{k+1})$ and compute $\rho_k^{\beta}$ in~\eqref{eqn:decrease-ratio-beta}.\;
		  \If{$\rho_k^{\beta} < 0$}{
		    Break and go to line 14.\;
		  }
		  Take the step $x_{k+1} = x_k+\beta_k d_k$.\;
	  }
	  \Else{
	    Reject the trial step.\;
	    Call the chosen $\alpha_k$ update rules to obtain $\alpha_{k+1}=\eta_{\alpha}\alpha_k$.\;
	  }
	}
\end{algorithm}



\subsection{Convergence analysis}\label{sec:alg-convg}
If the algorithm terminates in a finite number of steps, stopping test in step 4, which can be modified if needed, is satisfied with the error tolerance $\epsilon$ and the solution is considered found. 
Let $\epsilon=0$,  based on step 4, $\norm{d_k}= 0$.
As $d_k$ solves~\eqref{eqn:opt-ms-simp-bundle}, optimality conditions in~\eqref{eqn:simp-bundle-KKT}
are satisfied, of which the first equation reduces to 
\begin{equation}\label{alg:finite-steps}
 \begin{aligned}
	 g_k- \nabla c(x_k) \lambda^{k+1}-\zeta_l^{k+1}+\zeta_u^{k+1}=0.
 \end{aligned}
\end{equation}
Given $g_k\in\bar{\partial} r(x_k)$, we have $0\in \bar{\partial} r(x_k)-\nabla c(x_k) \lambda^{k+1}-\zeta_l^{k+1}+\zeta_u^{k+1}$. 
In addition, by $c(x_k)+\nabla c(x_k)^T d_k = 0$ from~\eqref{eqn:simp-bundle-KKT}, we have $c(x_k)=0$. So $x_k$ is feasible in terms of the equality constraints. 
Together with the bound constraints that are enforced in the subproblem~\eqref{eqn:opt-ms-simp-bundle}, the rest of the equations in~\eqref{eqn:opt-ms-simp-KKT} are also satisfied.
Therefore,  $x_k$ satisfies~\eqref{eqn:opt-ms-simp-KKT} and is a KKT point for~\eqref{eqn:opt-ms-simp} as the algorithm exits.
In what follows, the convergence analysis is carried out for the case with an infinite number of steps, \textit{i.e.}, $\norm{d_k}>0$. 
We start by showing that the parameter $\alpha_k$ in Algorithm~\ref{alg:simp-bundle} eventually stabilizes, \textit{i.e.}, becomes constant.


\begin{lemma}\label{lem:sufficientdecrease}
	Given the assumption~\eqref{assp:upperC2}, Algorithm~\ref{alg:simp-bundle} produces a finite 
	number of rejected steps. 
	As a consequence, the quadratic coefficient $\alpha_k$  is bounded above and remains constant for $k$ large enough. 
\end{lemma}
\begin{proof}
	From the upper-$C^2$ property~\eqref{eqn:uppc2-def}, we have  
\begin{equation} \label{eqn:opt-ms-appx-rec-p1}
 \centering
  \begin{aligned}
	  r(x_k+d)- r_k - g_{k}^Td \leq C \norm{d}^2\\
  \end{aligned}
\end{equation}
for a fixed constant $C>0$.
In the first part of the proof we show that if at some iteration $k$, $\alpha_k$ satisfies
\begin{equation} \label{eqn:opt-ms-appx-rec-p2}
 \centering
  \begin{aligned}
	  \alpha_k > 2 C, \\
  \end{aligned}
\end{equation}
then no rejected steps can occur in Algorithm~\ref{alg:simp-bundle} after iteration $k$. This means steps $8$ and $10$ of the algorithm $\rho_t> 0$ and $\rho_t^\beta> 0$ will hold for all iterations $t\geq k$.
The inequalities~\eqref{eqn:opt-ms-appx-rec-p1} and \eqref{eqn:opt-ms-appx-rec-p2} imply
\begin{equation} \label{eqn:opt-ms-appx-rec-obj}
 \centering
  \begin{aligned}
	  r_k - r(x_k+d_k) \geq& -g_{k}^T d_k -C \norm{d_k}^2\\
                        >& -g_{k}^T d_k -\frac{1}{2}\alpha_k \norm{d_k}^2\\
				 =& \Phi_k(0) -\Phi_k(d_k). \\
  \end{aligned}
\end{equation}
As in the definition~\eqref{eqn:decrease-ratio-1} of $\rho_k$, we distinguish between two cases based on the sign of $\delta_k$.
If $\delta_k = \Phi_k(0) -\Phi_k(d_k) \geq 0$, then since $0<\eta_l^+\leq 1$,~\eqref{eqn:opt-ms-appx-rec-obj} gives
\begin{equation} \label{eqn:opt-ms-appx-rec-obj2}
 \centering
  \begin{aligned}
	  r_k - r(x_k+d_k) >& \Phi_k(0) -\Phi_k(d_k) \\
			\geq & \eta_l^+\left[ \Phi_k(0) -\Phi_k(d_k)\right]. 
  \end{aligned}
\end{equation}
If $\delta_k = \Phi_k(0) -\Phi_k(d_k) < 0$, given $\eta_l^-\geq 1$, we can also write based on~\eqref{eqn:opt-ms-appx-rec-obj} that
\begin{equation} \label{eqn:opt-ms-appx-rec-obj3}
 \centering
  \begin{aligned}
	  r_k - r(x_k+d_k) >& \Phi_k(0) -\Phi_k(d_k), \\
			\geq & \eta_l^-\left[ \Phi_k(0) -\Phi_k(d_k)\right]. 
  \end{aligned}
\end{equation}
As a consequence, by definition~\eqref{eqn:decrease-ratio-1}, we conclude
 $\rho_k > 0$. Similar inequalities hold for  $x_{k+1} = x_k+\beta_k d_k$ since one can write based on~\eqref{eqn:opt-ms-appx-rec-p1} that
\begin{equation} \label{eqn:opt-ms-appx-rec-merit-beta}
 \centering
  \begin{aligned}
	  r(x_k) - r(x_{k+1}) \geq& -\beta_k g_{k}^Td_k - C \beta_k^2 \norm{d_k}^2 
                                 > -\beta_k g_{k}^Td_k -\frac{1}{2}\alpha_k\beta_k^2 \norm{d_k}^2 \\
				 =& \Phi_k(0) -\Phi_k(\beta_k d_k).\\
  \end{aligned}
\end{equation}
Same steps that lead to~\eqref{eqn:opt-ms-appx-rec-obj2} and~\eqref{eqn:opt-ms-appx-rec-obj3} for $\eta_{\gamma}^+,\eta_{\gamma}^-$ imply $\rho_k^{\beta}> 0$. Therefore, for $t\geq k$, $\rho_t>0$ and $\rho_t^{\beta}>0$ and thus $\alpha_t =\alpha_k$. 
Equivalently, no rejected steps occur once~\eqref{eqn:opt-ms-appx-rec-p2} holds.
Since $\alpha_k$ is increased monotonically with a ratio $\eta_{\alpha}>1$ whenever a rejected step is encountered,
only a finite number of rejected steps are needed to reach $\alpha_k > 2 C$, which are followed by serious steps.

For the second part of the proof, suppose now $\alpha_k \leq 2C$ for all $k$. 
Then, only no or a small number of rejected steps can be taken by the algorithm.
The monotonically increasing $\alpha_k$ ensures that there exists $k$ such that $\alpha_t=\alpha_k \leq 2C$  for all $t\geq k$.  This completes the proof. 
\end{proof}

\begin{remark}\label{rmrk:realalpha}
For simplicity, we choose to increase $\alpha_k$ monotonically in the algorithm. In practice, we encourage that 
$\alpha_k$ be reduced if $\rho_k>0$ and $\eta_l^+>\eta_u^+$  where $\eta_u^+$ is an upper threshold for the parameter $\eta_l^+$.
In other words, if the actual decrease in objective is bigger than a certain ratio of the predicted decrease, 
	then $\Phi_k(\cdot)$ is a good approximation and we reduce the quadratic coefficient to encourage larger step size.
From the convergence analysis point of view, the upper-$C^2$ constant $C$ is not uniform in the entire domain. 
A decrease in $\alpha_k$ allows the algorithm to adjust better to  
the local upper-$C^2$ constant that could be relatively small compared to $C$, which could result in improved convergence in practice.
\end{remark}

\begin{lemma}\label{lem:line-search-merit}
	Given Assumption~\ref{assp:boundedHc}, the line search process of Algorithm~\ref{alg:simp-bundle} is well-defined, in that a $\beta_k\in (0,1]$ that satisfies the line search conditions in~\eqref{eqn:line-search-cond-alt} exists and can be found in a finite number of steps through backtracking step 9 as long as the 
	Lagrange multipliers $\lambda^{k+1}$ from~\eqref{eqn:opt-ms-simp-bundle} remain finite.
\end{lemma}
\begin{proof}
If $\lambda^{k+1}$ remains finite throughout the algorithm, then a finite $\theta_k$ is guaranteed as well based on how it is chosen in Algorithm~\ref{alg:simp-bundle} step 7 as it stops increasing for $k$ large enough.
	Since $c(\cdot)$ is smooth, we apply Taylor expansion to the $j$th equality constraint, $j=1,...,m$, at $x_k$ for $x_{k+1}=x_k+\beta_k d_k$ to obtain 
   \begin{equation} \label{eqn:simp-bundle-c-ls-pf-1}
   \centering
    \begin{aligned}
	    c_j(x_{k+1})  = & c_j(x_k)+ \beta_k \nabla c_j(x_k)^T d_{k} + \frac{1}{2}\beta_k^2 d_k^T H_{k\beta}^{j} d_k,
   \end{aligned}
   \end{equation}
	where $H_{k\beta}^j$ is the Hessian $\nabla^2 c_j(\cdot)$ at a point on the line segment determined by $x_k$ and $x_{k+1}$. 
Given $d_k$ as the solution to~\eqref{eqn:opt-ms-simp-bundle}, we have that $c_j(x_k) +  \nabla c_j(x_k)^T d_k = 0$ and as a consequence, we can write based on~\eqref{eqn:simp-bundle-c-ls-pf-1} that
   \begin{equation*} 
   \centering
    \begin{aligned}
	    c_j(x_{k+1}) = (1-\beta_k) c_j(x_k)+ \frac{1}{2}\beta_k^2 d_k^T H_{k\beta}^{j} d_k.\\
    \end{aligned}
   \end{equation*}
By Assumption~\ref{assp:boundedHc}, $\left|c_j(x_{k+1})\right|  \leq \left|(1-\beta_k) c_j(x_k)\right| + \beta_k^2 H^{c}_u \norm{d_k}^2$, which in turn implies that
   \begin{equation} \label{eqn:simp-bundle-c-ls-pf-2.5}
   \centering
    \begin{aligned}
	    \norm{c(x_{k+1})}_1  \leq 
	     (1-\beta_k) \norm{c(x_k)}_1 + m\beta_k^2 H^{c}_u \norm{d_k}^2.
    \end{aligned}
   \end{equation}
Applying simple norm inequalities, we have  
\begin{equation} \label{eqn:simp-bundle-c-ls-pf-norm}
   \centering
    \begin{aligned}
	    \beta_k \left|(\lambda^{k+1})^T c(x_k)\right| \leq \beta_k\norm{\lambda^{k+1}}_{\infty} \norm{c(x_k)}_1.
    \end{aligned}
   \end{equation}
	Since step 7 of the algorithm chooses $\theta_k \geq \eta_{\gamma}^- \norm{\lambda^{k+1}}_{\infty}+\gamma$, where $\eta_{\gamma}^-$ and $\gamma$ are positive
	constants, we can write based on~\eqref{eqn:simp-bundle-c-ls-pf-2.5} and~\eqref{eqn:simp-bundle-c-ls-pf-norm} that
   \begin{equation} \label{eqn:simp-bundle-c-ls-pf-3} 
   \centering
    \begin{aligned}
	    &\theta_k\norm{c(x_k)}_1 -\eta_{\gamma}^- \beta_k \left|(\lambda^{k+1})^T c(x_k)\right|-\theta_k \norm{c(x_{k+1})}_1  \\ 
	    &\hspace{1.0cm}\geq(\theta_k - \eta_{\gamma}^- \beta_k\norm{\lambda^{k+1}}_{\infty}) \norm{c(x_k)}_1 -\theta_k(1-\beta_k) \norm{c(x_k)}_1 - \theta_km\beta_k^2 H^{c}_u \norm{d_k}^2\\
		&\hspace{1.0cm}=(\theta_k\beta_k- \eta_{\gamma}^- \beta_k\norm{\lambda^{k+1}}_{\infty}) \norm{c(x_k)}_1- \theta_km H^c_u \beta_k^2 \norm{d_k}^2 \\
		&\hspace{1.0cm}\geq\beta_k\left( \gamma \norm{c(x_k)}_1- \theta_k m H^c_u\beta_k\norm{d_k}^2\right).
    \end{aligned}
   \end{equation}
Therefore, if $\beta_k$ is reduced by the line search step through backtracking in Algorithm~\ref{alg:simp-bundle} to satisfy
   \begin{equation} \label{eqn:simp-bundle-c-ls-pf-4}
   \centering
    \begin{aligned}
	   0< \beta_k \leq \frac{\eta_{\beta} \alpha_k}{2 H^c_u\theta_k m},
    \end{aligned}
   \end{equation}
then
   \begin{equation*} 
   \centering
    \begin{aligned}
	    \theta_k \norm{c(x_k)}_1-\eta_{\gamma}^- \beta_k\left|(\lambda^{k+1})^T c(x_k)\right|-\theta_k \norm{c(x_{k+1})}_1 \geq  - \frac{1}{2}\eta_{\beta}\alpha_k\beta_k\norm{d_k}^2. \\
    \end{aligned}
   \end{equation*}
Using ceiling function $\lceil\cdot \rceil$, which returns the least integer greater than the input, we can write 
   \begin{equation} \label{eqn:simp-bundle-c-ls-pf-5}
   \centering
    \begin{aligned}
	    \beta_k \geq \frac{1}{2} ^{\lceil \log_{\frac{1}{2}} \frac{\eta_{\beta} \alpha_k}{2 H^c_u\theta_k m} \rceil}.
    \end{aligned}
   \end{equation}
We remark that both the denominator and numerator in~\eqref{eqn:simp-bundle-c-ls-pf-4} are positive and independent of the line search. 
Further, by Lemma~\ref{lem:sufficientdecrease} and finite $\lambda^{k+1}$, all terms in~\eqref{eqn:simp-bundle-c-ls-pf-4} remain finite. Therefore, the backtracking stops in finite steps. This completes the proof.
   \end{proof}

\begin{lemma}\label{lem:line-search-alt}
	The $\beta_k\in (0,1]$ that meets the line search condition in~\eqref{eqn:line-search-cond-alt} also satisfies the conditions 
	from~\eqref{eqn:line-search-cond}, or equivalently
\begin{equation} \label{eqn:simp-bundle-etagamma}
   \centering
    \begin{aligned}
	   \beta_k(\lambda^{k+1})^T c(x_k)\geq -\theta_k \norm{c(x_k)}_1+ \theta_k \norm{c(x_{k+1})}_1 - \eta_{\beta}\frac{1}{2}\alpha_k\beta_k\norm{d_k}^2,\\
	  \eta_{\gamma}^+\beta_k(\lambda^{k+1})^T c(x_k) \geq  -\theta_k \norm{c(x_k)}_1+ \theta_k \norm{c(x_{k+1})}_1   - \eta_{\beta}\frac{1}{2}\alpha_k\beta_k\norm{d_k}^2,\\
	 \eta_{\gamma}^-\beta_k(\lambda^{k+1})^T c(x_k) \geq  -\theta_k \norm{c(x_k)}_1+ \theta_k \norm{c(x_{k+1})}_1 - \eta_{\beta}\frac{1}{2}\alpha_k\beta_k\norm{d_k}^2.\\
    \end{aligned}
   \end{equation}
\end{lemma}
\begin{proof}
   From simple absolute value inequality, we have  
   \begin{equation} \label{eqn:line-search-alt-pf-1}
   \centering
    \begin{aligned}
	    \beta_k(\lambda^{k+1})^T c(x_k) \geq& -\beta_k\left|(\lambda^{k+1})^T c(x_k)\right|,\\ 
	     \eta_{\gamma}^+ \beta_k(\lambda^{k+1})^T c(x_k) \geq& -\eta_{\gamma}^+ \beta_k\left|(\lambda^{k+1})^T c(x_k)\right|,\\
	     \eta_{\gamma}^- \beta_k(\lambda^{k+1})^T c(x_k) \geq& -\eta_{\gamma}^-\beta_k\left|(\lambda^{k+1})^T c(x_k)\right|.\\
    \end{aligned}
   \end{equation}
	Given that $0<\eta_{\gamma}^{+}\leq 1\leq \eta_{\gamma}^-$, 
   \begin{equation} \label{eqn:line-search-alt-pf-2}
   \centering
    \begin{aligned}
            -\eta_{\gamma}^-\beta_k\left|(\lambda^{k+1})^T c(x_k)\right|\leq  -\beta_k\left|(\lambda^{k+1})^T c(x_k)\right| \leq -\eta_{\gamma}^+ \beta_k\left|(\lambda^{k+1})^T c(x_k)\right|.
    \end{aligned}
   \end{equation}
	Therefore, from~\eqref{eqn:line-search-alt-pf-1} and~\eqref{eqn:line-search-alt-pf-2}
   \begin{equation} \label{eqn:line-search-alt-pf-3}
   \centering
    \begin{aligned}
	    \beta_k(\lambda^{k+1})^T c(x_k) \geq& -\eta_{\gamma}^- \beta_k\left|(\lambda^{k+1})^T c(x_k)\right|,\\ 
	     \eta_{\gamma}^+ \beta_k(\lambda^{k+1})^T c(x_k) \geq& -\eta_{\gamma}^- \beta_k\left|(\lambda^{k+1})^T c(x_k)\right|,\\
	     \eta_{\gamma}^- \beta_k(\lambda^{k+1})^T c(x_k) \geq& -\eta_{\gamma}^-\beta_k\left|(\lambda^{k+1})^T c(x_k)\right|.\\
    \end{aligned}
   \end{equation}
   From the line search condition~\eqref{eqn:line-search-cond-alt}, we have
   \begin{equation} \label{eqn:line-search-alt-pf-4}
   \centering
    \begin{aligned}
	-\eta_{\gamma}^-\beta_k\left|(\lambda^{k+1})^T c(x_k)\right| \geq -\theta_k \norm{c(x_k)}_1+ \theta_k \norm{c(x_{k+1})}_1 - \eta_{\beta}\frac{1}{2}\alpha_k\beta_k\norm{d_k}^2.\\
    \end{aligned}
   \end{equation}
   Combined with~\eqref{eqn:line-search-alt-pf-3} the proof is completed.
\end{proof}

\begin{lemma}\label{lem:sufficientdecrease-merit}
	The step $x_{k+1} = x_k+\beta_k d_k$ is a decreasing step for the merit function~\eqref{eqn:opt-ms-simp-bundle-merit} if $\beta_k$ satisfies the line search condition. Further, if the Lagrange multipliers $\lambda^k$ is finite for all $k$, the speed of decrease satisfies $\phi_{1\theta_k}(x_k)-\phi_{1\theta_k}(x_{k+1})>c_{\phi} \norm{d_k}^2$ for some constant $c_{\phi}$. 

\end{lemma}
\begin{proof}
	For a serious step $x_{k+1}$ to be taken, step 8 and 10 are satisfied so that $\rho_k>0,\rho_k^{\beta}>0$. 
	We distinguish three cases based on the value of $\alpha_k$ and sign of $\delta_k^{\beta}$. 
	The first case is $\alpha_k>2C$. By upper-$C^2$ property in~\eqref{eqn:opt-ms-appx-rec-p1}, as shown in~\eqref{eqn:opt-ms-appx-rec-merit-beta}, we have  
\begin{equation*} 
 \centering
  \begin{aligned}
	  r(x_k) - r(x_{k+1}) >& -\beta_k g_{k}^Td_k -\frac{1}{2}\alpha_k\beta_k^2 \norm{d_k}^2 \\
				 =& \Phi_k(0) -\Phi_k(\beta_kd_k) = \delta_k^{\beta}.\\
  \end{aligned}
\end{equation*}
	In the second case, 
	$\alpha_k\leq 2C$ and $\delta_k^{\beta}\geq 0$. 
	From the definition of $\rho_k^{\beta}$ in~\eqref{eqn:decrease-ratio-beta},  
\begin{equation} \label{eqn:simp-bundle-merit-pf-2p}
 \centering
  \begin{aligned}
	  r(x_k) - r(x_{k+1}) > 
		  \eta_{\gamma}^+ \left[ -\beta_k g_{k}^Td_k -\frac{1}{2}\alpha_k\beta_k^2 \norm{d_k}^2\right]. \\
  \end{aligned}
\end{equation}
       The third case is when $\alpha_k\leq 2C$ and $\delta_k^{\beta}<0$, and we have 
\begin{equation} \label{eqn:simp-bundle-merit-pf-2m}
 \centering
  \begin{aligned}
	  r(x_k) - r(x_{k+1}) > 
		  \eta_{\gamma}^- \left[ -\beta_k g_{k}^Td_k -\frac{1}{2}\alpha_k\beta_k^2 \norm{d_k}^2\right]. \\
  \end{aligned}
\end{equation}
		Rearranging the first equation in optimality condition~\eqref{eqn:simp-bundle-KKT}, we have  
\begin{equation} \label{eqn:simp-bundle-KKT-beta}
  \centering
   \begin{aligned}
	   g_k + \alpha_k d_k =  \nabla c(x_k) \lambda^{k+1} +\zeta_l^{k+1} -\zeta_u^{k+1}.\\
   \end{aligned}
 \end{equation}
Then, taking the inner product with $-d_k$ and using the last equation from~\eqref{eqn:simp-bundle-KKT} we have 
\begin{equation} \label{eqn:simp-bundle-KKT-2}
  \centering
   \begin{aligned}
	   - g_k^T d_k-\alpha_k\norm{d_k}^2 &=- (\lambda^{k+1})^T \nabla c(x_k)^T d_k-d_k^T \zeta_l^{k+1}+d_k^T \zeta_u^{k+1}\\
			    &= (\lambda^{k+1})^T c(x_k) - (d_k - d_l^k +d_l^k)^T\zeta_l^{k+1} + (d_k-d_u^k+d_u^k)^T \zeta_u^{k+1} \\
			    &= (\lambda^{k+1})^T c(x_k) - (d_l^k)^T \zeta_l^{k+1}  +(d_u^k)^T \zeta_u^{k+1} \\
			    &= (\lambda^{k+1})^T c(x_k) + x_k^T \zeta_l^{k+1} + (x_u-x_k)^T\zeta_u^{k+1} \\
			    &\geq (\lambda^{k+1})^T c(x_k). \\
   \end{aligned}
 \end{equation}
	The third equality of~\eqref{eqn:simp-bundle-KKT-2} comes from the complementarity conditions $Z_{l}^{k+1}(d_k-d_l^k)=0$ and $Z_u^{k+1}(d_k-d_u^k)=0$ in~\eqref{eqn:simp-bundle-KKT}. 
	The inequality can be obtained from bound constraints in~\eqref{eqn:simp-bundle-KKT-bound} 
	where $x_k\geq 0$, $x_u-x_k\geq 0$, $\zeta_l^{k+1}\geq 0$ and $\zeta_u^{k+1}\geq 0$
	for the current and previous iteration. 
	Next, multiplying both sides of~\eqref{eqn:simp-bundle-KKT-2} by $\beta_k$ and then subtracting $\frac{1}{2}\alpha_k\beta_k^2\norm{d_k}^2$ leads to
  \begin{equation} \label{eqn:simp-bundle-KKT-3}
       \centering
       \begin{aligned}
	       -\beta_kg_k^T d_k-\frac{1}{2}\alpha_k\beta_k^2\norm{d_k}^2 &\geq \alpha_k\beta_k\norm{d_k}^2-\frac{1}{2}\alpha_k\beta_k^2\norm{d_k}^2+\beta_k(\lambda^{k+1})^Tc(x_k)\\
		 &\geq\frac{1}{2}\alpha_k\beta_k\norm{d_k}^2+\beta_k(\lambda^{k+1})^Tc(x_k),\\
         \end{aligned}
        \end{equation}
	where the second inequality makes use of $\beta_k\in (0, 1]$.
Notice that the left-hand side of~\eqref{eqn:simp-bundle-KKT-3} is $\delta_k^{\beta}$ and is not guaranteed to be positive. 
Multiplying both sides of~\eqref{eqn:simp-bundle-KKT-3} by $\eta_{\gamma}^+$ and $\eta_{\gamma}^-$ respectively, we obtain
  \begin{equation} \label{eqn:simp-bundle-KKT-4}
       \centering
       \begin{aligned}
	       -\eta_{\gamma}^+ \beta_kg_k^T d_k-\frac{1}{2}\eta_{\gamma}^+\alpha_k\beta_k^2\norm{d_k}^2 
		 &\geq\frac{1}{2}\eta_{\gamma}^+\alpha_k\beta_k\norm{d_k}^2+\eta_{\gamma}^+\beta_k(\lambda^{k+1})^Tc(x_k),\\
	       -\eta_{\gamma}^- \beta_kg_k^T d_k-\frac{1}{2}\eta_{\gamma}^-\alpha_k\beta_k^2\norm{d_k}^2 
		 &\geq\frac{1}{2}\eta_{\gamma}^-\alpha_k\beta_k\norm{d_k}^2+\eta_{\gamma}^-\beta_k(\lambda^{k+1})^Tc(x_k).\\
         \end{aligned}
        \end{equation}
Finally, we can examine the merit function $\phi_{1\theta_k}(\cdot)$. If $\alpha_k>2C$, 
combine the inequality in~\eqref{eqn:opt-ms-appx-rec-merit-beta},~\eqref{eqn:simp-bundle-KKT-3} and the first inequality from Lemma~\ref{lem:line-search-alt}, we have  
\begin{equation} \label{eqn:simp-bundle-merit-pf-3}
 \centering
  \begin{aligned}
	  \phi_{1\theta_k}(x_k) - \phi_{1\theta_k}(x_{k+1})  =& r(x_k) -r(x_{k+1})+\theta_k\norm{c(x_k)}_1-\theta_k\norm{c(x_{k+1})}_1\\
			 >& -\beta_kg_k^T d_k-\frac{1}{2}\alpha_k\beta_k^2\norm{d_k}^2+\theta_k\norm{c(x_k)}_1-\theta_k\norm{c(x_{k+1})}_1 \\
			 \geq&\frac{1}{2}\alpha_k\beta_k\norm{d_k}^2+\beta_k(\lambda^{k+1})^Tc(x_k)+\theta_k\norm{c(x_k)}_1-\theta_k\norm{c(x_{k+1})}_1\\
			 \geq&\frac{1}{2}\alpha_k\beta_k\norm{d_k}^2-\eta_{\beta}\beta_k\frac{1}{2}\alpha_k\norm{d_k}^2\\
			 =&(1-\eta_{\beta})\frac{1}{2}\alpha_k\beta_k\norm{d_k}^2.
  \end{aligned}
\end{equation}
Otherwise, with $\alpha_k\leq 2C$ and $\delta_k^{\beta}\geq 0$, we apply in order~\eqref{eqn:simp-bundle-merit-pf-2p},~\eqref{eqn:simp-bundle-KKT-4} and the second inequality from Lemma~\ref{lem:line-search-alt} to obtain
\begin{equation} \label{eqn:simp-bundle-merit-pf-4}
 \centering
  \begin{aligned}
	  \phi_{1\theta_k}(x_k) - \phi_{1\theta_k}(x_{k+1})  =& r(x_k) -r(x_{k+1})+\theta_k\norm{c(x_k)}_1-\theta_k\norm{c(x_{k+1})}_1\\
			 >& -\eta_{\gamma}^+\beta_kg_k^T d_k-\eta_{\gamma}^+\frac{1}{2}\alpha_k\beta_k^2\norm{d_k}^2+\theta_k\norm{c(x_k)}_1-\theta_k\norm{c(x_{k+1})}_1 \\
			 \geq&\frac{1}{2}\eta_{\gamma}^+\alpha_k\beta_k\norm{d_k}^2+\eta_{\gamma}^+\beta_k(\lambda^{k+1})^Tc(x_k)+\theta_k\norm{c(x_k)}_1-\theta_k\norm{c(x_{k+1})}_1\\
			 \geq&\frac{1}{2}\eta_{\gamma}^+\alpha_k\beta_k\norm{d_k}^2-\eta_{\beta}\beta_k\frac{1}{2}\alpha_k\norm{d_k}^2\\
			 =&(\eta_{\gamma}^+-\eta_{\beta})\frac{1}{2}\alpha_k\beta_k\norm{d_k}^2,
  \end{aligned}
\end{equation}
with $\eta_{\gamma}^+-\eta_{\beta}>0$. Similarly, when $\delta_k^{\beta}\leq 0$, applying
in order~\eqref{eqn:simp-bundle-merit-pf-2m},~\eqref{eqn:simp-bundle-KKT-4} and the third inequality from Lemma~\ref{lem:line-search-alt}, we have 
\begin{equation} \label{eqn:simp-bundle-merit-pf-5}
 \centering
  \begin{aligned}
	  \phi_{1\theta_k}(x_k) - \phi_{1\theta_k}(x_{k+1})  =& r(x_k) -r(x_{k+1}^{\beta})+\theta_k\norm{c(x_k)}_1-\theta_k\norm{c(x_{k+1})}_1\\
			 >& -\eta_{\gamma}^-\beta_kg_k^T d_k-\eta_{\gamma}^-\frac{1}{2}\alpha_k\beta_k^2\norm{d_k}^2+\theta_k\norm{c(x_k)}_1-\theta_k\norm{c(x_{k+1})}_1 \\
			 \geq&\frac{1}{2}\eta_{\gamma}^-\alpha_k\beta_k\norm{d_k}^2+\eta_{\gamma}^-\beta_k(\lambda^{k+1})^Tc(x_k)+\theta_k\norm{c(x_k)}_1-\theta_k\norm{c(x_{k+1})}_1\\
			 \geq&\frac{1}{2}\eta_{\gamma}^-\alpha_k\beta_k\norm{d_k}^2-\eta_{\beta}\beta_k\frac{1}{2}\alpha_k\norm{d_k}^2\\
			 =&(\eta_{\gamma}^- -\eta_{\beta})\frac{1}{2}\alpha_k\beta_k\norm{d_k}^2,
  \end{aligned}
\end{equation}
where $\eta_{\gamma}^- -\eta_{\beta}>0$.

Therefore, in all cases, a serious step $x_{k+1}=x_k+\beta_kd_k$ is a 
decreasing direction for the merit function $\phi_{1\theta_k}(\cdot)$.
If $\lambda^{k}$ is finite for all $k$, $\theta_k$ will stay constant for $k$ large enough from step 7 of Algorithm~\ref{alg:simp-bundle}. Let $\bar{\theta}$ be the constant value for $k$ large enough so that $\theta_k \leq \bar{\theta}$ for all $k$.  
Then, by~\eqref{eqn:simp-bundle-c-ls-pf-5},
\begin{equation} \label{eqn:simp-bundle-merit-pf-6}
 \centering
  \begin{aligned}
	  \beta_k \geq \frac{1}{2} ^{\lceil \log_{\frac{1}{2}} \frac{\eta_{\beta} \alpha_0}{2 H^c_u\bar{\theta} m} \rceil}:=\bar{\beta},
  \end{aligned}
\end{equation}
due to the monotonicity of $\alpha_k$ and $\theta_k$. In other words, $\beta_k$ is bounded below by $\bar{\beta}$ for all $k$.
From~\eqref{eqn:simp-bundle-merit-pf-3},~\eqref{eqn:simp-bundle-merit-pf-4},~\eqref{eqn:simp-bundle-merit-pf-5}, $\phi_{1\theta_k}(x_k)-\phi_{1\theta_k}(x_{k+1})>(\eta_{\gamma}^+-\eta_{\beta})\frac{1}{2}\alpha_0 \bar{\beta}\norm{d_k}^2$. 
Or simply, there exists $c_{\phi}$ such that $\phi_{1\theta_k}(x_k)-\phi_{1\theta_k}(x_{k+1})>c_{\phi}\norm{d_k}^2$. 
\end{proof}
In order to obtain a stabilized $\lambda^k$, a constraint qualification is necessary 
for the constraints $c(x)=0$ in~\eqref{eqn:opt-ms-simp}.
In Section~\ref{sec:prob}, we discussed calmness as the weak constraint qualification that 
would ensure a KKT point instead of a Fritz-John critical point in our nonsmooth upper-$C^2$ setup.
Here, we resort to the stronger LICQ~\cite{Nocedal_book} to prove stabilization of Lagrange multipliers for our proposed algorithm. A topic of further research will be to derive the results of this section under a weak constraint qualification such as calmness. 

\begin{lemma}\label{lem:bounded-lp}
	If LICQ of the constraints in~\eqref{eqn:opt-ms-simp} are satisfied at every accumulation points $\bar{x}$ of serious steps $\{x_k\}$ generated by the algorithm, then 
	the sequence of Lagrange multipliers for the solutions to problem~\eqref{eqn:opt-ms-simp-bundle} $\{\zeta_u^{k+1}\},\{\zeta_l^{k+1}\}$ and $\{\lambda^{k+1}\}$ are bounded. Thus, there exists $k$, such that 
	$\norm{\lambda^{t}}_{\infty} \leq \lambda^U$, $\zeta^{t}_u \leq \zeta_u^U$ and $\zeta^{t}_l \leq \zeta_l^U$ for all $t\geq k$, where $\lambda^U>0,\zeta_u^U>0,\zeta_l^U>0$ are
	the upper bounds. Further, this means there exists $\bar{\theta}$ such that $\theta_t = \bar{\theta}$ for all $t\geq k$.
\end{lemma}
\begin{proof}
     We rewrite the first equation in optimality condition in~\eqref{eqn:simp-bundle-KKT} as
    \begin{equation} \label{eqn:simp-bundle-KKT-full}
     \centering
     \begin{aligned}
	     g_k + \alpha_k d_k - \sum_{j=1}^m \lambda^{k+1}_j \nabla c_j(x_k) -\sum_{i=1}^n (\zeta_l^{k+1})_i e_i 
	     +\sum_{i=1}^n (\zeta_u^{k+1})_i e_i &=0,\\
     \end{aligned}
    \end{equation}
    where $e_i \in \Rbb^n$ is a vector such that $(e_i)_i = 1$ and $(e_i)_k = 0,k\neq i$. 
    Since $(\zeta_l^{k+1})_i(\zeta_u^{k+1})_i = 0$,  the bound constraints Lagrange multipliers are combined into  $\zeta^{k+1}=\zeta_l^{k+1}- \zeta_u^{k+1}$. A component of $\zeta_l^{k+1}$ or $\zeta_u^{k+1}$ is unbounded if and only if the corresponding component in $\zeta^{k+1}$ is unbounded.
	Let $I$ be the index set of the active bound constraints, hence
    \begin{equation} \label{eqn:simp-bundle-KKT-full-2}
     \centering
     \begin{aligned}
	     g_k + \alpha_k d_k = \sum_{j=1}^m \lambda^{k+1}_j \nabla c_j(x_k) +\sum_{i \in I} (\zeta^{k+1})_i e_i.\\
     \end{aligned}
    \end{equation}
	Since $\{x_k\},\{g_k\}$ are bounded ($r(\cdot)$ being Lipschitz continuous on a bounded domain) and $\{\alpha_k\}$ is finite by Lemma~\ref{lem:sufficientdecrease}, 
	the left-hand side of the equation stays bounded throughout the iterations. 
	From LICQ at $\bar{x}$, we know that $\nabla c_j(\bar{x}) \in \Rbb^n $ and $e_i,i \in I$ are linearly independent and bounded vectors. 
        Without losing generality, suppose $\lambda^{k+1}_j,j\in[1,m]$ is not 
	bounded as $k\to\infty$. Then, we have $\norm{\lambda^{k}}_{\infty} \to \infty$. Passing on to a subsequence if necessary, we can assume $x_k\to\bar{x}$ as $k\to\infty$, where $\bar{x}$ is an accumulation point. Regardless of the behavior of $\{\zeta^{k+1}\}$, the right-hand side of~\eqref{eqn:simp-bundle-KKT-full-2} will be unbounded due to linear independence among the vectors.
	This is a contradiction. Same process can be repeated for $\zeta^{k+1}_j,j\in[1,m]$. 

	Therefore,  there exist $\lambda^U, \zeta_l^U\geq 0,\zeta_u^U\geq 0$ such that 
	$\norm{\lambda^{t}}_{\infty} \leq \lambda^U$, $\zeta^{t}_u \leq \zeta_u^U$ and $\zeta^{t}_l \leq \zeta_l^U$ for all $k$.
	Since $\theta_k$ is determined by $\lambda^k$ (step 7 in Algorithm~\ref{alg:simp-bundle}), there exists $k$ and $\bar{\theta}$ such that $\theta_t=\bar{\theta}$ for all $t\geq k$. 
\end{proof}

\begin{theorem}\label{thm:simp-KKT}
	Given the Assumptions~\ref{assp:upperC2} and~\ref{assp:boundedHc}, if the constraints in~\eqref{eqn:opt-ms-simp} satisfy the 
	conditions in Lemma~\ref{lem:bounded-lp}, then every accumulation point of the solution steps $\{x_k\}$ generated from Algorithm~\ref{alg:simp-bundle} 
	is a KKT point of the problem~\eqref{eqn:opt-ms-simp}. 
	That is, there exists a subsequence of $\{x_k\}$ that converges to $\bar{x}$, and $\bar{\lambda}\in\Rbb^m$, $\bar{\zeta}_u\in\Rbb^n$, $\bar{\zeta}_l\in\Rbb^n$ such that the first-order optimality conditions are satisfied at $\bar{x}$
\begin{equation} \label{eqn:simp-bundle-KKT-limit}
 \centering
  \begin{aligned}
	  0 \in \bar{\partial} r(\bar{x}) - \nabla c(\bar{x}) \bar{\lambda}-\bar{\zeta}_l +\bar{\zeta}_u&, \\
	  \bar{Z}_l\bar{x} = 0,\  \bar{Z}_u(\bar{x}-x_u) = 0&,\\
	  c(\bar{x}) = 0&,\\
	  \bar{\zeta}_l,\bar{\zeta}_u,\bar{x}, x_u-\bar{x} \geq 0&.\\
  \end{aligned}
\end{equation}
\end{theorem}
\begin{proof}
	By Lemma~\ref{lem:sufficientdecrease}, there exists $k_0>0$ such that for all $t>k_0$, $\alpha_t=\alpha_{k_0}=\bar{\alpha}$ and 
	all following steps are serious steps.
	By Lemma~\ref{lem:bounded-lp}, there exists $k_1>0$ such that for $t>k_1$, the Lagrange multipliers are bounded above and $\theta_t=\theta_{k_1}=\bar{\theta}$. 
	We say $k$ is large enough if $k\geq \max{(k_0,k_1)}$, in which case the parameters of the algorithm stabilizes at $\alpha_t=\bar{\alpha}$ and $\theta_t=\bar{\theta}$ for $t\geq k$.

	Since the domain of $x$ is bounded and $r(\cdot)$ is Lipschitz, the serious steps sequence $\{x_k\}$ 
	as well as the subgradient sequence $\{g_{k}\}$ are bounded. 
	Therefore, there exists at least one accumulation point for $\{x_k\}$.
	Let $\bar{x}$ be an accumulation point of $\{x_k\}$ and $\{x_{k_s}\}$ be a subsequence of $\{x_k\}$ such that $x_{k_s}\to \bar{x}$.

	From Lemma~\ref{lem:line-search-merit}, line search terminates successfully and  
	by Lemma~\ref{lem:sufficientdecrease-merit}, for $k$ large enough, $\{\phi_{1\theta_k}(x_k)\}$ is a decreasing and bounded sequence 
	with a fixed parameter $\bar{\theta}$. 
	Thus, $\phi_{1\theta_k}(x_k)$ converges. Let $\lim_{k\to\infty} \phi_{1\theta_k}(x_k)\to \bar{\phi}_{1\bar{\theta}}$, \textit{i.e.}, $\lim_{k\to\infty}r(x_k)+\bar{\theta}\norm{c(x_k)}_1 \to \bar{\phi}_{1\bar{\theta}}$. 
	From the proof of Lemma~\ref{lem:sufficientdecrease-merit},~\eqref{eqn:simp-bundle-merit-pf-3},~\eqref{eqn:simp-bundle-merit-pf-4} and~\eqref{eqn:simp-bundle-merit-pf-5}, we know that $\phi_{1\theta_k}(x_k)-\phi_{1\theta_k}(x_{k+1})$ is 
	bounded below in the order of $\norm{d_k}^2$. Therefore,  $\lim_{k\to\infty} \norm{d_k}\to 0$. In particular, $\lim_{s\to \infty}\norm{d_{k_s}}\to 0$.
	By the last equation in~\eqref{eqn:simp-bundle-KKT}, $c(x_{k_s})\to 0$. Thus, $\bar{x}$ satisfies the equality constraints $c(\cdot)$.
        Given that the bound constraints are satisfied by all $x_k$, $0\leq \bar{x}\leq x_u$.

	Passing on to a subsequence if necessary, we let $g_{k_s}\to \bar{g}$, 
	$\lambda_{k_s} \to \bar{\lambda}$, $\zeta_u^{k_s} \to \bar{\zeta}_u$, $\zeta_l^{k_s} \to \bar{\zeta}_l$. 
        From the first equation in the optimality conditions~\eqref{eqn:simp-bundle-KKT}, we have 
        \begin{equation} \label{eqn:simp-bundle-KKT-limit-1}
          \centering
          \begin{aligned}
		  0 = \bar{g} - \nabla c(\bar{x}) \bar{\lambda}  -\bar{\zeta}_l +\bar{\zeta}_u. \\
          \end{aligned}
        \end{equation}
	By the outer semicontinuity of Clarke subdifferential, with $g_{k_s}\in \bar{\partial} r(x_{k_s})$, we have $\bar{g} \in \bar{\partial} r(\bar{x})$.
        As a result, 
	  $0 \in \bar{\partial} r(\bar{x}) -\nabla c(\bar{x}) \bar{\lambda} -\bar{\zeta}_l +\bar{\zeta}_u$.
	The complementarity conditions of bound constraints from~\eqref{eqn:simp-bundle-KKT} leads 
	to $\bar{Z}_u(\bar{x}-x_u)$, $\bar{Z}_l\bar{x}=0$. Together with the 
	equality constraints $c(\bar{x}) = 0$, the first-order
	optimality conditions~\eqref{eqn:simp-bundle-KKT-limit} of problem~\eqref{eqn:opt-ms-simp} at $\bar{x}$ are satisfied.

\end{proof}

While the line search is only conducted on the less computationally expensive and analytically known 
smooth constraints $c(\cdot)$, it is 
possible to avoid it altogether. In bundle methods, it has been shown that a convex feasible set for $x$ can make the 
algorithm converge without line search~\cite{hare2015}. Similarly, if the constraints form a convex feasible set, 
they do not need to be linearized and the simplified bundle algorithm converges without line search.  
\begin{proposition}\label{prop:convex-constraint}
	If the equality constraint $c(\cdot)$ and bound constraints in~\eqref{eqn:opt-ms-simp} form a convex and bounded set in $\Rbb^n$, 
	then instead of~\eqref{eqn:opt-ms-simp-bundle} we solve subproblem
        \begin{equation} \label{eqn:opt-ms-simp-bundle-2}
        \centering
         \begin{aligned}
          &\underset{\substack{x}}{\text{minimize}} 
	  & & \phi_k(x)\\
          &\text{subject to}
	  & & c(x) = 0, \\
	  &&& 0 \leq x\leq x_u.
  \end{aligned}
\end{equation}
	And the line search step can be skipped with $x_{k+1} = x_k +d_k$, $d_k$ being the solution to~\eqref{eqn:opt-ms-simp-bundle-2}.
	The convergence properties are maintained.
\end{proposition}

\subsection{Application to two-stage stochastic optimization problem}\label{sec:app}
The algorithm and convergence analysis can be readily extended to two-stage stochastic programming problems, 
where the quadratic approximation function $\phi_k(\cdot)$ is needed only for the nonsmooth nonconvex 
second-stage solution functions. Problem~\eqref{eqn:opt-ms} is approximated locally as 
\begin{equation} \label{eqn:opt-ms-phi}
 \centering
  \begin{aligned}
   &\underset{\substack{x}}{\text{minimize}} 
	  & & f(x)+ \phi_k(x)\\
   &\text{subject to}
	  & & c(x) =c_E\\
	  &&& d^l \leq d(x) \leq d^u\\
	  &&& x^l \leq x \leq x^u.\\
  \end{aligned}
\end{equation}
The first-stage objective $f(\cdot)$, which is continuously differentiable, 
is kept as it is. As a result, we can take advantage of the sparsity structure arising
from $f(\cdot)$ since the Hessian of $\phi_k(\cdot)$ is diagonal.

The update rule of $\alpha_k$ is critical and problem dependent. 
It is a trade-off between robust convergence behavior (large $\alpha_k$) and fast but potentially unstable convergence (small $\alpha_k$). The $\alpha_k$ in Algorithm~\ref{alg:simp-bundle} does not 
bear much meaningful structure from the objective function since the Hessian itself might not exist. 
Nevertheless, for problems with better differentiability, it is 
possible to explore ways to extract more second-order information.

One option is to use the Barzilai-Borwein (BB) gradient method~\cite{barzilai_1988}, 
which can be interpreted as an approximation to the secant equation. The update rule for $\alpha_k$ is
\begin{equation} \label{eqn:alpha-BB}
 \centering
  \begin{aligned}
	  \alpha_k = \frac{s_{k-1}^Ty_{k-1}}{y_{k-1}^Ty_{k-1}},
  \end{aligned}
\end{equation}
where $s_{k-1}=x_k-x_{k-1},y_{k-1}=g_k-g_{k-1}$. 
This choice of $\alpha_k$ can in practice  
increase the convergence rate if the objective $r(\cdot)$ have more 
favorable properties~\cite{raydan_1993}.
Alternatively, $\alpha_k$ can be viewed as 
a measure of the inverse of a trust-region radius. The larger $\alpha_k$ is, the smaller the step size will be. Hence, 
$\alpha_k$ can be updated based on how accurate the previous approximation is, as
in trust-region methods~\cite{Nocedal_book}. This view is adopted in the proposed algorithm. 
A simple multiplication rule where $\alpha_{k+1}=\eta_{\alpha}\alpha_k,\eta_{\alpha}>1$ could be effective when $\alpha_k$ is 
increased.
In all cases, problem specific $\alpha_{max}$ and $\alpha_{min}$ can be assigned to 
make the algorithm more efficient and robust.
This is the area of the algorithm that is rich 
for experimentation.

It is also possible to gauge $\alpha_k$ based on function value, in addition to the trust-region ratio $\rho_k$. 
If the function value range of $r(\cdot)$ is known, 
such rules might provide better estimate of $\alpha_k$.
For example, we can find $\alpha_k$ by
requiring the minimum value of $\phi_k(\cdot)$ over a chosen subset of domain $X'\subset X$ to be larger than certain ratio of 
the function value at $x_k$, \textit{i.e.},  
\begin{equation} \label{eqn:rc-quad-appx-2}
 \centering
  \begin{aligned}
        &\underset{\substack{x\in X'}}{\text{minimize}} \ \phi_k(x) \geq \eta_k r_k
  \end{aligned}
\end{equation}
where $\eta_k$ is the chosen ratio. 

The same $\rho_k$ is computed and if $\rho_k>0$ is not satisfied as in step 8 and 14 of Algorithm~\ref{alg:simp-bundle},
$\eta_k$ is increased by the fixed increase ratio $\eta_{\alpha}$ with $\eta_{k+1}=\eta_{\alpha}\eta_k$.
Using $\eta_k$ as an intermediate parameter, $\alpha_k$ is then obtained 
as the minimum value that would hold~\eqref{eqn:rc-quad-appx-2} true.
The choice of $\alpha_k$ thus depends on local function value $r_k$ and subgradient $g_k$  
as well as $\eta_k$ and will no longer stay monotonic throughout the iterations as in Algorithm~\ref{alg:simp-bundle}.

More importantly in practice, $\alpha_k$ can be reduced when $\rho_k$ 
behaves well, \textit{e.g.}, is close to $1$. In our experience, reducing $\alpha_k$ 
helps to achieve convergence faster while the algorithm remains robust due to the mechanism of rejecting a step.
To further speed up convergence, scalar $\alpha_k$ can also be replaced by a diagonal matrix
with varying values. 
One way of specifying the diagonal values is to take into account the distance 
between a component of $x$ and its upper and lower bounds. 
It is possible that multiple components of the optimization variable $x$ reach their upper/lower bounds.
Since they are more likely to stay at the bounds, it is reasonable to assign them larger corresponding diagonal values
of the matrix $\alpha_k$ to encourage movement of other components of $x$. 
For a first-order algorithm, this could make a difference in
convergence and proves to be so in the SCACOPF application.

\subsection{Consistency restoration in linearized constraint}\label{sec:lincons}
As mentioned previously, the linearized constraints of the model subproblem~\eqref{eqn:opt-ms-simp-bundle} can become infeasible even when the original problem~\eqref{eqn:opt-ms-simp} is feasible, a phenomenon referred to as inconsistency, which is also present in SQP methods. In this section we propose a supplemental consistency restoration algorithm to tackle this difficulty. 
This algorithm solves, instead of~\eqref{eqn:opt-ms-simp-bundle}, a penalized subproblem  where the constraints are incorporated into the objective in hope of generating a new serious point with consistent linearized constraints.
As is common with penalty methods, the accumulation points might not be feasible KKT points. 
For the update rule of the penalty parameter,
we borrow an idea from a sequential linear-quadratic programming (SLQP) method in~\cite{byrd2005} that
relies on a feasibility problem solution.

Whenever problem~\eqref{eqn:opt-ms-simp-bundle} has inconsistent linearized constraints,
the following penalty problem is formulated:
\begin{equation} \label{eqn:opt-ms-simp-bundle-penal-nonsmooth}
 \centering
  \begin{aligned}
   &\underset{\substack{d}}{\text{minimize}} 
	  & & \pi_k \Phi_k(d) + \norm{c(x_k)+\nabla c(x_k)^T d}_1 \\
   &\text{subject to}
	  & & d_l^k\leq d\leq d_u^k,
  \end{aligned}
\end{equation}
where $\pi_k\geq 0$ is the penalty parameter. 
While~\eqref{eqn:opt-ms-simp-bundle-penal-nonsmooth} is straightforward, 
 to avoid the difficulties with nonsmooth objective, as conventional in SQP methods, the following equivalent quadratic programming problem is solved instead
\begin{equation} \label{eqn:opt-ms-simp-bundle-penal}
 \centering
  \begin{aligned}
   &\underset{\substack{d,v,w}}{\text{minimize}} 
	  & & \pi_k \Phi_k(d) + \sum_{j=1}^m (v_j + w_j)  \\
   &\text{subject to}
	  & & c_j(x_k)+\nabla c_j(x_k)^T d = v_j - w_j, \ j=1,...m,\\
	  &&& d_l^k \leq d\leq d_u^k,\\
	  &&& 0 \leq v, w ,
  \end{aligned}
\end{equation}
where $v,w\in \Rbb^m$ are slack variables. Denoting $d_k,v^k,w^k$ as the solutions to~\eqref{eqn:opt-ms-simp-bundle-penal}, 
the first-order optimality conditions of problem~\eqref{eqn:opt-ms-simp-bundle-penal} 
involving $d$ are
\begin{equation} \label{eqn:simp-bundle-penal-KKT-1}
  \centering
   \begin{aligned}
	   \pi_k \left[g_k + \alpha_k d_k\right] + \sum_{i=1}^m \lambda^{k+1}_j \nabla c_j(x_k)-\zeta_l^{k+1}+\zeta_u^{k+1} =0&,\\
	   \lambda^{k+1}_j \left[c_j(x_k)+\nabla c_j(x_k)^T d_k - v^k_j + w^k_j \right] =0, \ j=1,\dots,m,\\
	   c_j(x_k)+\nabla c_j(x_k)^T d_k - v^k_j + w^k_j =0, \ j=1,\dots,m,\\
	   Z_u^{k+1}(d_k-d_u^k) = 0, Z_l^{k+1}(d_k-d_l^k) = 0&,\\
	   \zeta_u^{k+1},\zeta_l^{k+1},d_k+x_k,x_u-x_k-d_k\geq 0&.\\
   \end{aligned}
 \end{equation}
Here, $\lambda^{k+1}\in\Rbb^m$,$\zeta_u^{k+1},\zeta_l^{k+1}\in \Rbb^n$ are the Lagrange multipliers for the constraints on $d$. 
The matrices $Z_u^{k+1},Z_l^{k+1}$ are diagonal matrices whose diagonal values are $\zeta_u^{k+1}$ and $\zeta_l^{k+1}$, respectively.
The remaining optimality conditions on slack variables $v$ and $w$ are 
\begin{equation} \label{eqn:simp-bundle-penal-KKT-2}
  \centering
   \begin{aligned}
	   1-\lambda^{k+1}_j - p^{k+1}_j =0, \ j=1,...m,\\
           1+\lambda^{k+1}_j-q^{k+1}_j=0, \ j=1,...m,\\
	   P^{k+1} v^k  = 0, Q^{k+1} w^k = 0&,\\
	   v^k,w^k,p^{k+1},q^{k+1} \geq 0&,\\
   \end{aligned}
 \end{equation}
where $p^{k+1},q^{k+1}\in\Rbb^m$ are the Lagrange multipliers for $v^k,w^k$.
The matrices $P^{k+1},Q^{k+1}$ are diagonal matrices whose diagonal values are $p^{k+1}$ and $q^{k+1}$, respectively.

Based on whether the slack variable bound constraints are active, the relations between Lagrange multipliers 
can be simplified. 
To see that, define the sign function $\sigma_j^k:\Rbb^n\to\Rbb,j=1,\dots,m$ of $d$ such that 
\begin{equation} \label{eqn:opt-ms-simp-sign}
 \centering
  \begin{aligned}
	  \sigma_j^k(d)  =  
	  \begin{cases}
		  -1,& c_j(x_k)+\nabla c_j(x_k)^T d < 0 \\  
		  \phantom{-} 0,& c_j(x_k) + \nabla c_j(x_k)^T d = 0\\
		  \phantom{-} 1,& c_j(x_k) + \nabla c_j(x_k)^T d > 0
    \end{cases}.
  \end{aligned}
\end{equation}
In addition, we divide the constraints into two sets based on the value of $ c_j(x_k)+ \nabla c_j(x_k)^T d_{k}$. 
For simplicity, the two sets are referred to as the set of active and inactive equality constraints as we do without slack variables. 
More specifically, the active equality constraint set is defined at $d_k$ as
\begin{equation} \label{eqn:opt-ms-simp-bundle-fea-A}
   \centering
    \begin{aligned}
     A_k = \{ 1\leq j\leq m|c_j(x_k)+\nabla c_j(x_k)^T d_{k} = 0\},
     \end{aligned}
\end{equation}
and the inactive equality constraint set is
\begin{equation} \label{eqn:opt-ms-simp-bundle-fea-I}
  \centering
    \begin{aligned}
      V_k = \{ 1\leq j\leq m|c_j(x_k)+\nabla c_j(x_k)^T d_{k} \neq 0\}.
    \end{aligned}
\end{equation}
We can now integrate the optimality conditions~\eqref{eqn:simp-bundle-penal-KKT-2} into~\eqref{eqn:simp-bundle-penal-KKT-1}
in the following Lemma.

\begin{lemma}\label{lem:lambda-relaxed-prop}
	For inactive equality constraints $c_j(\cdot), j\in V_k$, $\lambda_j^{k+1} = \sigma_j^k(d_k)$. For active equality constraints, \textit{i.e.}, $j\in A_k$, 
        $-1\leq \lambda_j^{k+1}\leq 1$.
\end{lemma}
\begin{proof}
	Note first that for any $1\leq j\leq m$, the slack variable solutions satisfy $v^k_j w^k_j = 0$. This is 
	due to the bound constraints on $v,w$ and their presence in the objective.
	Next, we consider the three cases given the value of $c_j(x_k)+\nabla c_j(x_k)^T d_k$, which corresponds 
	to the three values of $\sigma_j^k(d_k)$ for $j=1,\dots,m$.
	The first two cases both have $j\in V_k$.
	If $\sigma_j^k(d_k)=1$, then by the third equation in~\eqref{eqn:simp-bundle-penal-KKT-1}, $v_j^k>0,w_j^k=0$.
	From the complementarity equations in~\eqref{eqn:simp-bundle-penal-KKT-2}, the corresponding Lagrange multiplier to $v_j^k$ is $0$, \textit{i.e.},
        $p^{k+1}_j=0$. By the first equation in~\eqref{eqn:simp-bundle-penal-KKT-2}, $\lambda_j^{k+1}=1$.
	If $\sigma_j^k(d_k)=-1$, then similarly using the third equation in~\eqref{eqn:simp-bundle-penal-KKT-1}, $v_j^k=0,w_j^k>0$ and the corresponding Lagrange multiplier  
$q^{k+1}_j=0$. By the second equation in~\eqref{eqn:simp-bundle-penal-KKT-2}, $\lambda_j^{k+1}=-1$. The first part of the Lemma is proven.

	In the last case, $j\in A_k$, \textit{i.e.}, $\sigma_j^k(d_k)=0$.  By the third equation in~\eqref{eqn:simp-bundle-penal-KKT-1}, we have $v_j^k=0,w_j^k=0$.  
Combine the first two equations in~\eqref{eqn:simp-bundle-penal-KKT-2} through summation and subtraction, we obtain 
      \begin{equation} \label{eqn:lambda-pq}
       \centering
        \begin{aligned}
		\lambda_j^{k+1}&= \frac{1}{2}(q_j^{k+1}-p_j^{k+1}), \\
		2&=p_j^{k+1}+q_j^{k+1}.\\
        \end{aligned}
       \end{equation}
	Applying the bound constraints on $p^{k+1},q^{k+1}$ to the second equation in~\eqref{eqn:lambda-pq}, we have $0\leq p_j^{k+1}\leq 2$, $0\leq q_j^{k+1}\leq 2$ 
	and therefore from the first equation in~\eqref{eqn:lambda-pq} $-1\leq \lambda_j^{k+1} \leq 1$.
\end{proof}

Similar to~\eqref{def:pd}, we define $\delta_k^{\pi_k}$ to be the change in objective of the penalty subproblem~\eqref{eqn:opt-ms-simp-bundle-penal} with penalty $\pi_k$, which based on~\eqref{eqn:opt-ms-simp-bundle-penal-nonsmooth} is 
\begin{equation}\label{def:pd-penal}
   \begin{aligned}
	   \delta_k^{\pi_k} =& \pi_k\left(-g_k^T d_k -\frac{1}{2}\alpha_k \norm{d_k}^2\right) +\norm{c(x_k)}_1 -\norm{c(x_k)+\nabla c(x_k)^T d_k}_1.\\
   \end{aligned}
\end{equation}
The ratios $\rho_k$ and $\rho_k^{\beta}$ are again used to address the nonsmoothness of $r(\cdot)$, 
whose definitions are given in~\eqref{eqn:decrease-ratio-1} and~\eqref{eqn:decrease-ratio-beta}.
The algorithm also requires line search given $d_k$ to obtain a serious step 
$x_{k+1}=x_k+\beta_k d_k$, $\beta_k\in (0, 1]$. 
To simplify the analysis, we adopt in this section $\eta_{\gamma}^-=\eta_{\gamma}^+=1$ and $\eta_l^-=\eta_l^+=1$, 
making the definition of $\rho_k,\rho_k^{\beta}$ in~\eqref{eqn:decrease-ratio-1} and~\eqref{eqn:decrease-ratio-beta} identical 
across branches.
Thus for a serious step, regardless of the value of $\alpha_k$, 
the same expression between predicted and actual change in $r(\cdot)$ is satisfied. 
That is, for a serious step, 
$r(x_k)-r(x_{k+1}) > \Phi_k(0)-\Phi_k(\beta_kd_k)$ and $r(x_k)-r(x_{k}+d_k) > \Phi_k(0)-\Phi_k(d_k)$. 

The renewed merit function and line search conditions are 
\begin{equation}\label{def:merit-fea}
   \begin{aligned}
	   \phi_{1\pi_k}(x)  =  r(x) + \frac{1}{\pi_k} \norm{c(x)}_1,
   \end{aligned}
\end{equation}
and 
\begin{equation}\label{eqn:line-search-cond-penal}
   \begin{aligned}
   \centering
	   \frac{1}{\pi_k} \norm{c(x_k)}_1  +  \frac{\beta_k}{\pi_k} (\lambda^{k+1})^T \nabla c(x_k)^T d_k  \geq
	   \frac{1}{\pi_k} \norm{c(x_{k+1})}_1 -\eta_{\beta}\frac{1}{2}\alpha_k\beta_k\norm{d_k}^2, 
   \end{aligned}
\end{equation}
respectively.

To update the penalty parameter, 
the following feasibility problem is also solved:
\begin{equation} \label{eqn:opt-ms-simp-bundle-fea}
 \centering
  \begin{aligned}
   &\underset{\substack{d}}{\text{minimize}} 
	  & & \norm{c(x_k) + \nabla c(x_k)^T d}_1\\
   &\text{subject to}
	  & & d_l^k\leq d\leq d_u^k.
  \end{aligned}
\end{equation}
Denote by $d_{k}^f$ the solution to~\eqref{eqn:opt-ms-simp-bundle-fea} and $\delta_k^f$ its predicted decrease,
whose form is 
\begin{equation}\label{def:pd-fea} 
   \begin{aligned}
	   \delta_k^f = \norm{c(x_k)}_1-\norm{c(x_k)+\nabla c(x_k)^T d_k^f}_1.
   \end{aligned}
\end{equation}
Notice that $\delta_k^f\geq 0$. This value is compared against $\delta_k^{\pi_k}$. 

The consistency restoration algorithm is given in Algorithm~\ref{alg:simp-bundle-const}. It is called upon by 
Algorithm~\ref{alg:simp-bundle} when inconsistency occurs at step 3  and exits after one serious step iteration in step 14 of Algorithm~\ref{alg:simp-bundle-const}. However, 
it is possible that the linearized constraints remain inconsistent and   
Algorithm~\ref{alg:simp-bundle-const} is called repeatedly. 
In this case, the update rule of the penalty parameter ensures that the algorithm converges toward critical points for linearized constraint violations. A point $\bar{x}$ is called a critical point of the linearized constraint violation of $c(\cdot)$ if $\delta_k^f=0$ at $\bar{x}$. Notice such a critical point can be either feasible or infeasible to the original problem~\eqref{eqn:opt-ms-simp}.
\begin{algorithm}
 \DontPrintSemicolon
 \SetAlgoNoLine 
   \caption{Simplified bundle method: consistency restoration}\label{alg:simp-bundle-const}
	Given $x_k$, $\alpha_k$, $r(x_k)$, $g(x_k)$, $\theta_k$ and other parameters such as $\epsilon$ from Algorithm~\ref{alg:simp-bundle}, choose the update coefficient $0<\eta_{\pi},\eta_f<1$ for $\pi_k$ and error tolerance $\epsilon^f\geq 0$.\; 
	If $\pi_{k-1}$ does not exist, let $\pi_k=\frac{1}{\theta_k}$. Otherwise let $\pi_k=\min{(\pi_{k-1},\frac{1}{\theta_k})}$. 
	Solve~\eqref{eqn:opt-ms-simp-bundle-penal} with $\pi_k$ and obtain $d_k$.\;
	Solve the feasibility problem~\eqref{eqn:opt-ms-simp-bundle-fea} to obtain solution $d_k^f$ and compute 
	 $\delta_k^f$ from~\eqref{def:pd-fea}.\;
	\If{$\delta_k^f < \epsilon^f$}{
         Stop the iteration and exit the algorithm. \;
        }
	\While{$\delta_k^{\pi_k}<\eta_f \delta_k^f$}{
		Reduce $\pi_k$ through $\pi_k = \eta_{\pi} \pi_{k}$ and re-solve~\eqref{eqn:opt-ms-simp-bundle-penal} with the updated $\pi_k$.\;
	}
	Obtain the solution $d_k$ and Lagrange multipliers $\lambda^{k+1}$ given $\pi_k$.
	Evaluate $r(x_k+d_k)$ and compute $\delta_k$ in~\eqref{def:pd} and $\rho_k$ in~\eqref{eqn:decrease-ratio-1}.\;
	\If{$\rho_k > 0$}{
		  Find the line search parameter $\beta_k>0$ using backtracking, starting at $\beta_k=1$ and halving 
		  if too large, such that~\eqref{eqn:line-search-cond-penal} is satisfied. Compute $\rho_k^{\beta}$ in~\eqref{eqn:decrease-ratio-beta}.\;
            \If{$\rho_k^{\beta} < 0$}{
		    Break and go to 17.\;
	    }
	    Take the serious step $x_{k+1} = x_k+\beta_k d_k$. \;
	    Exit consistency restoration. Go back to Algorithm~\ref{alg:simp-bundle} and start a new iteration.\;
	  }
	  \Else{
	    Reject the trial step and update $\alpha_k$ with $\alpha_{k+1}=\eta_{\alpha}\alpha_k$.\;
	    Go back to step 2.
	}
\end{algorithm}

While Algorithm~\ref{alg:simp-bundle-const} solves a penalized subproblem instead, 
it includes all the elements in Algorithm~\ref{alg:simp-bundle} to deal with the nonsmoothness of $r(\cdot)$, including the update rule for $\alpha_k$.
Thus, we can reuse many of the same conclusions from Section~\ref{sec:alg-convg}  
and only provide rigorous proofs if necessary.
Since the acceptance and rejection of a trial step is based on $\rho_k$ and $\rho_k^{\beta}$, which in turn 
solely relies on properties of $r(\cdot)$, Lemma~\ref{lem:sufficientdecrease} holds true, as claimed in the following Lemma.
\begin{lemma}\label{lem:sufficientdecrease-penal}
	Under the Assumption~\eqref{assp:upperC2} of upper-$C^2$ property for the objective $r(\cdot)$, the consistency restoration Algorithm~\ref{alg:simp-bundle-const} produces a finite number 
	of rejected steps. Consequently, the parameter $\alpha_k$ of Algorithm~\ref{alg:simp-bundle-const} stabilizes, \textit{i.e.}, there exists $k$ such that $\alpha_t = \alpha_k$ for all $t\geq k$. 
\end{lemma}
\begin{proof}
	Since Algorithm~\ref{alg:simp-bundle-const} has identical mechanism for rejecting steps and increasing $\alpha_k$ to Algorithm~\ref{alg:simp-bundle}, which only relies on the property of $r(\cdot)$, the proof of Lemma~\ref{lem:sufficientdecrease} can be directly applied here.
	That is, only a finite number of rejected steps are needed to achieve $\alpha_k>2C$, which guarantees $\rho_k>0$, $\rho_k^{\beta}>0$, and produces a serious step. If $\alpha_k\leq 2C$ for all $k$, then only finite number of rejected steps are generated, which also ensures $\rho_k>0,\rho_k^{\beta}>0$ for $k$ large enough.
As a result, there exists a $k$ such that $\alpha_t=\alpha_k$ for $t\geq k$  (see proof of Lemma~\ref{lem:sufficientdecrease}), with a finite number of rejected steps produced.
\end{proof}
The following lemma shows that the update rule for $\pi_k$ in Algorithm~\ref{alg:simp-bundle-const} is well-defined.
\begin{lemma}\label{lem:consistency-penal-defined}
	The steps 6-7 in Algorithm~\ref{alg:simp-bundle-const} terminates successfully, \textit{i.e.}, there exists a $\pi_k>0$ such that $\delta_k^{\pi_k}\geq \eta_f \delta_k^f$ and such a $\pi_k$ can be found within finite steps.  
\end{lemma}
\begin{proof}
	Since $d_k$ is the solution to~\eqref{eqn:opt-ms-simp-bundle-penal} (and equivalently~\eqref{eqn:opt-ms-simp-bundle-penal-nonsmooth}), we have
	by~\eqref{def:pd-fea} and~\eqref{def:pd-penal} 
   \begin{equation} \label{eqn:simp-bundle-fea-finite}
   \centering
    \begin{aligned}
	    \delta_k^{\pi_k} \geq& \pi_k\left(-g_k^T d_k^f  - \frac{1}{2} \alpha_k \norm{d_k^f}^2\right) 
	    +\norm{c(x_k)}_1 - \norm{c(x_k)+ \nabla c(x_k)^T d_k^f}_1 \\
		     =& \pi_k\left(-g_k^T d_k^f  - \frac{1}{2} \alpha_k \norm{d_k^f}^2\right) + \delta_k^f \geq \pi_k\left(-\norm{g_k}\norm{d_k^f} - \frac{1}{2}\alpha_k\norm{d_k^f}^2\right)+ \delta_k^f,\\
    \end{aligned}
   \end{equation}
	where the last inequality uses Cauchy-Schwarz inequality.
	From Lemma~\ref{lem:sufficientdecrease-penal} and Lipschitz continuity, $d_k^f$, $g_k$, and $\alpha_k$ are all bounded. Assigning $D = \norm{d_u^k-d_l^k}=\norm{x_u}$, we derive the condition on $\pi_k$ such that $\delta_k^{\pi_k}\geq \eta_f \delta_k^f$ as 
   \begin{equation} \label{eqn:simp-bundle-fea-pi}
   \centering
    \begin{aligned}
	    \pi_k  \leq  \frac{(1-\eta_f)\delta_k^f}{ \norm{g_k} D + \frac{1}{2}\alpha_k D^2 }.\\
    \end{aligned}
   \end{equation}
	Thus, such $\pi_k>0$ exists as long as $\delta_k^f>0$. And since $\pi_k$ is reduced through $\eta_{\pi}<1$, only a finite number of steps are needed to obtain a $\pi_k$ that satisfied~\eqref{eqn:simp-bundle-fea-pi} through step 9.
	If $\delta_k^f = 0$ and consistency restoration Algorithm~\ref{alg:simp-bundle-const} is still called, the algorithm would terminate at step 5 (and had converged to a critical point to the linearized constraint). 
\end{proof}

\begin{lemma}\label{lem:line-search-merit-2}
	Given a nonzero penalty parameter $\pi_k>0$ and Assumption~\ref{assp:boundedHc}, the line search step 10 of Algorithm~\ref{alg:simp-bundle-const} finds $\beta_k\in (0,1]$ satisfying the condition~\eqref{eqn:line-search-cond-penal} in a finite number of steps.
\end{lemma}
\begin{proof}
	Since $c(\cdot)$ is smooth,  by Taylor expansion of its $jth$ component,
   \begin{equation} \label{eqn:simp-bundle-fea-c-taylor}
   \centering
    \begin{aligned}
	    c_j(x_k+d_k) - c_j(x_k)  = \nabla c_j(x_k)^T d_k  + \frac{1}{2} d_k^T H_{k}^j d_k,\\
    \end{aligned}
   \end{equation}
	where $1\leq j \leq m$ and the Hessian $H_k^j$ depends on both $x_k$ and $d_k$. Similarly, 
       \begin{equation} \label{eqn:simp-bundle-fea-c-ls-pf-1}
         \centering
         \begin{aligned}
		 c_j(x_{k+1})  = & c_j(x_k)+ \beta_k \nabla c_j(x_k)^T d_{k} + \frac{1}{2}\beta_k^2 d_k^T H_{k\beta}^{j} d_k.\\
      \end{aligned}
     \end{equation}
	From Assumption~\ref{assp:boundedHc},
   \begin{equation} \label{eqn:simp-bundle-fea-c-ls-pf-2}
   \centering
    \begin{aligned}
	    |c_j(x_{k+1})|  &= | c_j(x_k)+ \beta_k \nabla c_j(x_k)^T d_{k}+ \frac{1}{2}\beta_k^2 d_k^T H_{k\beta}^{j} d_k|\\
	    &\leq | c_j(x_k) + \beta_k \nabla c_j(x_k)^T d_{k}| + \beta_k^2 H^{c}_u \norm{d_k}^2.
    \end{aligned}
   \end{equation}
	From the definition of $\sigma_j^k$ in~\eqref{eqn:opt-ms-simp-sign}, we can write 
   \begin{equation} \label{eqn:simp-bundle-fea-c-ls-pf-2.1}
   \centering
    \begin{aligned}
	    |c_j(x_k)+ \nabla c_j(x_k)^T d_{k}| = \sigma_j^k(d_k) \left[c_j(x_k)+ \nabla c_j(x_k)^T d_{k}\right].
    \end{aligned}
   \end{equation}
	Using the fact that the absolute function  $|\cdot|$ is convex, the function $|c_j(x_k)+\nabla c_j(x_k)^T d|$ is convex in $d$.
	Taking its value at $d=0,d=d_k$ and $\beta_k d_k$, we have by convexity 
 \begin{equation} \label{eqn:simp-bundle-fea-c-ls-pf-2.2}
   \centering
    \begin{aligned}
	    | c_j(x_k) + \beta_k \nabla c_j(x_k)^T d_{k}|  & \leq (1-\beta_k)| c_j(x_k)| + \beta_k |c_j(x_k)+ \nabla c_j(x_k)^T d_{k}|. 
    \end{aligned}
   \end{equation}
	Applying~\eqref{eqn:simp-bundle-fea-c-ls-pf-2.2} to~\eqref{eqn:simp-bundle-fea-c-ls-pf-2}, 
\begin{equation} \label{eqn:simp-bundle-fea-c-ls-pf-2.3}
   \centering
    \begin{aligned}
	    |c_j(x_{k+1})|  & \leq (1-\beta_k)| c_j(x_k)| + \beta_k |c_j(x_k)+ \nabla c_j(x_k)^T d_{k}| + \beta_k^2 H^{c}_u \norm{d_k}^2\\
		  & = (1-\beta_k)|c_j(x_k)| + \beta_k \sigma_j^k(d_k)\left[c_j(x_k)+ \nabla c_j(x_k)^T d_{k}\right] + \beta_k^2 H^{c}_u \norm{d_k}^2.
    \end{aligned}
   \end{equation}
	Let us denote the cardinality of $A_k$ and $V_k$ by $|A_k|$ and $|V_k|$, respectively. By summing up~\eqref{eqn:simp-bundle-fea-c-ls-pf-2.3} over $j\in V_k$ and realizing that $|c_j(x_k)|\geq \sigma_j^k(d_k) c_j(x_k)$, we have
   \begin{equation} \label{eqn:simp-bundle-fea-c-ls-pf-2.5}
     \centering
      \begin{aligned}
	      \sum_{j\in V_k}|c_j(x_{k+1})|  \leq& 
	       \sum_{j \in V_k}  |c_j(x_k)|+\beta_k\sum_{j \in V_k} \sigma_j^k(d_k)\nabla c_j(x_k)^Td_k 
	      + |V_k|\beta_k^2 H^{c}_u \norm{d_k}^2.\\
      \end{aligned}
    \end{equation}
	Similarly, we sum up~\eqref{eqn:simp-bundle-fea-c-ls-pf-2.3} over $j\in A_k$ and apply its definition in~\eqref{eqn:opt-ms-simp-bundle-fea-A} to write 
   \begin{equation} \label{eqn:simp-bundle-fea-c-ls-pf-2.6}
     \centering
      \begin{aligned}
\sum_{j\in A_k}|c_j(x_{k+1})|  \leq& 
	      (1-\beta_k)\sum_{j \in A_k}  |c_j(x_k)| + |A_k| \beta_k^2 H^{c}_u \norm{d_k}^2. \\
      \end{aligned}
    \end{equation}
	Further, by~\eqref{eqn:simp-bundle-fea-c-ls-pf-2.5} and~\eqref{eqn:simp-bundle-fea-c-ls-pf-2.6}, we have 
\begin{equation} \label{eqn:simp-bundle-fea-c-ls-pf-2.8}
     \centering
      \begin{aligned}
	      \sum_{j\in V_k}\left(|c_j(x_k)|-|c_j(x_{k+1})|\right)  \geq& 
	       -\beta_k \sum_{j \in V_k} \sigma_j^k(d_k)\nabla c_j(x_k)^Td_k 
	      - |V_k|\beta_k^2 H^{c}_u \norm{d_k}^2,\\
	      \sum_{j\in A_k}\left(|c_j(x_k)|-|c_j(x_{k+1})|\right)  \geq& 
	    \beta_k  \sum_{j \in A_k}  |c_j(x_k)| - |A_k| \beta_k^2 H^{c}_u \norm{d_k}^2. \\
      \end{aligned}
    \end{equation}
    Summing the two equations in~\eqref{eqn:simp-bundle-fea-c-ls-pf-2.8} and applying $|A_k|+|V_k|=m$ gives us 
\begin{equation} \label{eqn:simp-bundle-fea-c-ls-pf-3}
     \centering
      \begin{aligned}
	      \norm{c(x_k)}_1-\norm{c(x_{k+1})}_1 \geq  
	       -\beta_k \sum_{j \in V_k} \sigma_j^k(d_k)\nabla c_j(x_k)^T d_k 
	    +\beta_k  \sum_{j \in A_k}  |c_j(x_k)| - m\beta_k^2 H^{c}_u \norm{d_k}^2. \\
      \end{aligned}
    \end{equation}
    From Lemma~\eqref{lem:lambda-relaxed-prop} and the definition of $A_k$ in~\eqref{eqn:opt-ms-simp-bundle-fea-A}, we can write
\begin{equation} \label{eqn:simp-bundle-fea-c-ls-pf-3.1}
     \centering
      \begin{aligned}
	      \sum_{j=1}^m \lambda_j^{k+1}\nabla c_j(x_k)^T d_k =& \sum_{j\in V_k} \lambda_j^{k+1} \nabla c_j(x_k)^T d_k +\sum_{j\in A_k} \lambda_j^{k+1} \nabla c_j(x_k)^T d_k\\
	      =& \sum_{j\in V_k} \sigma_j^k(d_k) \nabla c_j(x_k)^T d_k -\sum_{j\in A_k} \lambda^{k+1}_j  c_j(x_k).\\
	      \geq& \sum_{j\in V_k} \sigma_j^k(d_k) \nabla c_j(x_k)^T d_k -\sum_{j\in A_k}  |c_j(x_k)|.\\
      \end{aligned}
    \end{equation}
    The inequality in~\eqref{eqn:simp-bundle-fea-c-ls-pf-3.1} comes from the second part of Lemma~\ref{lem:lambda-relaxed-prop}.
	Through simple algebraic calculations and applying~\eqref{eqn:simp-bundle-fea-c-ls-pf-3} and~\eqref{eqn:simp-bundle-fea-c-ls-pf-3.1} 
   \begin{equation} \label{eqn:simp-bundle-fea-c-ls-pf-3.5}
   \centering
    \begin{aligned}
	    & \frac{1}{\pi_k} \norm{c(x_k)}_1 -\frac{1}{\pi_k}\norm{c(x_{k+1})}_1 + \frac{\beta_k}{\pi_k} \sum_{j=1}^m \lambda^{k+1}_j\nabla c_j(x_k)^T d_k
	     \geq \\
	    &\qquad \frac{1}{\pi_k} \sum_{j \in A_k} \beta_k |c_j(x_k)|  
	      -\frac{\beta_k}{\pi_k} \sum_{j \in V_k} \sigma_j^k(d_k)\nabla c_j(x_k)^T d_k
	      - \frac{1}{\pi_k} m\beta_k^2 H^{c}_u \norm{d_k}^2 \\ 
	    & \qquad+ \frac{\beta_k}{\pi_k}\sum_{j\in V_k} \sigma_j^k(d_k) \nabla c_j(x_k)^T d_k
	     - \frac{\beta_k}{\pi_k} \sum_{j\in A_k} | c_j(x_k) | \geq 
	    - \frac{1}{\pi_k} m \beta_k^2 H^{c}_u \norm{d_k}^2.
    \end{aligned}
   \end{equation}
   Thus, if $\beta_k$ satisfies 
   \begin{equation} \label{eqn:simp-bundle-fea-c-ls-pf-4}
   \centering
    \begin{aligned}
	    0< \beta_k \leq \frac{\eta_{\beta} \alpha_k \pi_k}{2m H^c_u },
    \end{aligned}
   \end{equation}
   where both the denominator and numerator are positive and independent of the line search, we have~\eqref{eqn:line-search-cond-penal} 
   satisfied.
Using ceiling function $\lceil\cdot \rceil$, we can write 
   \begin{equation} \label{eqn:simp-bundle-fea-c-ls-pf-5}
   \centering
    \begin{aligned}
	    \beta_k \geq \frac{1}{2} ^{\lceil \log_{\frac{1}{2}} \frac{\eta_{\beta} \alpha_k\pi_k}{2 m H^c_u} \rceil}.
    \end{aligned}
   \end{equation}
	If ${\pi_k}>0$, the line search then successfully terminates after a finite number of steps based on the backtracking rule. 
\end{proof}


\noindent The decrease in merit function follows, similar to Lemma~\ref{lem:sufficientdecrease-merit}. 
\begin{lemma}\label{lem:sufficientdecrease-merit-penal}
	The serious step $x_{k+1}=x_k+\beta_k d_k$ is a descent step for the merit function~\eqref{def:merit-fea} if 
	$\beta_k$ is obtained through line search in Algorithm~\ref{alg:simp-bundle-const}.
        Further, if $\pi_k$ stabilizes at a finite value,
	\textit{i.e.}, there exists $k$ such that $\pi_t=\pi_k:=\bar{\pi}>0$ for all $t\geq k$, then the speed of descent 
	satisfies $\phi_{1\pi_k}(x_k)-\phi_{1\pi_k}(x_{k+1})>c_{\phi}^{\pi}\norm{d_k}^2$ for some $c_{\phi}^{\pi}>0$.
\end{lemma}
\begin{proof}
	Since  $\rho_k>0,\rho_k^{\beta}>0$ at any serious step,  and $\eta_l^+=\eta_l^-=\eta_{\gamma}^+=\eta_{\gamma}^-=1$, we can compactly write based on definitions~\eqref{eqn:decrease-ratio-1} and~\eqref{eqn:decrease-ratio-beta}
         \begin{equation} \label{eqn:opt-ms-appx-rec-obj4}
         \centering
         \begin{aligned}
		 r_k - r(x_k+d_k) 
			> &  \Phi_k(0) -\Phi_k(d_k), \\
		 r_k - r(x_{k+1}) 
			> &  \Phi_k(0) -\Phi_k(\beta_k d_k).
         \end{aligned}
       \end{equation}
	Using the upper-$C^2$ property of $r(\cdot)$, as in~\eqref{eqn:opt-ms-appx-rec-merit-beta} 
	from Lemma~\ref{lem:sufficientdecrease-merit}, we have  
\begin{equation} \label{eqn:simp-bundle-merit-penal-pf-1}
 \centering
  \begin{aligned}
	  r(x_k) - r(x_{k+1}) >& -\beta_k g_{k}^Td_k -\frac{1}{2}\alpha_k\beta_k^2 \norm{d_k}^2. \\
  \end{aligned}
\end{equation}
	Let us rearrange the first equation in the KKT conditions~\eqref{eqn:simp-bundle-penal-KKT-1}  and obtain 
\begin{equation} \label{eqn:simp-bundle-penal-merit-pf-1}
  \centering
   \begin{aligned}
	   \pi_k \left[g_k + \alpha_k d_k \right] = -\sum_{j=1}^m\lambda_j^{k+1} \nabla c_j(x_k) 
	     + \zeta_l^{k+1}-\zeta_u^{k+1}.\\
   \end{aligned}
 \end{equation}
	Taking the dot product with $-d_k$ on both sides of~\eqref{eqn:simp-bundle-penal-merit-pf-1} leads to
\begin{equation} \label{eqn:simp-bundle-penal-merit-pf-2}
  \centering
   \begin{aligned}
	   - \pi_k \left[g_k^Td_k + \alpha_k\norm{d_k}^2\right] =& \sum_{j=1}^m\lambda^{k+1}_j \nabla c_j(x_k)^T d_k 
	    -d_k^T \zeta_l^{k+1} + d_k^T \zeta_u^{k+1}.\\
   \end{aligned}
 \end{equation}
	Recall that $d_l^k=-x_k$ and $d_u^k=x_u-x_k$.
	Applying the complementarity conditions, which are the fourth, fifth equation in~\eqref{eqn:simp-bundle-penal-KKT-1},~\eqref{eqn:simp-bundle-penal-merit-pf-2} is simplified to
\begin{equation} \label{eqn:simp-bundle-penal-merit-pf-2.5}
  \centering
   \begin{aligned}
	   -\pi_k  \left[g_k^Td_k + \alpha_k\norm{d_k}^2\right] =& \sum_{j=1}^m \lambda^{k+1}_j \nabla c_j(x_k)^T d_k 
	   -(d_k-d_l^k+d_l^k)^T\zeta_l^{k+1}+d_k^T\zeta_u^{k+1}\\
            =& \sum_{j=1}^m \lambda^{k+1}_j \nabla c_j(x_k)^T d_k 
	   +x_k^T\zeta_l^{k+1}+(d_k-d_u^k+d_u^k)^T\zeta_u^{k+1}\\
	   =& \sum_{j=1}^m \lambda_j^{k+1} \nabla c_j(x_k)^T d_k 
	    +x_k^T \zeta_l^{k+1}
	   +(x_u-x_k)^T \zeta_u^{k+1}\\
	   \geq& \sum_{j=1}^m \lambda_j^{k+1} \nabla c_j(x_k)^T d_k. 
   \end{aligned}
 \end{equation}
	The last inequality utilizes the bound constraints from the previous iteration $0\leq x_k\leq x_u$, and $\zeta_u^{k+1},\zeta_l^{k+1}\geq 0$.
	Multiplying by $\beta_k$ and 
	subtracting $\pi_k \frac{1}{2}\alpha_k\beta_k^2\norm{d_k}^2$ from both sides of~\eqref{eqn:simp-bundle-penal-merit-pf-2.5} 
	and using $\beta_k\in (0,1]$ results in
\begin{equation} \label{eqn:simp-bundle-penal-merit-pf-3}
  \centering
   \begin{aligned}
	   \pi_k\left[-\beta_kg_k^Td_k -\frac{1}{2}\alpha_k\beta_k^2\norm{d_k}^2 \right]
	   \geq& \pi_k \beta_k \alpha_k\norm{d_k}^2 -\pi_k \frac{1}{2}\alpha_k\beta_k^2\norm{d_k}^2 +\beta_k\sum_{j=1}^m \lambda_j^{k+1} \nabla c_j(x_k)^T d_k\\ 
	   \geq& \pi_k \frac{1}{2}\alpha_k\beta_k\norm{d_k}^2 +\beta_k\sum_{j=1}^m \lambda_j^{k+1} \nabla c_j(x_k)^T d_k\\ 
	   =& \pi_k \frac{1}{2}\alpha_k\beta_k\norm{d_k}^2 +\beta_k  (\lambda^{k+1})^T \nabla c(x_k)^T d_k. 
   \end{aligned}
 \end{equation}
	From~\eqref{eqn:simp-bundle-merit-penal-pf-1}, the merit function satisfies
\begin{equation} \label{eqn:simp-bundle-penal-merit-pf-4}
 \centering
  \begin{aligned}
	  \phi_{1\pi_k}(x_k) - \phi_{1\pi_k}(x_{k+1})  =& r(x_k) -r(x_{k+1})+ \frac{1}{\pi_k}\norm{c(x_k)}_1
		    -\frac{1}{\pi_k}\norm{c(x_{k+1})}_1\\
			 \geq& -\beta_kg_k^T d_k-\frac{1}{2}\alpha_k\beta_k^2\norm{d_k}^2+ \frac{1}{\pi_k}\norm{c(x_k)}_1
		    -\frac{1}{\pi_k}\norm{c(x_{k+1})}_1.\\
  \end{aligned}
\end{equation}
Applying~\eqref{eqn:simp-bundle-penal-merit-pf-3} and the line search condition~\eqref{eqn:line-search-cond-penal}, we have
\begin{equation} \label{eqn:simp-bundle-penal-merit-pf-5}
 \centering
  \begin{aligned}
	  \phi_{1\pi_k}(x_k) - \phi_{1\pi_k}(x_{k+1})  
			 \geq&\frac{1}{2}\alpha_k\beta_k\norm{d_k}^2 +\frac{\beta_k}{\pi_k} (\lambda^{k+1})^T \nabla c(x_k)^T d_k\\
			 &+ \frac{1}{\pi_k}\left(\norm{c(x_k)}_1 - \norm{c(x_{k+1})}_1\right)\\
			 \geq&\frac{1}{2}\alpha_k\beta_k\norm{d_k}^2-\eta_{\beta}\frac{1}{2}\alpha_k\beta_k\norm{d_k}^2\\
			 =&(1-\eta_{\beta})\frac{1}{2}\alpha_k\beta_k\norm{d_k}^2.
  \end{aligned}
\end{equation}
Given $\pi_k>0$, the conclusion follows. 
In addition, if $\pi_k$ stabilizes, based on~\eqref{eqn:simp-bundle-fea-c-ls-pf-5}
\begin{equation} \label{eqn:simp-bundle-penal-merit-pf-6}
 \centering
  \begin{aligned}
	  \beta_k \geq \frac{1}{2} ^{\lceil \log_{\frac{1}{2}} \frac{\eta_{\beta} \alpha_0 \bar{\pi}}{2m H^c_u} \rceil}:=\bar{\beta}^{\pi}.
  \end{aligned}
\end{equation}
Thus, from~\eqref{eqn:simp-bundle-penal-merit-pf-5}, $\phi_{1\pi_k}(x_k)-\phi_{1\pi_k}(x_{k+1})> (1-\eta_{\beta})\frac{1}{2}\alpha_0 \bar{\beta}^{\pi}\norm{d_k}^2$. Or equivalently, there exists $c_{\phi}^{\pi}>0$ such that $\phi_{1\pi_k}(x_k)-\phi_{1\pi_k}(x_{k+1})>c_{\phi}^{\pi}\norm{d_k}^2$.
\end{proof}
In general, the global convergence analysis from Section~\ref{sec:alg-convg} stands when $\pi_k$ is bounded away from $0$ and $\theta_k$ is bounded from above. This is reflected in the following two theorems similar to Theorem~\ref{thm:simp-KKT}.

%
\begin{theorem}\label{thm:simp-KKT-consistency-finite}
	Under the Assumptions~\ref{assp:upperC2},~\ref{assp:boundedHc} and LICQ conditions of Lemma~\ref{lem:bounded-lp}, if Algorithm~\ref{alg:simp-bundle-const} is called finite many times, then every accumulation points of the serious step sequence $\{x_k\}$ generated from Algorithm~\ref{alg:simp-bundle}
	and~\ref{alg:simp-bundle-const}
	is a KKT point of problem~\eqref{eqn:opt-ms-simp}. 
\end{theorem}
The proof is similar to that of Theorem~\ref{thm:simp-KKT} and straightforward.
Since there are only finite number of consistency restoration steps, the linearized constraints become consistent for $k$ large enough. Thus, only Algorithm~\ref{alg:simp-bundle} is called for $k$ large enough. We can directly apply Theorem~\ref{thm:simp-KKT} to obtain~\ref{thm:simp-KKT-consistency-finite}.

Before stating the next theorem, we note that $\pi_k$ and $\theta_k$ are designed to impact each other through step 7 of Algorithm~\ref{alg:simp-bundle} and step 1 in Algorithm~\ref{alg:simp-bundle-const}. 
Therefore, if $\frac{1}{\pi_k}$ does not stay bounded, $\theta_k$ will not either. On the other hand, if $\pi_k$ is bounded below from a nonzero value,  together with the conditions in Lemma~\ref{lem:bounded-lp}, both $\frac{1}{\pi_k}$ and $\theta_k$ are finite for all $k$. In addition, for $k$ large enough, both stop increasing. A stabilized $\pi_k$ at nonzero values essentially requires step 7 
in Algorithm~\ref{alg:simp-bundle-const} to be encountered only finitely many times. 
\begin{theorem}\label{thm:simp-KKT-consistency-bounded}
	Under the Assumptions~\ref{assp:upperC2},~\ref{assp:boundedHc} and LICQ conditions of Lemma~\ref{lem:bounded-lp}, if Algorithm~\ref{alg:simp-bundle-const} is called infinitely many times with nonzero stabilized penalty parameter, \textit{i.e.}, $\pi_t=\pi_k > 0$ for all $t\geq k$,
	then any accumulation points of the sequence $\{x_k\}$ generated by Algorithm~\ref{alg:simp-bundle} 
	and~\ref{alg:simp-bundle-const}
	is either a KKT point of~\eqref{eqn:opt-ms-simp} or a critical point 
	of the linearized constraint violation of $c(\cdot)$. 
\end{theorem}
The proof is again similar to that of Theorem~\ref{thm:simp-KKT} and a brief framework is presented here.
First, by Lemma~\ref{lem:sufficientdecrease-penal}, let $k$ be large enough such that $\alpha_t = \alpha_k$ for all $t\geq k$ and all steps produced by both algorithms are serious steps.
We note that by design both penalty parameters $\theta_k$ and $\frac{1}{\pi_k}$ in merit function~\eqref{eqn:opt-ms-simp-bundle-merit} and~\eqref{def:merit-fea} increase monotonically across iterations and algorithms. By Lemma~\ref{lem:bounded-lp}, $\lambda^k$ is bounded for $k$ large enough. Given a nonzero and stabilized $\pi_k$ for $k$ large enough, $\theta_k$ is also bounded and remains constant based on step 7 in Algorithm~\ref{alg:simp-bundle}.
Together with step 7 in Algorithm~\ref{alg:simp-bundle-const}, the merit functions $\phi_{1\theta_k}(\cdot)$ and  $\phi_{1\pi_k}(\cdot)$ for both algorithms~\eqref{eqn:opt-ms-simp-bundle-merit} and~\eqref{def:merit-fea} have the same parameter $\theta_k=\frac{1}{\pi_k}=\bar{\theta}$, where $\bar{\theta}$ denotes the stabilized value.
Hence, by Lemma~\ref{lem:sufficientdecrease-merit} and~\ref{lem:sufficientdecrease-merit-penal}, $\{\phi_{1\pi_k}(x_k)\}$ and $\{\phi_{1\theta_k}(x_k)\}$ at serious steps $x_k$ decreases monotonically in the order of $\norm{d_k}^2$ and is bounded below. Therefore, $d_k$ generated by both algorithms, regardless of the order they are called, satisfies $d_k \to 0$.

Compared to the case in Theorem~\ref{thm:simp-KKT-consistency-finite}, it is possible that a finite number of calls of Algorithm~\ref{alg:simp-bundle} are followed by 
all consistency restoration calls of Algorithm~\ref{alg:simp-bundle-const}, in which case an accumulation point of $\{x_k\}$ might be infeasible. This can be seen from the constraints of the penalty problem in~\eqref{eqn:opt-ms-simp-bundle-penal} where $v,w$ could be nonzero at an accumulation point $\bar{x}$ of $\{x_k\}$ even as $d_k\to 0$. 

If an accumulation point $\bar{x}$ is feasible, \textit{i.e.}, $c(\bar{x}) = 0$, then $\{x_k\}$ converges subsequently 
to a KKT point of problem~\eqref{eqn:opt-ms-simp}, the proof of which is the same as in Theorem~\ref{thm:simp-KKT} at this point. 
Otherwise, if $\bar{x}$ is infeasible, \textit{i.e.}, $c(\bar{x})\neq 0$, it is not a KKT point. 
Meanwhile, from~\eqref{def:pd-penal}, given still $d_k\to 0$, we have $\delta_k^{\pi_k}\to 0$. 
From step 7 in Algorithm~\ref{alg:simp-bundle-const}, the update rule of $\pi_k$ enforces $\eta_f \delta_k^f\leq \delta_k^{\pi_k}$. Therefore, with a constant parameter $\eta_f$, we obtain $\delta_k^f\to 0$.
Hence, for an infeasible accumulation point $\bar{x}$, the update rule for the penalty parameter $\pi_k$
results in $\bar{x}$ being a critical point of linearized constraint violation of $c(\cdot)$. 
This convergence result is similar to the exact penalty method for smooth objectives~\cite{Nocedal_book}.

Finally, the following theorem covers the case when the penalty parameter $\pi_k \to 0$.
\begin{theorem}\label{thm:simp-KKT-consistency-unbounded}
	Under the Assumptions~\ref{assp:upperC2},~\ref{assp:boundedHc} and LICQ conditions of Lemma~\ref{lem:bounded-lp}, if Algorithm~\ref{alg:simp-bundle-const} is called
	infinite many times and $\pi_k \to 0$,
	then any accumulation points of the serious steps $\{x_k\}$ generated from Algorithm~\ref{alg:simp-bundle}
	and~\ref{alg:simp-bundle-const}
	is a critical point of the linearized constraint violation $c(\cdot)$. 
\end{theorem}
\begin{proof}
	From Lemma~\ref{lem:bounded-lp} , we know that the Lagrange multipliers from Algorithm~\ref{alg:simp-bundle} are bounded. Therefore, a $\pi_k\to 0$ is caused by step 7 in Algorithm~\ref{alg:simp-bundle-const} being called infinitely many times which in turn increases $\theta_k$ as well. 
	Because we are considering an infinite number of iterations where $\delta_k^f>0$ (otherwise the algorithm would have terminated at step 4), by Lemma~\ref{lem:consistency-penal-defined}, the number of loops 
	between step 6 and 7 is finite for each $k$. Thus, to have infinite many step 7, we must have an infinite number 
	of iterations that would enter step 7 at least once.
	Let $k$ be one of the iterations where $\pi_k$ is reduced through step 7. 
	To simplify the analysis, we denote by $\pi_k^0 = \min{(\pi_{k-1},\frac{1}{\theta_k})}$ the penalty parameter at iteration $k$ after step 2.
	Then the change in objective of the penalized problem~\eqref{eqn:opt-ms-simp-bundle-penal} with $\pi_{k}^0$,
	before the update to $\pi_k$ at step 7, is 
	\begin{equation}
          \centering
	  \begin{aligned}
		  \delta_k^{0} = \pi_{k}^0 \left(-g_k^T d_k^{0} -\frac{1}{2}\alpha_k \norm{d_k^{0}}^2\right) +\norm{c(x_k)}_1-\norm{c(x_k) +\nabla c(x_k)^T d_k^{0}}_1, 
          \end{aligned}
	\end{equation}
	where $d_k^{0}$ is the solution of~\eqref{eqn:opt-ms-simp-bundle-penal} with $\pi_{k}^0$.
		Since $\pi_{k}^0 $ enters step 7 in Algorithm~\ref{alg:simp-bundle-const}, $\delta_k^{0} < \eta_f\delta_k^f$. 
	Given $d_k^{0}$ as the solution to~\eqref{eqn:opt-ms-simp-bundle-penal} with $\pi_{k}^0$ and using the definition of $\delta_k^f$ in~\eqref{def:pd-fea}, 
   \begin{equation} \label{eqn:simp-bundle-fea-pi-0-1}
   \centering
    \begin{aligned}
	    \eta_f\delta_k^f > \delta_k^{0} =& \pi_{k}^0\left(-g_k^T d_k^{0} -\frac{1}{2}\alpha_k \norm{d_k^{0}}^2\right) +\norm{c(x_k)}_1- \norm{c(x_k)+ \nabla c(x_k)^T d_k^{0}}_1\\
	      \geq& \pi_{k}^0\left(-g_k^T d_k^f -\frac{1}{2}\alpha_k \norm{d_k^f}^2\right) +\norm{c(x_k)}_1- \norm{c(x_k)+ \nabla c(x_k)^T d_k^f}_1\\
		     =& \pi_{k}^0\left(-g_k^T d_k^f  - \frac{1}{2} \alpha_k \norm{d_k^f}^2\right) + \delta_k^f \\
		     \geq& \pi_{k}^0\left(-\norm{g_k}\norm{d_k^f} - \frac{1}{2}\alpha_k\norm{d_k^f}^2\right)+ \delta_k^f.\\
    \end{aligned}
   \end{equation}
	Since $x_k$, $g_k$ and $\alpha_k$ are all bounded, assigning $D = \norm{d_u-d_l}=\norm{x_u}$, we have 
   \begin{equation} \label{eqn:simp-bundle-fea-pi-0-2}
   \centering
    \begin{aligned}
	    (1-\eta_f) \delta_k^f \leq& \pi_{k}^0 \left(\norm{g_k}\norm{d_k^f} + \frac{1}{2}\alpha_k\norm{d_k^f}^2\right)\\
			       =&  \pi_{k}^0\left( \norm{g_k} D + \frac{1}{2}\alpha_k D^2 \right).\\
    \end{aligned}
   \end{equation}
	Therefore, as $\pi_k$ and $\pi_k^0$ approach 0, so does $\delta_k^f$.
	This proves that as $x_k\to \bar{x}$, $\delta_k^f\to 0$. Thus, $\bar{x}$ is a critical point of the linearized constraint violation of $c(\cdot)$.
\end{proof}

Theorem~\ref{thm:simp-KKT-consistency-unbounded} is a relatively weak result in the sense that it does not distinguish between an accumulation point $\bar{x}$ that 
is feasible, \textit{i.e.}, $c(\bar{x})=0$  and infeasible. Stronger results are possible for smooth optimization even under a  
less restrictive constraint qualification, Mangasarian–Fromovitz constraint qualification. For example, Byrd et. al.~\cite{byrd2005} show that if $\pi_k \to 0$ and $c(\bar{x}) = 0$, then their SLQP algorithm converges 
to a KKT point. Such result for (nonsmooth) upper-$C^2$ objective function is not evident to us and will be the subject of our future research.

%% file: Sections/Examples.tex
\section{Numerical Applications}\label{sec:exp}
We present three numerical examples to demonstrate the theoretical and numerical 
capabilities the proposed algorithm offers. The examples are chosen within 
the general formulation of two-stage optimization problems.
For nonsmooth nonconvex optimization problems, a wildly popular assumption of the 
objective function is lower-$C^2$ or prox-regular and the constraints are often assumed to be convex.
The first two problems are synthetic problems designed to showcase the extra theoretical
convergence analysis our algorithm brings to problems that do not satisfy these lower-type of properties.
They are simple and computationally inexpensive. We also present a comparison of results with the 
redistributed bundle method in~\cite{hare2010}, which we drew inspiration from.

\begin{example}\label{exmp:1}
	(Differentiable but not continuously differentiable objective) Example 1 has the following mathematical formulation:
\begin{equation}\label{prob:ex1-1st}
 \begin{aligned}
  &\underset{x \in \Rbb^3}{\text{min}} 
	  &    f(x_1) + &\mu [(x_{2}-\frac{1}{2})^2+ x_3^2] +  r(x)\\
   &\text{s.t.}
	  &  -5 \leq &x_1 \leq 5, \ 0 \leq x_2 \leq 50,\ -1 \leq x_3 \leq 10,\\
 \end{aligned}
\end{equation}
where $\mu=10^5$ and $f(x_1):\Rbb\to\Rbb$ is a continuously differentiable function.
	The function $r(\cdot)$ is the solution to the second-stage problem
\begin{equation}\label{prob:ex1-2nd}
 \begin{aligned}
  &\underset{y \in \Rbb^3}{\text{min}} 
	  &&  \norm{x-y}^2 \\
   &\text{s.t.}
	  &&  y_2 \leq y_3^2, \  -5 \leq y_1 \leq 5,\\
	  &&& -5 \leq y_2\leq 5, \ 0 \leq y_3 \leq 10.\\
 \end{aligned}
\end{equation}
\end{example}
It is obvious that $r(\cdot)$ is a squared-distance function and thus upper-$C^2$. 
In addition, $r(\cdot)$ turns out to also be lower-$C^1$, but not lower-$C^2$ at $\tilde{x}=[x_1,\frac{1}{2},0]$, where $x_1$ 
can be any value within its bounds. This translates to $r(\cdot)$ being differentiable but not continuously differentiable 
at $\tilde{x}$, as shown in the feasible region plot on the left of Figure~\ref{fig:ex_fs}.
In this case, our proposed algorithm offers global convergence support compared to algorithms
that require a lower-$C^2$ objective. It needs to be pointed out that $r(\cdot)$ is  
continuously differentiable at remaining points in the domain and thus other algorithms with carefully 
chosen parameters can succeed in solving Example 1 regardless.
\begin{figure}
  \centering
  \includegraphics[width=0.95\textwidth]{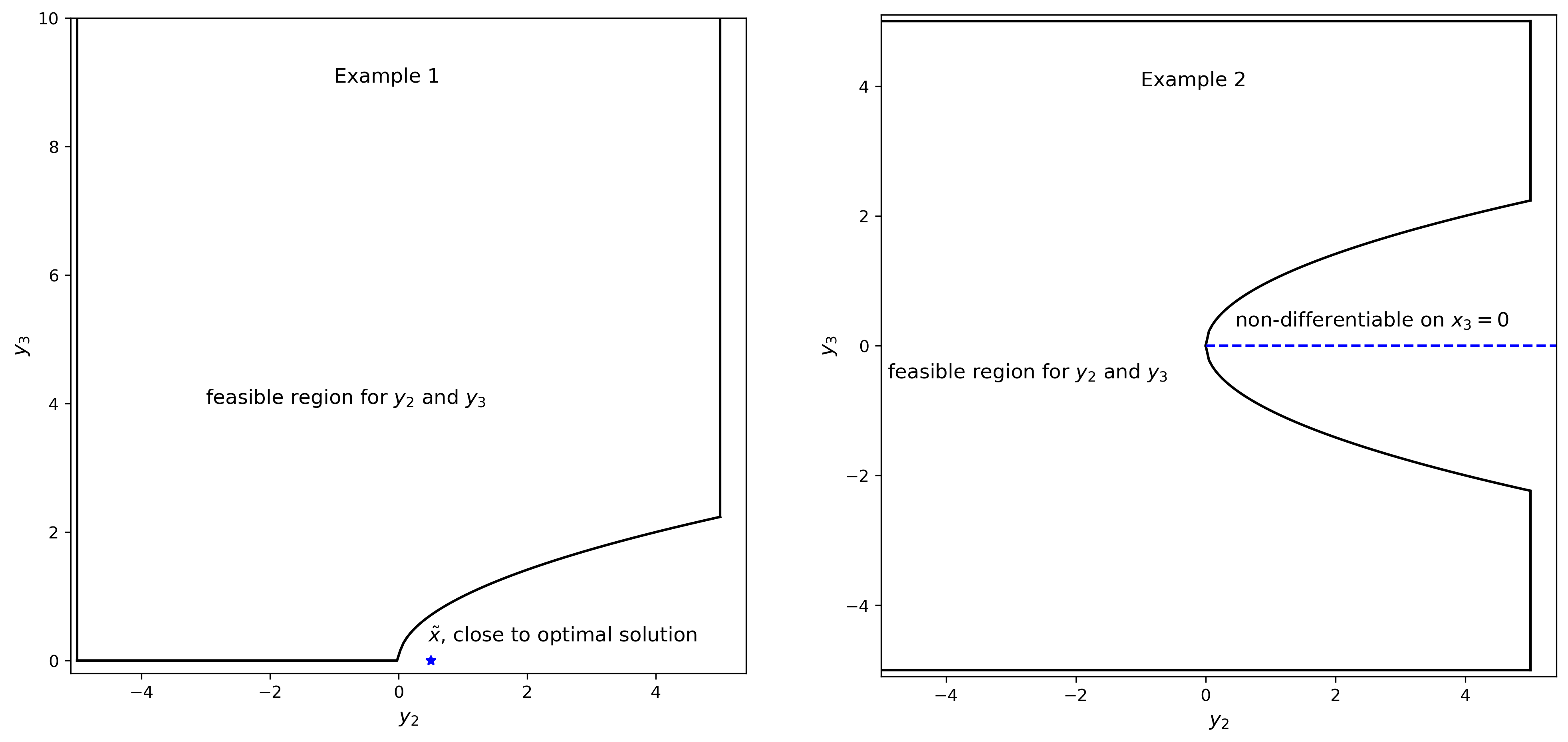}
	\caption{Feasible set of $y_2,y_3$ plane of example 1 (left) and example 2 (right)}
\label{fig:ex_fs}
\end{figure}

The true solution is obtained by treating the two-stage problem as one problem with variables in $\Rbb^6$ and solved with Ipopt. The proposed Algorithm~\ref{alg:simp-bundle} starts with $\alpha_0=1.0,\epsilon=10^{-8}$ and the redistributed bundle method in~\cite{hare2010} is implemented with $\Gamma=2,\mu_0=1,\eta_0=1$.
The initial point is set to $x_0=[1,50,5]^T$. 
The simplified bundle method exits in 4 iterations.
While both algorithms quickly moved close to the solution, due to the lack of lower-$C^2$ property at $\tilde{x}=[x_1,\frac{1}{2},0]$ ($\forall x_1\in[-5,5]$), the convexification parameter $\eta_n$ registers a large value for the redistributed bundle method. 
Given the error tolerance at $10^{-8}$, the redistributed bundle method will require more iterations and could potentially be destabilized due to numerical error from the large value of $\eta_n$. This problem disappears if error tolerance is set to a larger number. Figure~\ref{fig:ex1} shows the numerical result of error measure against the number of iterations for both redistributed and simplified bundle method. 
The quadratic coefficient, which decreases in the simplified bundle method as explained in Remark~\ref{rmrk:realalpha}, is also plotted for both algorithms.
\begin{figure}
  \centering
  \includegraphics[width=0.95\textwidth]{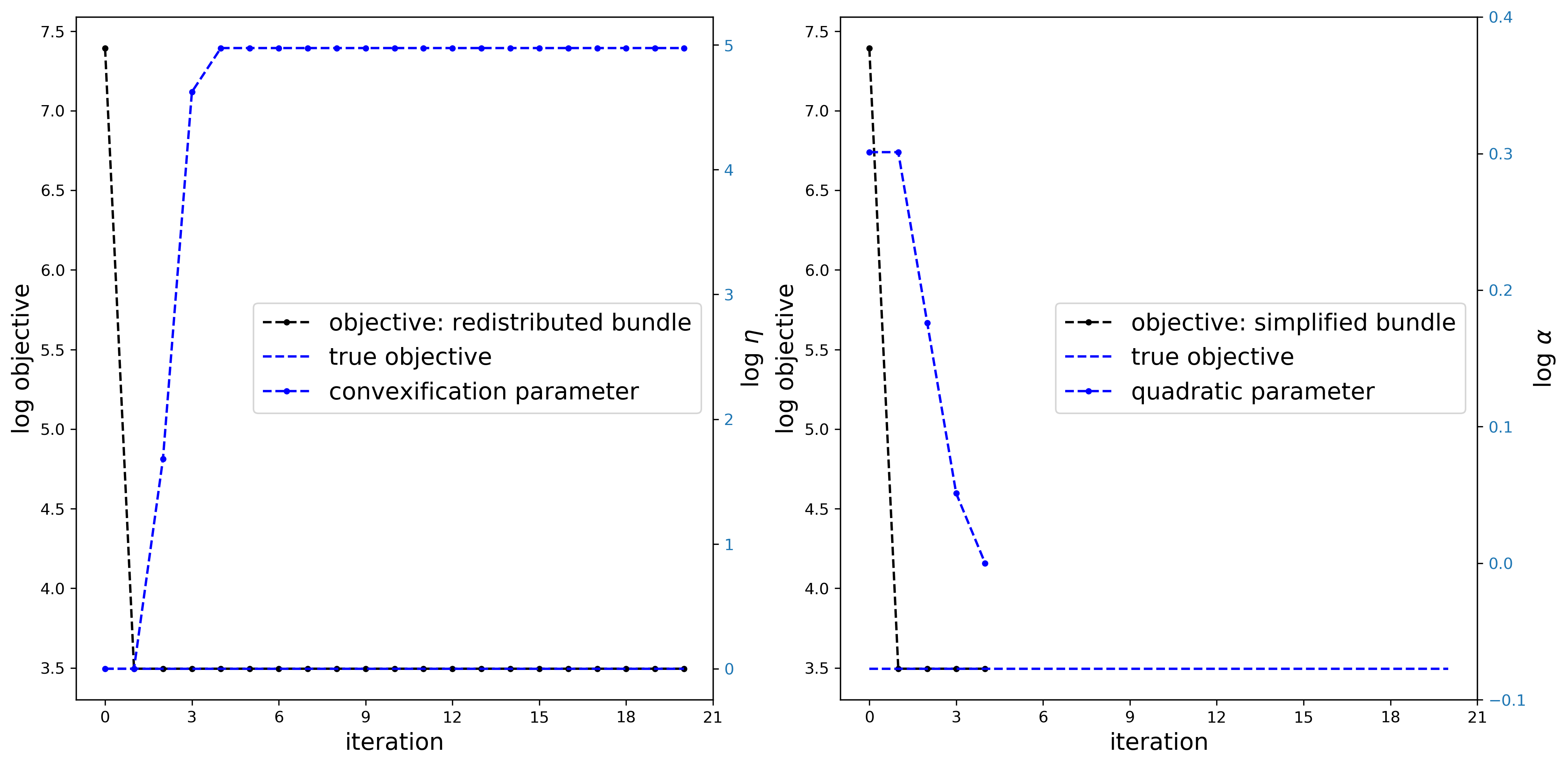}
 \caption{Convergence and quadratic coefficient plots for example 1}
\label{fig:ex1}
\end{figure}

\begin{example}\label{exmp:2}
  (Non-differentiable) Example 2 has the following form:
\begin{equation}\label{prob:ex2-1st}
 \begin{aligned}
  &\underset{x \in \Rbb^3}{\text{min}} 
	  &    f(x_1) + &\mu [(x_{2}-\frac{1}{2})^2+ x_3^2] +  r(x)\\
   &\text{s.t.}
	  &  -5 \leq &x_1 \leq 5, \ 0 \leq x_2 \leq 50,\\
	 &&  -5 \leq &x_3 \leq 5.\\
 \end{aligned}
\end{equation}
Again, $\mu=10^5$ and $f(x_1)$ is a continuously differentiable function.
	The function $r(\cdot)$ is the solution to the second-stage problem
\begin{equation}\label{prob:ex2-2nd}
 \begin{aligned}
  &\underset{y \in \Rbb^3}{\text{min}} 
	  & & \norm{x-y}^2 \\
   &\text{s.t.}
	  &&  y_2 \leq y_3^2, \ -5 \leq y_1,y_2,y_3 \leq 5.\\
 \end{aligned}
\end{equation}
\end{example}
Example 2 is designed to only vary slightly from Example 1 to illustrate the large group of problems 
the proposed algorithm can tackle. 
With a slight change in the constraint to allow $y_3<0$, the solution function $r(\cdot)$ is no longer differentiable on $x_3=0$ as 
multiple solutions $y$ exist. This is illustrated on the right plot in Figure~\ref{fig:ex_fs}.

However, $r(\cdot)$ remains upper-$C^2$ and the convergence analysis for the proposed algorithm applies.
Figure~\ref{fig:ex2} shows the objective and quadratic coefficient for both redistributed and simplified bundle method from the same starting point as in Example 1. Similar conclusions as in Example 1 can be drawn.
\begin{figure}
  \centering
  \includegraphics[width=0.95\textwidth]{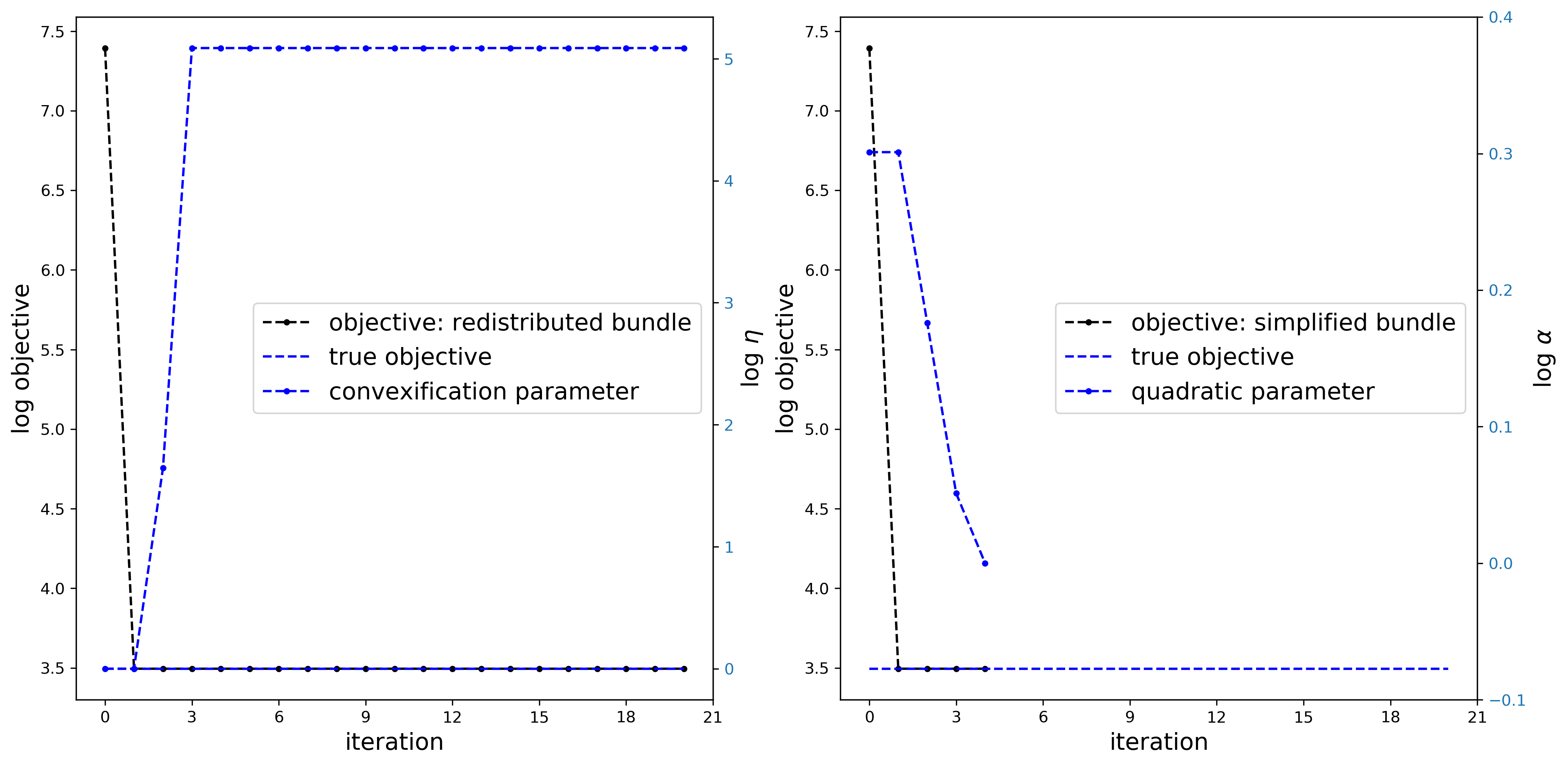}
 \caption{Convergence and quadratic coefficient plots for example 2}
\label{fig:ex2}
\end{figure}

\begin{example}\label{exmp:3}
	(smoothed SCACOPF) Example 3 is a SCACOPF problem with affine active power constraint for 
	contingency (second-stage) problems. The network data used in this example is from the ARPA-E Grid Optimization competition~\cite{petra_21_gollnlp}. 
	The complete mathematical formulation is complex but the master (first-stage) problem fits in the form of~\eqref{eqn:opt0}, 
	where $r$ is the recourse function of the contingency problems. 
	Details of the 
	problem setup can be seen in~\cite{petra_21_gollnlp}. The number of contingency problems that are solved to evaluate $r$ is 100.
\end{example}
The coupling constraint between master and contingency variables can be viewed as linear in the former ($x$) but it is nonsmooth.
This means recourse function $r$ might not be upper-$C^2$.
However, using a quadratic penalty of the coupling constraints in the contingency problems, $r$ in~\eqref{eqn:opt0} becomes upper-$C^2$ and the problem is referred to as the smoothed SCACOPF, in contrast to the original non-smoothed one. 
The proposed algorithm is applied to the smoothed SCACOPF, where the quadratic penalty parameter $\mu$ is set to $10^9$.
While this means the convergence analysis applies, we only solve an approximated problem.
To verify the accuracy of the solution, the true solution is obtained by solving the extensive form of the SCACOPF with Ipopt. 
It is plotted with the optimization results in Figure~\ref{fig:ex3}. We also plot the non-smoothed objective evaluated at the optimal solution $x$ gained from the smoothed problem at each iteration. 
The rejected steps are marked as well.
Within 200 iterations, the non-smoothed objective reach within $0.010\%$ error of the true solution, which
is acceptable and useful in practice. To speed up convergence of this first-order method, 
the quadratic coefficient $\alpha_k$ is reduced whenever possible. 
For large-scale problems with 
$10^4$ coupled optimization variables and $10^5$ contingencies, the extensive form of the SCACOPF
would be impractical, while the simplified bundle algorithm has been successfully deployed to supercomputers~\cite{wang2021}.

\begin{figure}
  \centering
  \includegraphics[width=0.95\textwidth]{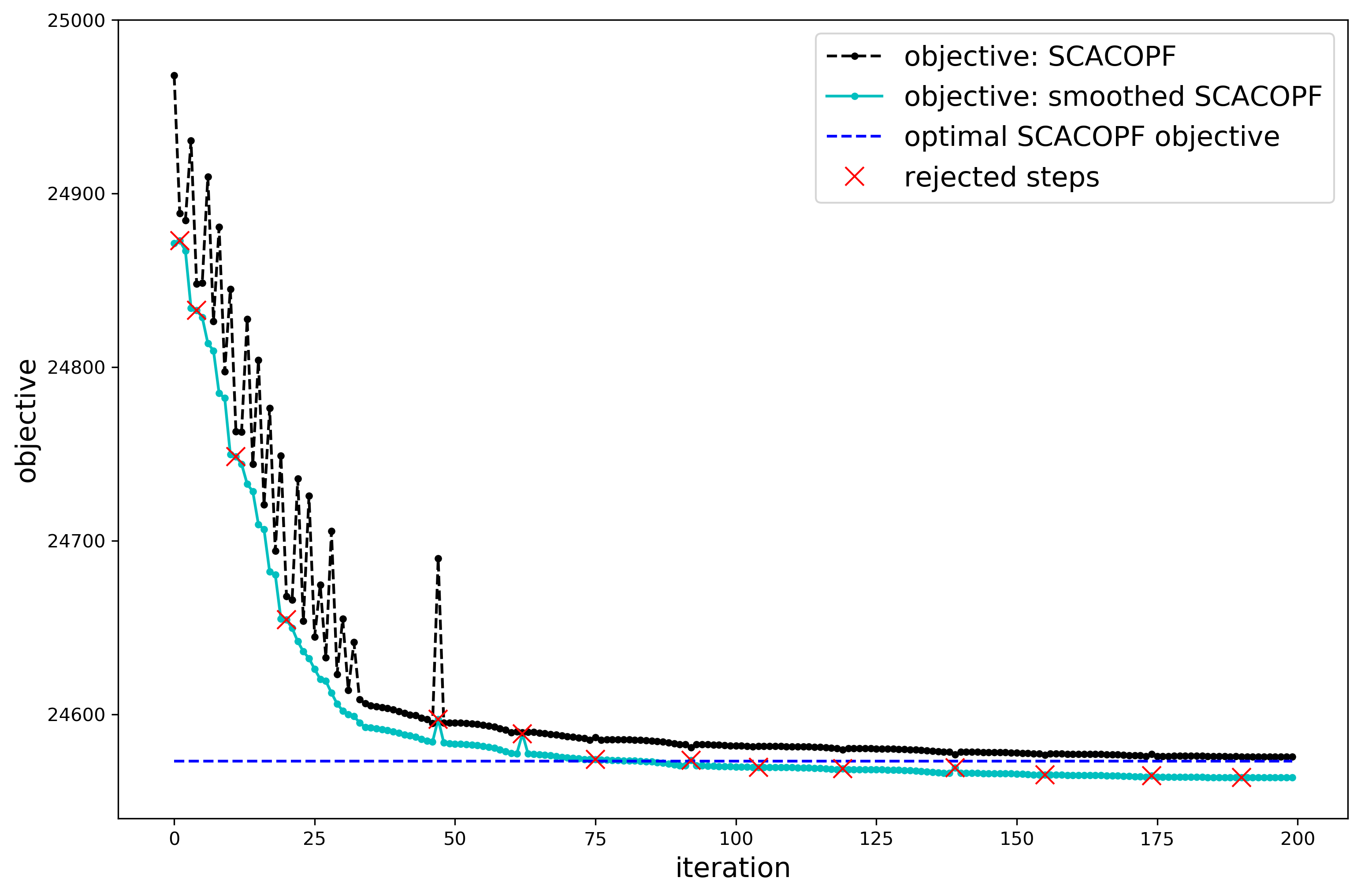}
 \caption{Convergence plots for example 3}
\label{fig:ex3}
\end{figure}

%% file: Sections/Conclusion.tex
\section{\normalsize Conclusions}\label{sec:con}
In this report, we have motivated, proposed and analyzed algorithms for a group of nonsmooth, nonconvex optimization problems. 
We show that many two-stage (stochastic) optimization problems, including our target application SCACOPF problems exhibit interesting properties which are not thoroughly investigated previously. 
This has lead to our design and analysis of the simplified bundle algorithm whose global convergence can be achieved under upper-$C^2$ objectives. The algorithm is scalable and has been implemented on parallel computing platforms. Numerical experiments show promising convergence and scaling results.

%% file: Sections/Appendix.tex
\setcounter{section}{1}
\renewcommand{\thesection}{\Alph{section}}
\addcontentsline{toc}{section}{Appendix A: Slicing algorithm} 
\section*{Appendix A: Continuity of the second-stage solution function}\label{sec:cont}
\par\noindent
This appendix details the continuity of $r_{\mu}(\cdot)$ in~\eqref{eqn:opt-rc-mu-al}. 
We consider continuity of $r_{\mu}(\cdot)$ using Proposition 4.4 from~\cite{shapiro2000}. The related notations
are denoted as follows.
The set $S(x)\subset Y$ is the optimal solutions at $x$.
We denote by $\Phi(x) \subset \Rbb^m$ the feasible set of $y$ in the recourse subproblem.
From the constraints in~\eqref{eqn:opt-rc-mu-al}, an important observation is that $\Phi(x)=\Phi$, for all x, \textit{i.e.}, the 
feasible set for $y$ is independent of $x$ due to smoothing.
 To simplify 
the notations, instead of applying compact set theories on the extended real vector space, it is reasonable to assume the following.
\begin{assumption}\label{assp:compact-set}
 The optimization variables $x$ and $y$ are bounded. The feasibility sets for $x$ and $y$, denoted as $X\in \Rbb^n$ and $Y\in\Rbb^m$, respectively, are bounded. Moreover, $Y$ is compact.
\end{assumption}
We establish in Lemma~\ref{lem:continuity} that under the given assumption, 
the optimization problem in~\eqref{eqn:opt-rc-mu-al} meets the conditions in Proposition 4.4 in~\cite{shapiro2000}, which 
directly establishes the continuity of $r_{\mu}(\cdot)$.

\begin{lemma}\label{lem:continuity}
The optimization problem~\eqref{eqn:opt-rc-mu-al} satisfies the conditions in Proposition 4.4 in~\cite{shapiro2000} at a given x,
which are (1) the function $f(x,y)$ is continuous
on $X \times Y$, (2) the multifunction $\Phi(\cdot)$ is closed, (3) there exists $\alpha\in \Rbb$ and a compact
set $C\subset Y$ such that for every $x'$ in a neighborhood of $x$, the level set
\begin{equation} \label{eqn:inf-compactness}
 \centering
  \begin{aligned}
	  \text{lev}_{\alpha}f(x',\cdot):=\{y \in \Phi(x): f(x',y)\leq \alpha \}
  \end{aligned}
\end{equation}
is nonempty and contained in C, (4)
for any neighborhood $V_y$ of the set $S(\bar{x})$,$\bar{x}\in X$, there exists a neighborhood $V_x$ of $\bar{x}$ such that
$V_y\cap \Phi(x)\neq \emptyset$ for all $x \in V_x$.
\end{lemma}
\begin{proof}
The objectives and constraints in the recourse subproblem are twice continuously differentiable, as mentioned 
	when introducing~\eqref{eqn:2ndstage},
	which guarantees (1) for the entire feasible set $\Phi(x)$. Assumption~\ref{assp:compact-set}, consistent with 
the constraints and bounds on $y$, ensures a closed feasible set and thus (2) is met. 
Since $f(x',y)$ is continuously differentiable on $X \times Y$, it is also bounded.
Denoting a neighborhood of $x$ as $V_x$, 
the obvious choice of $\alpha$ to make the level set $\text{lev}_{\alpha}f(x',\cdot),  \forall x' \in V_x$ nonempty is to let it be the maximum value of $f(x',\cdot)$ on $C$. Therefore, an $\alpha$ exists such that $f(x',y)\leq \alpha, \forall x' \in V_x,y\in \Phi(x')$.
	Given the compact set $Y\subset \Rbb^m$, 
a compact subset $C\subset Y$ can be found such that the level set $\text{lev}_{\alpha}f(x',\cdot)$ is contained in $C$. 
Thus, (3) is satisfied.
To see (4),
	let $\bar{y} \in S(\bar{x})$ and $V_{y}$ be a neighborhood of $\bar{y}$. Since $\Phi(x)$ is independent of $x$
and compact, it is clear $V_y\cap \Phi(x)\neq \emptyset$.
\end{proof}
We can then prove Lemma~\ref{lemma:continuity}.

\begin{proof}
Applying Proposition 4.4 in~\cite{shapiro2000} and Lemma~\ref{lem:continuity} directly,  $r_{\mu}(\cdot)$ is continuous for any $x\in X$ and the multifunction $x\to S(x)$ is upper semicontinuous at $x$.
\end{proof}